\documentclass[a4paper, 10pt]{article}
\usepackage{srcltx,graphicx,epstopdf}
\usepackage{amsmath, amssymb, amsthm}
\usepackage{color}
\usepackage{xcolor}
\usepackage{multirow}
\usepackage{subfig} % it is suggested to use subfloat rather than subfigure
\usepackage[margin=1in]{geometry}
\usepackage[T1]{fontenc}
\usepackage{cleveref}
\usepackage{cite}
\crefname{section}{Section}{Sections}
\crefname{subsection}{Subsection}{Subsections}
\crefname{appendix}{Appendix}{Appendix}
\crefname{figure}{Figure}{Figures}
\crefname{table}{Table}{Tables}
\crefname{property}{Property}{Properties}
\crefname{theorem}{Theorem}{Theorem}
\usepackage{booktabs}
\usepackage{tabularx}
\usepackage{threeparttable}
\usepackage{enumerate}
\usepackage[normalem]{ulem}
\graphicspath{{graphh/}}

\newtheorem{theorem}{Theorem}

\newtheorem{proposition}{Proposition}
\newtheorem{definition}[theorem]{Definition}

{\theoremstyle{remark} }

\numberwithin{equation}{section}

\definecolor{electricpurple}{rgb}{0.75,0.0,1.0}
\definecolor{darkred}{rgb}{0.65,0,0}
\definecolor{green}{rgb}{0.0, 0.5, 0.0}

 \newcommand\deletei{\bgroup\markoverwith{\textcolor{darkred}{\rule[0.5ex]{1pt}{1pt}}}\ULon}
 \newcommand\deleteii{\bgroup\markoverwith{\textcolor{blue}{\rule[0.5ex]{2pt}{2pt}}}\ULon}

\title {A spectral method for a Fokker-Planck equation in neuroscience with applications in neural networks with learning rules }

\author{Pei Zhang\thanks{Beijing Computational Science Research Center, Beijing, China, 100193, email: {\tt zhangpei@csrc.ac.cn}},
   ~~Yanli Wang\thanks{Beijing Computational Science Research Center, Beijing, China, 100193, email: {\tt ylwang@csrc.ac.cn}},
   ~~Zhennan Zhou\thanks{Beijing International Center for Mathematical Research, Peking University, Beijing, China, 100871, email: {\tt zhennan@bicmr.pku.edu.cn}}.
    }

\begin{document}
\maketitle
%\tableofcontents
%\clearpage 

\begin{abstract}
    In this work, we consider the Fokker-Planck equation of the Nonlinear Noisy Leaky Integrate-and-Fire (NNLIF) model for neuron networks. Due to the firing events of neurons at the microscopic level, this Fokker-Planck equation contains dynamic boundary conditions involving specific internal points. To efficiently solve this problem and explore the properties of the unknown, we construct a flexible numerical scheme for the Fokker-Planck equation in the framework of spectral methods that can accurately handle the dynamic boundary condition. This numerical scheme is stable with suitable choices of test function spaces, and asymptotic preserving, and it is easily extendable to variant models with multiple time scales. We also present extensive numerical examples to verify the scheme properties, including order of convergence and time efficiency, and explore unique properties of the model, including blow-up phenomena for the NNLIF model and learning and discriminative properties for the NNLIF model with learning rules.
\end{abstract}

\vspace*{4mm}

\section{Introduction}\label{sec:introduction}

In recent years, there has been a growing interest in studying large-scale neuron network models, e.g. \cite{nykamp2000population}\cite{caceres2011analysis}\cite{renart2004mean}\cite{delarue2015global}, placing a greater emphasis on the proper mathematical tools for analyzing and simulating the dynamics of such networks. The underlying formulation of these models is based on deterministic or stochastic differential equations which describe the activities of neuron ensembles.

In this article, we consider the Nonlinear Noisy Leaky Integrate-and-Fire (NNLIF)  model, which was originally proposed in the pioneering works \cite{brunel1999fast}\cite{brunel2000dynamics}, and it is one of the fundamental models in computational neuroscience. In the microscopic perspective, this model takes the membrane potential $v$ of neurons as the state variable, which is restricted by a given threshold value $V_F$ \cite{renart2004mean}\cite{sirovich2006dynamics}\cite{omurtag2000simulation}\cite{mattia2002population}. A defining characteristic of this model is the inclusion of firing events, which are described by a reset mechanism: when the membrane potential $v$ reaches the threshold value of $V_F$, a spike occurs, and the membrane potential is then reset to a lower value $V_R$. Moreover, the neurons within an ensemble interact with each other only through spikes. In the macroscopic perspective, this model is related to the Fokker-Planck equation \cite{nykamp2000population}\cite{liu2022rigorous}\cite{Liu2021investigating}, as follows:
\begin{equation}
    \label{eq:problem1}
    \begin{cases}
        \partial_{t}p+\partial_{v}(hp)-a\partial_{v v}p=0,\qquad v\in(-\infty,V_F]/\{V_R\},\\
        p(v,0)=p^0(v),\qquad p(-\infty,t)=p(V_F,t)=0,\\
        p(V^-_R,t)=p(V^+_R,t),\quad \partial _vp(V^-_R,t)=\partial _vp(V^+_R,t)+\frac{N(t)}{a},\\
    \end{cases}
\end{equation}
where the probability density function $p(v,t)$ represents the probability of finding a neuron at voltage $v$ and given time $t$, and $p^0(v)$ is the initial condition. The spiking behavior is described by the mean firing rate $N (t)$, which is implicitly given by
\begin{equation}
    \label{eq:Nt}
    N(t)=-a(N(t))\frac{\partial p}{\partial v}(V_F,t).
\end{equation}
The drift coefficient $h$ and diffusion coefficient $a$ are typically expressed as functions of the mean firing rate $N(t)$
\begin{equation}
    \label{eq:ha}
    h(v,N)=-v+bN,\qquad a(N)=a_0+a_1N,
\end{equation}
where $-v$ models the leaky mechanism and $b$ represents the connectivity of the network: $b > 0$ for excitatory-average networks and $b < 0$ for inhibitory-average networks. The connectivity of the network has an essential effect on the properties of \eqref{eq:problem1}, such as its steady states and blow-up phenomenon. Besides, $a$ stands for the amplitude of the noise, where $a_0 > 0$ and $a_1 \geq 0$. 
The probability density function $p(v, t)$  should satisfy the condition of conservation of mass
\begin{equation}
    \int_{-\infty}^{V_{F}} p(v, t) d v=\int_{-\infty}^{V_{F}} p^{0}(v) d v=1.
\end{equation}

In recent years, there have been significant progress in the numerical and analytical studies of the NNLIF models.
\cite{delarue2015global}\cite{delarue2015particle}\cite{caceres2020understanding} analyze the stability and asymptotic behavior from the point of view of microscopic stochastic differential equation (SDE). From the macroscopic perspective, \cite{caceres2011analysis}\cite{carrillo2013classical}\cite{carrillo2015qualitative} establish the existence theory of the Fokker-Planck equation \eqref{eq:problem1}, and the classical solution exists only when  the firing rate N(t) does not diverge. In \cite{caceres2011analysis}, the authors analyze the model's steady states and blow-up phenomenon. In \cite{dou2022dilating}, aiming at investigating the solution structure in the presence of the blow-up phenomenon, a new notion of generalized solution is proposed by introducing the dilated time scale.

Some variants of the model \eqref{eq:problem1} have been studied to incorporate more biological ingredients \cite{caceres2018analysis}, including multi-species populations (excitatory and inhibitory) \cite{caceres2016blow}, the refractory state \cite{caceres2014beyond}\cite{sharma2020discontinuous}, the transmission delay between neurons \cite{caceres2019global}\cite{hu2021structure}\cite{sharma2020discontinuous} and the age of the neuron \cite{dumont2016noisy}. Besides, there have also been multi-scale models with additional state variables. For example, the kinetic voltage-conductance model for neuron networks has been explored in \cite{caceres2011numerical}\cite{dou2022bounds}\cite{carrillo2022simplified}.  In \cite{he2022structure}\cite{perthame2017distributed}, the authors consider the learning behavior of the NNLIF model which is structured by the synaptic weights, and the distribution of the weights evolve according to the Hebbian learning rule. 

In this paper, we focus on the Fokker-Planck equation \eqref{eq:problem1}, and the primary goal is to investigate its efficient numerical approximation and applications to other variants. This Fokker-Planck equation is distinguished from other classical kinetic models due to its complex nonlinearity through the boundary flux and the dynamic boundary conditions. In spite of the existing results, the properties of this model are far from being fully understood, necessitating further numerical experiments to gain a deeper understanding. In \cite{caceres2011numerical}, the authors propose a numerical scheme combining the WENO-finite differences and the Chang-Cooper method. The numerical tests are mainly concerned with the blow-up phenomenon and the steady states. In \cite{hu2021structure}\cite{he2022structure}, the authors propose conservative and conditionally positivity-preserving schemes and show that the corresponding discrete entropy is dissipating in time. Besides, the finite element method and discontinuous Galerkin method are also applied to solve the Fokker-Planck type equations \cite{sharma2019numerical}\cite{sharma2020discontinuous}. Properly addressing the flux shift terms in the Fokker-Planck equation is of essential significance in the numerical approximation. In most cases, the flux offset term is implicitly included in the equation, requiring the modification of dynamic boundary conditions into equations with $\delta$ source terms for better implementation of numerical methods. Then the original problem \eqref{eq:problem1} is transformed into the following boundary value problem:
\begin{equation}
\label{eq:problem_delta}
    \begin{cases}
        \partial_{t}p+\partial_{v}(hp)-a\partial_{v v}p=N(t)\delta(v-V_R),\qquad v\in(-\infty,V_F],\\
        p(v,0)=p^0(v),\qquad p(-\infty,t)=p(V_F,t)=0,\\
    \end{cases}
\end{equation}
where $\delta(v)$ stands for the Delta function. Although  this transformation facilitates the construction of numerical schemes, it also causes certain restrictions due to the particularity of the $\delta$ function, such as requiring $V_R$ to fall on the grid point.

To achieve improved efficiency while properly handling the dynamic boundary conditions, we aim to construct a spectral approximation for the Fokker-Planck equation \eqref{eq:problem1}. The spectral method that we shall present relies on semi-globally differentiable and integrable basis functions, which accurately capture the dynamic boundary conditions in \eqref{eq:problem1} rather than complicating the equation with a $\delta$ source as in \eqref{eq:problem_delta}. 

There are two key factors that enable our scheme to meet the desired properties. First, the construction of the basis functions enforces the approximate solution to satisfy the dynamic boundary conditions exactly. Second, the time evolution of the approximate solution is determined by the Galerkin method, and there exist suitable choices of the test function spaces  that make the method stable and asymptotic preserving.

Beyond that, we perform systematic numerical tests to verify the convergence order of the method and to investigate the diverse solution properties such as the blow-up phenomenon and discrete relative entropy. Taking advantage of the scheme's flexibility, we apply it to the NNLIF model with learning rules proposed in \cite{perthame2017distributed}, testing the discrimination property and further exploring the learning behavior of the model.

The paper is structured as follows. In Section \ref{sec:weak_form}, we define the weak solution of \eqref{eq:problem1} and establish its relationship to the classical solution. In Section \ref{sec:scheme}, we introduce the numerical scheme for the NNLIF model based on spectral methods in detail and analyze the choices of different test functions in constructing the numerical solution. In Section \ref{sec:lr}, we introduce the NNLIF model with learning rules and apply the proposed method to the model. In Section \ref{sec:numerical_test}, we perform some numerical experiments, including the convergence order of the scheme, comparison with existing numerical methods, and some other numerical explorations.
\section{Weak formulation}\label{sec:weak_form}

In this section, we introduce the weak formulation of the problem, which is the foundation for constructing numerical solutions. The link between 
the classical solution and the weak solution of this model will be analyzed as well.

%Spatial discretization involves constructing a weak formulation of the problem \eqref{eq:problem1} over a given domain. 
For simplicity, we choose a finite interval $[V_{\min}, V_F]$ as the computation domain and suppose $V_{\min}$ is small enough such that the density function $p(v,t)$ for  $v < V_{\min}$ is negligible. Then the semi-unbounded problem \eqref{eq:problem1} can be truncated to boundary value problem as follow:
\begin{equation}
    \label{eq:problem2}
    \begin{cases}
        \partial_{t}p+\partial_{v}(hp)-a\partial_{v v}p=0,\qquad v\in[V_{\min},V_F]/\{V_R\},\\
        p(v,0)=p^0(v),\qquad p(V_{\min},t)=p(V_F,t)=0,\\
        p(V^-_R,t)=p(V^+_R,t),\quad \partial _vp(V^-_R,t)=\partial _vp(V^+_R,t)+\frac{N(t)}{a}.\\
    \end{cases}
\end{equation}
The truncated equation \eqref{eq:problem2} should still satisfy the mass conservation, i.e.
\begin{equation}
    \label{eq:mass_conservation}
    \int_{V_{\min}}^{V_{F}} p(v, t) d v=\int_{V_{\min}}^{V_{F}} p^{0}(v) d v=1.
\end{equation}
By integrating \eqref{eq:problem2} and using the boundary conditions therein, this conversation implies the following boundary condition.
\begin{equation}
    \label{eq:leftd}
    \frac{\partial}{\partial v}p(V_{\min},t)=0.
\end{equation}
In fact, \eqref{eq:leftd} is never precisely satisfied, but as long as $V_{\min}$ is chosen properly, $\partial_vp(V_{\min},t)$ is negligible.

%Before defining the weak solution, it is necessary to have a clear interpretation of the classical solution. 
 We adopt the definition of the classical solution in \cite{carrillo2013classical}\cite{liu2022rigorous} for the truncated problem.
\begin{definition}[classical solution] \label{class_solution}
    For any given $0<T<+\infty$, $p(v,t)$ is a classical solution of \eqref{eq:problem2} in the time interval $(0, T]$  in the following sense:
    \begin{itemize}
        \item[1.] $N(t)=-a\partial_vp(V_F^-,t)$ is a continuous function for $t\in [0,T]$,
        \item[2.] $p(v,t)$ is continuous in the region $\{(v,t):V_{\min}<v<V_F, t\in[0,T]\}$,
        \item[3.] $p_{vv}$ and $p_t$ are well defined in the region $\{(v,t): v\in [V_{\min},V_R)\cup(V_R,V_F], t \in (0,T]\}$,
        \item[4.] $p_v(V_R^-,t)$ and $p_v(V_R^+,t)$ are well defined for $t\in(0,T]$,
        \item[5.] For $t\in (0,T]$, equation \eqref{eq:problem2} is satisfied,
        \item[6.] $p(v,0)=p^0(v)$ for $v\in [V_{\min},V_R)\cup(V_R,V_F] $.
    \end{itemize}
\end{definition}
In this paper, we consider classical solutions of \eqref{eq:problem2} which additionally satisfy \eqref{eq:leftd}. Having explicitly defined the classical solution of \eqref{eq:problem2}, we can now move on to discuss the weak solution.

If $p(v,t)$ is the classical solution of \eqref{eq:problem2} , weak formulation of \eqref{eq:problem2} is obtained by multiplying \eqref{eq:problem2} with some test function $\phi \in C^{\infty}([V_{\min}, V_F])$ and integrating over $[V_{\min}, V_F]$
\begin{equation}
    \label{variational_1}
    \int_{V_{\min}}^{V_{F}} \left(
    \partial_{t}p+\partial_{v}(hp)-a\partial_{v v}p\right)\phi dv =0.
\end{equation}
Integrating by parts in intervals $[V_{\min},V_R]$ and $[V_R,V_F]$ respectively, we obtain 
\begin{equation}
    \label{eq:int_by_part}
\begin{aligned}
    &\int_{V_{\min}}^{V_{F}} \left(\partial_tp \phi-hp\partial_v\phi+a\partial_vp\partial_v\phi\right) dv\\
    +&\left(hp\phi|_{V_{\min}}^{V_R^-}+hp\phi|_{V_R^+}^{V_F}\right)-\left(a\partial_vp\phi|_{V_{\min}}^{V_R^-}+a\partial_vp\phi|_{V_{V_R^+}}^{V_F}\right)=0.
\end{aligned}
\end{equation}
By substituting the boundary conditions in \eqref{eq:problem2} and \eqref{eq:leftd}, \eqref{eq:int_by_part} can be simplified as
\begin{equation}
    \label{variational_2}
    \int_{V_{\min}}^{V_{F}} \left(\partial_tp \phi-hp\partial_v\phi+a\partial_vp\partial_v\phi\right) dv+a\partial_vp(V_F)\left(\phi(V_R)-\phi(V_F)\right)=0.
\end{equation}
The above derivation helps to formally introduce the definition of the weak solution of \eqref{eq:problem2}.
\begin{definition}[weak solution] \label{weak_solution}
    The variational space appropriate for the present case is
    \begin{equation}
        \label{eq:variational_space}
        \mathbb{H}^{1}_0(V_{\min},V_F)=\{p\in \mathbb{H}^{1}(V_{\min},V_F):p|_{V_{\min}}=p|_{V_F}=0\}.
    \end{equation}
    We say  $p(v,t)\in  C^{1}([0,T];\, \mathbb{H}^{1}_0(V_{\min},V_F)) $ is a weak solution of \eqref{eq:problem2}  if for any test function $\phi(v) \in \mathbb{H}^{1}(V_{\min},V_F)$, \eqref{variational_2} holds for $\forall t\in (0,T]$ and $p(v,0)=p^0(v)$. 
\end{definition}
The weak solution in Definition \ref{weak_solution} still inherits the essence of the original problem \eqref{eq:problem2}, and the relation between the weak solution and the classical solution is established in the following.
\begin{theorem} [Relation with the classical solution]
If $p(v,t)$ is a classical solution of \eqref{eq:problem2} in the time interval $(0, T]$ which also satisfies \eqref{eq:leftd}, then it is a weak solution of \eqref{eq:problem2} in the same time interval. Conversely, if $p(v,t)$ is a weak solution of \eqref{eq:problem2} in the time interval $(0, T]$ and additionally we assume that $$p(v,t) \in C^{1}\left((0,T];\,C^{2}\left([V_{\min},V_R)\cup(V_R,V_F]\right)\right) $$ satisfies $p(V_R^-,t)=p(V_R^+,t)$ and the one-sided derivatives of $p(v,t)$ exist at each side of $V_R$ for all $t\in(0,T] $, then it is a classical solution of \eqref{eq:problem2} in the same time interval and it satisfies \eqref{eq:leftd}.
\end{theorem}
\begin{proof}
The first part of the theorem can obviously be proved by the derivation of the weak solution. 

For the other direction, let $p(v,t)$ be a weak solution of \eqref{eq:problem2} in the time interval $(0, T]$, and $p$ satisfies all the additional assumptions in the statement. We aim to prove that $p(v,t)$ is a classical solution in Definition \ref{class_solution}, and satisfies \eqref{eq:leftd}.

By the definition of the weak solution and the additional conditions it satisfies, it is straightforward to show that the solution $p$ meets the first four and the last criteria laid out in Definition \ref{class_solution}. In particular, the smoothness assumption at $V_F$ (from the left-hand side) implies the continuity of $N(t)$.  In the following, we will thoroughly demonstrate that $p$ conforms to the fifth item of Definition \ref{class_solution} and \eqref{eq:leftd}. By integration by parts, \eqref{variational_2} can be rewritten as 
\begin{equation}
\label{eq:int_by_part2}
\begin{aligned}
    &\int_{V_{\min}}^{V_{F}} \left(\partial_{t}p+\partial_{v}(hp)-a\partial_{v v}p\right)\phi dv\\ 
    -&\left(hp\phi|_{V_{\min}}^{V_R^-}+hp\phi|_{V_R^+}^{V_F}\right)+\left(a\partial_vp\phi|_{V_{\min}}^{V_R^-}+a\partial_vp\phi|_{V_{V_R^+}}^{V_F})+a\partial_vp(V_F)(\phi(V_R)-\phi(V_F)\right)=0.
\end{aligned}
\end{equation}
The definition of $\mathbb{H}^{1}_0(V_{\min},V_F)$ states that $p(V_{\min})=p(V_F)=0$, thus
\begin{equation}
    \label{eq:int_by_part3}
    \begin{aligned}
        &\int_{V_{\min}}^{V_{F}} \left(\partial_{t}p+\partial_{v}(hp)-a\partial_{v v}p\right)\phi dv-h(V_R)\phi(V_R)\left(p(V_R^-)-p(V_R^+)\right)\\
        +&a\phi(V_R)\left(\partial_vp(V_R^-)-\partial_vp(V_R^+)+\partial_vp(V_F)\right)-a\partial_vp(V_{\min})\phi(V_{\min})=0,
    \end{aligned}
\end{equation}
The key to the proof is selecting different test function spaces to simplify \eqref{eq:int_by_part3}, such that the equations identified in \eqref{eq:problem2} and the boundary conditions delineated in \eqref{eq:problem2} and \eqref{eq:leftd} are successively established. First, taking the test functions $\phi \in \mathbb{V}_1(V_{\min},V_F)=\{\phi \in \mathbb{H}^1(V_{\min},V_F): \phi(V_R)=\phi(V_{\min})=0\}$, \eqref{eq:int_by_part3} reduce to
\begin{equation}
    \int_{V_{\min}}^{V_{F}} \left(\partial_{t}p+\partial_{v}(hp)-a\partial_{v v}p\right)\phi dv=0.
\end{equation}
 Since $p\in C^{2}\left([V_{\min},V_R)\cup(V_R,V_F]\right)$, $\partial_{t}p+\partial_{v}(hp)-a\partial_{v v}p$ is continuous on the interval $(V_{\min},V_R)\cup(V_R,V_F)$, and it can be inferred from the arbitrariness of $\phi$ that $p$ satisfies
\begin{equation}
    \label{pde_equation}
    \partial_{t}p+\partial_{v}(hp)-a\partial_{v v}p=0, \quad \forall v\in (V_{\min},V_R)\cup(V_R,V_F).
\end{equation}
This verifies that the equation in \eqref{eq:problem2} holds for  $p(v,t)$ within the interval, and the next step in the proof is that the weak solution satisfies the boundary conditions in \eqref{eq:problem2} and \eqref{eq:leftd}. By the definition of trial function space $\mathbb{H}^{1}_0(V_{\min},V_F)$, it is easy to see that $p(v,t)$ satisfies the following boundary conditions
\begin{equation}\label{eq:Dirichlet_boundary}
\begin{aligned}
    &p(V_{\min},t)=p(V_F,t)=0.\\
\end{aligned}
\end{equation}
Changing the test functions $\phi \in \mathbb{V}_2(V_{\min},V_F)=\{\phi \in \mathbb{H}^1(V_{\min},V_F): \phi(V_R)=0\}$ and using \eqref{pde_equation}, \eqref{eq:int_by_part3} can be written as
\begin{equation}
    a\partial_vp(V_{\min})\phi(V_{\min})=0.
\end{equation}
Since the arbitrariness of $\phi(V_{\min})$, we obtain
\begin{equation}\label{eq:left_boundary}
    \partial_vp(V_{\min})=0.
\end{equation}
Similarly, changing the test functions $\phi \in \mathbb{V}_3(V_{\min},V_F)=\{\phi \in \mathbb{H}^1(V_{\min},V_F): \phi(V_{\min})=0\}$ again and using \eqref{pde_equation}, \eqref{eq:int_by_part3} is reduced into
\begin{equation}
    -h(V_R)\phi(V_R)\left(p(V_R^-)-p(V_R^+)\right)+a\phi(V_R)\left(\partial_vp(V_R^-)-\partial_vp(V_R^+)+\partial_vp(V_F)\right)=0.
\end{equation}
Since $p(V^-_R)=p(V^+_R) $ and the arbitrariness of $\phi(V_R)$, we deduce
\begin{equation}
    \partial_vp(V_R^-)-\partial_vp(V_R^+)+\partial_vp(V_F)=0.
\end{equation}
By the definition of trace, $p(v,t)$ satisfies boundary conditions in \eqref{eq:problem2} and \eqref{eq:leftd}. Now, we have proved that $p(v,t)$ satisfies the fifth item in Definition \ref{class_solution}.  To conclude, we have shown that $p(v,t)$ is a classical solution of equation \eqref{eq:problem2}

\end{proof}

\section{Numerical scheme and analysis}\label{sec:scheme}
In this section, we present a spectral approximation for the weak solution to the Fokker-Planck equation \eqref{eq:problem2} and construct a fully discrete numerical scheme. Numerical solutions are sought in a specific function space in which the functions satisfy the boundary conditions and can be determined by solving the derived equation system after specifying the test function space.
\subsection{A fully discrete numerical scheme based on Legendre spectral method}\label{sec:fully_discrete_scheme}
In this part, we construct the numerical scheme of the Fokker-Planck equation \eqref{eq:problem2}, which is implemented in two steps. First, the spectral approximation is used for space discretization, resulting in a system of ordinary differential equations; second, a semi-implicit scheme is applied for time discretization. The spectral method is established such that the numerical solution inherently satisfies the boundary conditions. Formally, the approximate variational problem is
\begin{equation}
    \label{var_problem}
    \begin{cases}
        \text{Find } p\in \mathrm{W} \text{ such that}\\
        \int_{V_{\min}}^{V_{F}} \left(
    \partial_{t}p+\partial_{v}(hp)-a\partial_{v v}p\right)\phi =0,\quad \forall \phi \in \mathrm{V},
    \end{cases}
\end{equation}
where $\mathrm{W}$ is the trial function space and $\mathrm{V}$ is the test function space. Compared to Definition \ref{weak_solution}, the variational problem \eqref{var_problem} requires a more complex trial function space, which will be further described below. The specific form of the test function space will be introduced in Section \ref{sec:stability}.
\subsubsection{Construction of trial function space and space discretization}
A challenging aspect of the spectral method is constructing the trial function space $\mathrm{W}$ so as to satisfy the complex boundary conditions, including the discontinuous derivative of the density function and the dynamic boundary. To that end, the trial function space should be a subset of $\mathbb{H}^1_0$ wherein strong boundary derivatives can be defined. Specifically, the polynomial space that fulfills the boundary conditions in \eqref{eq:problem2} and \eqref{eq:leftd} can be used as the trial function space. That is $\mathrm{W} \subseteq \mathrm{P}_{\infty}(V_{\min}, V_R) + \mathrm{P}_{\infty}(V_{R}, V_F)$ and for all $ p\in \mathrm{W}$, it holds that
\begin{equation}
    \label{boudary_condition}
    \begin{cases}
        p(V_{\min})=\partial _vp(V_{\min})=0,\\
        p(V_F)=0,\\
        p(V^-_R)=p(V^+_R),\\
        \partial _vp(V^-_R)=\partial _vp(V^+_R)+\partial _vp(V_F),
    \end{cases}
\end{equation}
where $\mathrm{P}_{\infty}(a,b)$ is the set of all real polynomials defined on the interval $(a, b)$. 
With integration by parts and the boundary conditions \eqref{boudary_condition} in the trial function space, the solution to the above variational problem \eqref{var_problem} agrees with the weak solution specified in Definition \ref{weak_solution}.

Let $\{\psi_k\}_{k=0}^{\infty}$ be a set of basis functions of $\mathrm{W}$. The approximate solution of problem \eqref{eq:problem2} can be expanded as
\begin{equation}
    \label{eq:approximate_solution1}
    p(v,t)=\sum_{k=0}^{\infty}\hat{u}_k(t)\psi_k(v).
\end{equation}
The essence of constructing the trial function space $\mathrm{W}$ is to determine the specific form of its basis functions $\{\psi_k\}_{k=0}^{\infty}$. This is accomplished by dividing the interval into two segments by the discontinuity point $V_R$, utilizing a fixed number of basis functions to meet the dynamic boundary conditions, and employing basis functions with homogeneous boundary conditions for each segment to improve accuracy. That is
\begin{equation}
    \mathrm{W}=\mathrm{W}_1+\mathrm{W}_2,
\end{equation}
where $\mathrm{W}_1$ is a finite-dimensional space that handles the conditions in \eqref{boudary_condition}, and $\mathrm{W}_2$ enhances accuracy within the interval and satisfies the homogeneous conditions of points $V_{\min}$, $V_R$, and $V_F$, which is
\begin{equation}
    \label{eq:W2condition}
    \begin{cases}
        p(V_{\min})=\partial _vp(V_{\min})=0,\\
         p(V_R^-)=\partial _vp(V_R^-)=0,\\
         p(V_R^+)=\partial _vp(V_R^+)=0,\\
        p(V_F)=\partial _vp(V_F)=0,\\
    \end{cases}\qquad \forall p \in \mathrm{W}_2,
\end{equation}

For simplicity, it is preferable to keep the dimension of $\mathrm{W}_1$ as low as possible. In the case of taking into account the function value and first derivative value, there are eight degrees of freedom at the boundary, comprising of the function value and derivative value at $V_{\min}$, $V_F$, and both sides of $V_R$. Since the five conditions in \eqref{boudary_condition} have to be satisfied, there are three degrees of freedom remaining. Therefore, $\mathrm{W}_1$ can be spanned by three basis functions
\begin{equation}
    \mathrm{W}_1=\text{span}\{g_1,g_2,g_3\},
\end{equation}
where
\begin{equation}
    g_1\Rightarrow
    \begin{gathered}
        \begin{cases}
            g_1(V_{\min})=0,\\
            g_1(V_R)=1,\\
            \partial_v g_1(V_{\min})=0,\\
            \partial_v g_1(V_R)=0,
        \end{cases}\quad v\in(V_{\min},V_R),\qquad
        \begin{cases}
            g_1(V_R)=1,\\
            g_1(V_F)=0,\\
            \partial_v g_1(V_R)=0,\\
            \partial_v g_1(V_F)=0,
        \end{cases}\quad v\in(V_R,V_F).
    \end{gathered}
\end{equation}
\begin{equation}
    g_2\Rightarrow
    \begin{gathered}
        \begin{cases}
            g_2(V_{\min})=0,\\
            g_2(V_R)=0,\\
            \partial_v g_2(V_{\min})=0,\\
            \partial_v g_2(V_R)=1,
        \end{cases}\quad v\in(V_{\min},V_R), \qquad
        \begin{cases}
            g_2(V_R)=0,\\
            g_2(V_F)=0,\\
            \partial_v g_2(V_R)=1,\\
            \partial_v g_2(V_F)=0,
        \end{cases}\quad v\in(V_R,V_F).
    \end{gathered}
\end{equation}
\begin{equation}
    g_3\Rightarrow
    \begin{gathered}
    \begin{cases}
            g_3(V_{\min})=0,\\
            g_3(V_R)=0,\\
            \partial_v g_3(V_{\min})=0,\\
            \partial_v g_3(V_R)=0,
        \end{cases}\quad v\in(V_{\min},V_R), \qquad
    \begin{cases}
        g_3(V_R)=0,\\
        g_3(V_F)=0,\\
        \partial_v g_3(V_R)=1,\\
        \partial_v g_3(V_F)=1,
    \end{cases}\qquad v\in(V_R,V_F).\\
    \end{gathered}
\end{equation}
 The specific form of the basis functions in $\mathrm{W}_1$ are presented in Appendix \ref{app:basis}.

\begin{figure}[!htb]
    \centering
        \begin{minipage}[c]{0.8\textwidth}
            \centering
            \includegraphics[width=1\textwidth]{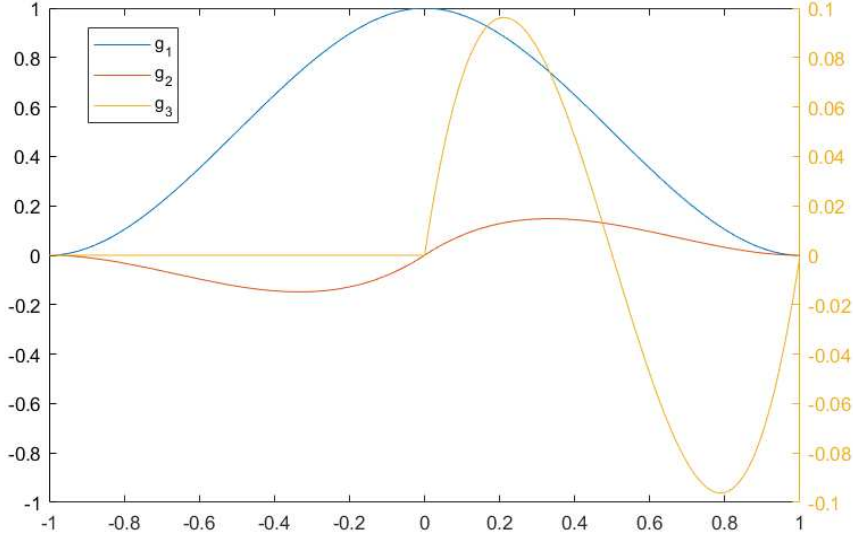}
        \end{minipage}
         \caption{The basis functions of $p_3$ with Equation parameters $V_{\min}=-1,V_R=0,V_F=1$. Here, $g_1$ and $g_2$ are measured using the left axis, while $g_3$ is measured with the right axis.}
        \label{fig:g}
\end{figure}
The specific illustration of $g_1,g_2,g_3$ are shown in Figure \ref{fig:g}. After the basis function of $\mathrm{W}_1$ is determined in this way, we set
\begin{equation}
    \label{eq:w1_expansion}
    p_3=\sum_{k=1}^3 \lambda_kg_k \quad \in \mathrm{W}_1.
\end{equation}
The boundary conditions can be well satisfied by adjusting the coefficients of $g_1,g_2,g_3$ in the following way,
\begin{equation}
    \label{eq:lambda}
    \begin{gathered}
        \begin{cases}
            p(V_R^-)=p_3(V_R^-)=\lambda_1,\\
            p(V_R^-)=p_3(V_R^+)=\lambda_1,
        \end{cases}\qquad
        \begin{cases}
            \partial_vp(V_R^-)=\partial_vp_3(V_R^-)=\lambda_2,\\
            \partial_vp(V_R^-)=\partial_vp_3(V_R^+)=\lambda_2+\lambda_3,\\
            \partial_vp(V_R^-)=\partial_vp_3(V_F^-)=\lambda_3,
        \end{cases}
    \end{gathered}
\end{equation}
where $p$ denotes the numerical solution mentioned in \eqref{eq:approximate_solution1}. 

The construction of the $\mathrm{W}_2$ space is motivated by spectral methods for solving general homogeneous boundary value problems.  According to \eqref{eq:W2condition}, the interval $(V_{\min},V_F) $ is divided into two segments by $V_R$ naturally. Assuming $I_L=(V_{\min},V_R)$ and $I_R=(V_R,V_F)$, we further denote
\begin{equation}
    \label{eq:X_ab}
        \mathrm{X}_{(a,b)}=\left\{\varphi \in P_{\infty}(a,b) : \varphi(a)=\varphi(b)=\varphi'(a)=\varphi'(b)=0   \right\}.\\
\end{equation}
$\mathrm{X}_{(a,b)}$ represents the set of real polynomials defined on the interval $(a, b)$, where the function value and derivative are zero at boundary points. So we can divide $\mathrm{W}_2$ into two parts 
\begin{equation}
    \mathrm{W}_2=\mathrm{X}_{(V_{\min},V_R)}+ \mathrm{X}_{(V_R,V_F)}.
\end{equation}

In the spectral methods, in order to minimize the interaction of basis functions in the frequency space, the basis functions should take the form of adjacent orthogonal polynomials \cite{shen1994efficient}. Therefore, it is reasonable to use a compact combination of Legendre polynomials as basis functions of $\mathrm{X}_{(a,b)}$, namely,
\begin{equation}
    \hat{h}_k=\mathcal{H}_k+\alpha_k\mathcal{H}_{k+1}+\beta_k\mathcal{H}_{k+2}+\gamma_k\mathcal{H}_{k+3}+\eta_k\mathcal{H}_{k+4},\quad k=0,1,2,...,
\end{equation}
where $\mathcal{H}_k$ is the scaling of the kth-degree Legendre polynomial $L_k$
\begin{equation}
    \mathcal{H}_k(v)=L_k(x),\qquad x=\frac{v-\left(\frac{a+b}{2}\right)}{\frac{b-a}{2}},
\end{equation}
and the parameter $\{\alpha_k,\beta_k,\gamma_k,\eta_k\}$ are chosen to satisfy the boundary conditions in \eqref{eq:X_ab}
\begin{equation}
    \alpha_k=0,\, \beta_k=-\frac{4k+10}{2k+7},\,\gamma_k=0,\,\eta_k=\frac{2k+3}{2k+7}.
\end{equation}

After constructing the trial function space, spatial discretization will be discussed, yielding the system of ordinary differential equations for the coefficients. Let $\{h_k\}_{k=0}^{\infty}$ be the basis functions of $\mathrm{W}_2$. Then the approximate solution \eqref{eq:approximate_solution1} can be rewritten as
\begin{equation}
    \label{eq:approximate_solution2}
    p(v,t)=\sum_{k=0}^{\infty}u_k(t)h_k(v)+\sum_{k=1}^3\lambda_k(t)g_k(v).
\end{equation}
The basal functions in \eqref{eq:approximate_solution2} correspond to ones in \eqref{eq:approximate_solution1} in the following way
\begin{equation}
    \{\psi_k\}_{k=0}^{\infty}=\{g_k\}_{k=1}^{3}+\{h_k\}_{k=0}^{\infty}.
\end{equation}
And the expansion coefficients $\{u_k(t)\}_{k=0}^{\infty},\{\lambda_k(t)\}_{k=1}^3$ are to be determined.
Assuming the initial value is to satisfy the boundary conditions \eqref{boudary_condition}, the initial expansion coefficients $\{u_k(0)\}_{k=0}^{\infty},\{\lambda_k(0)\}_{k=1}^3$ can be obtained by the best squares approximation,
\begin{equation}
    \label{eq:initial_vector}
    \int_{V_{\min}}^{V_F} \left(\sum_{k=0}^{\infty}u_k(0)h_k(v)+\sum_{k=1}^3\lambda_k(t)g_k(v)\right) \phi_j dv=\int_{V_{\min}}^{V_F} p^0(v)\phi_j dv, \qquad \forall \phi_j \in \mathrm{V}.
\end{equation}
For a properly defined test function space, the solvability of the \eqref{eq:initial_vector} is guaranteed by the Gram-Schmidt orthogonalization of the basis functions. Note again that the specific form of the test function space is discussed in Section \ref{sec:stability}. We denote the initial value vector as
\begin{equation}
    \label{eq:initial_value}
    \mathbf{P^0}=(\lambda_1(0),\lambda_2(0),\lambda_3(0),u_1(0),u_2(0),...)^T.
\end{equation}

It should be noted that while constructing the basis functions, we assume that the value of $N(t)$ is already known. In fact, $N(t)$ is self-consistently determined in the dynamic process, and $N(t)$ is part of the degrees of freedom of the solution. It follows from \eqref{eq:lambda} that
\begin{equation}
    \partial_vp(V_F,t)=\lambda_3(t).
\end{equation}
One can rewrite the mean firing rate using \eqref{eq:ha} and \eqref{eq:Nt}
\begin{equation}
    \label{Nt1}
    N(t)=-\frac{a_0\lambda_3(t)}{1+a_1\lambda_3(t)}.
\end{equation}
Define
\begin{equation}
    \label{eq:operator}
    \mathcal{L}p(v,t)=\partial_{t}p-\partial_v(vp)-\left(b\frac{a_0\lambda_3(t)}{1+a_1\lambda_3(t)}\right)\partial_{v}p-\left(a_0-a_1\frac{a_0\lambda_3(t)}{1+a_1\lambda_3(t)}\right)\partial_{v v}p.
\end{equation}
The expansion coefficients $\{u_k(t)\}_{k=0}^{\infty},\,\{\lambda_k(t)\}_{k=1}^{3}(t>0)$ in \eqref{eq:approximate_solution2} can be determined by variational problem \eqref{var_problem} with using the mean firing rate $N(t)$ in \eqref{Nt1}:
\begin{equation}
    \label{eq:nonlinear_system1}
    \begin{cases}
        \text{Find } p\in  \mathrm{W} \text{ such that}\\
        (\mathcal{L}p,\phi_j)=0,\quad \forall \phi_j \in \mathrm{V},
    \end{cases}
\end{equation}
where $(\cdot,\cdot)$ is the inner product of the usual $L^2$ space.

The nonlinear system of ordinary differential equations of the above scheme is obtained by substituting \eqref{eq:approximate_solution2} into \eqref{eq:nonlinear_system1}. More precisely, setting
\begin{equation}
    \begin{aligned}
        &\mathbf{P}=(\lambda_1(t),\lambda_2(t),\lambda_3(t),u_1(t),u_2(t),...)^T,\\
        &s_{jk}=\begin{cases}
            (g_k,\phi_j),\qquad &1\leq k \leq3,\\
            (h_{k-4},\phi_j),& k\geq4.
        \end{cases}, &S=(s_{jk})_{j,k=1,2,...},\\
        &a_{jk}=\begin{cases}
            (\partial_v(vg_k),\phi_j),\qquad &1\leq k \leq3,\\
            (\partial_v(vh_{k-4}),\phi_j),& k\geq4.
        \end{cases}, &A=(a_{jk})_{j,k=1,2,...},\\
        &b_{jk}=\begin{cases}
            (\partial_vg_k,\phi_j),\qquad &1\leq k \leq3,\\
            (\partial_vh_{k-4},\phi_j),& k\geq4.
        \end{cases}, &B=(b_{jk})_{j,k=1,2,...},\\
        &c_{jk}=\begin{cases}
            (\partial_{vv}g_k,\phi_j),\qquad &1\leq k \leq3,\\
            (\partial_{vv}h_{k-4},\phi_j),& k\geq4.
        \end{cases}, &C=(c_{jk})_{j,k=1,2,...}.\\
    \end{aligned}
\end{equation}
The nonlinear system \eqref{eq:nonlinear_system1} becomes
\begin{equation}
    \label{eq:nonlinear_sde}
    S\partial_t\mathbf{P}=\left(A+\left(b\frac{a_0\lambda_3(t)}{1+a_1\lambda_3(t)}\right)B+\left(a_0-a_1\frac{a_0\lambda_3(t)}{1+a_1\lambda_3(t)}\right)C\right)\mathbf{P}.
\end{equation}
After the spatial discretization, the solution of problem \eqref{eq:problem2} converts into the solution of the nonlinear ordinary differential equation system of initial value problem \eqref{eq:nonlinear_sde}\eqref{eq:initial_value}.

\subsubsection{Fully discrete numerical scheme}

To finish the construction of the numerical scheme, we need to truncate the approximate solution \eqref{eq:approximate_solution2} to a finite-dimensional one and perform time discretization. The finite-dimensional form of \eqref{eq:X_ab} is denoted as
\begin{equation}
\label{eq:X_N}
    \mathrm{X}_{N(a,b)}=\left\{\varphi \in P_{N+3}(a,b) : \varphi(a)=\varphi(b)=\varphi'(a)=\varphi'(b)=0   \right\},
\end{equation}
where $P_N(a,b)$ is the set of all real polynomials of degree no more than $N$ and the dimension of $P_N{(a,b)}$ is $N+1$. It is evident that a  non-trivial polynomial with the homogeneous boundary conditions in \eqref{eq:X_N} must be of at least fourth degree, thus leading to a reduced dimension of the set in \eqref{eq:X_N}. The polynomial space in \eqref{eq:X_N} is selected as $P_{N+3}(a,b)$ so that the dimension of the $\mathrm{X}_{N(a,b)}$ space is $N$. Then the trial function space can be truncated as

\begin{equation}
    \label{eq:trial_function}
    \mathrm{W}_N=\mathrm{X}_{N(V_{\min},V_R)} + \mathrm{X}_{N(V_R,V_F)} + \mathrm{W}_1.
\end{equation}
Assuming $\{h^L_k\}_{k=0}^{N-1}$ is a set
of basis functions of $X_{N(V_{\min},V_R)}$ and $\{h^R_k\}_{k=0}^{N-1}$ is a set
of basis functions of $X_{N(V_R,V_F)}$. $\{\psi_k\}_{k=1}^{2N+3}=\{h^L_0,...,h^L_{N-1},h^R_0,...,h^R_{N-1},g_1,g_2,g_3\}$ is a basis of $\mathrm{W}_N$. Then the numerical solution $p_N(v,t)$ can be expressed as
\begin{equation}
    \label{eq:approximate_solution3}
    p_N(v,t)=\sum_{k=0}^{N-1} u_k^L(t)h^L_k(v)+\sum_{k=0}^{N-1} u_k^R(t)h^R_k(v)+\sum_{k=1}^3 \lambda_k(t)g_k(v)=\sum_{k=1}^{2N+3}\hat{u}_k(t)\psi_k(v).
\end{equation}
 The initial condition for the expansion coefficients $\{\hat{u}_{k}(0)\}_{k=0}^{2N+3}$ can be obtained by the least square approximation,
\begin{equation}
    \label{eq:initial_vector2}
    \int_{V_{\min}}^{V_F} \sum_{k=1}^{2N+3}\hat{u}_k(0)\psi_k(v) \phi_j dv=\int_{V_{\min}}^{V_F} p^0(v)\phi_j dv, \qquad \forall \phi_j \in \mathrm{V}_N.
\end{equation}

Suppose the truncated test function space is denoted by $\mathrm{V}_N$, which shall be specified later. The expansion coefficients $\{\hat{u}_k(t)\}_{k=0}^{2N+3}(t>0)$ can be determined by the semi-discrete variational formulation
\begin{equation}
    \label{eq:variational_form2}
    \begin{cases}
        \text{Find } p_N\in \mathrm{W}_{N} \text{ such that}\\
        (\mathcal{L}p_N,\phi_j)=0,\quad \forall \phi_j \in \mathrm{V}_N.
    \end{cases}
\end{equation}

For time discretization, we use a semi-implicit method. The interval $[0,T_{\text{max}}]$ is divided  into $n_t$ equal sub-intervals with size
\begin{equation}
    \Delta t=\frac{T_{\text{max}}}{n_t},
\end{equation}
and the grid points can be represented as follows
\begin{equation}
    t^{n}=n \Delta t, \qquad n=0,1,2, \cdots, n_{t}.
\end{equation}

The semi-implicit scheme of \eqref{eq:operator} is denoted by
\begin{equation}
    \label{eq:semi_implicit}
\begin{aligned}
    \tilde{\mathcal{L}}p_N(v,t^{n+1})&=\frac{p_N(v,t^{n+1})-p_N(v,t^{n})}{\Delta t}-\partial_v(vp_N(v,t^{n+1}))+bN(t^n)\partial_vp_N(v,t^{n+1})\\
    &-a(N(t^n))\partial_{vv}p_N(v,t^{n+1})=0,\qquad\qquad\qquad n=1,2,...,n_t.
\end{aligned}
\end{equation}
Note that, the mean firing rate $N(t^n)$ is treated explicitly, but the rest of the terms are implicit. Such a time discretization naturally avoids the use of a nonlinear solver. Then we can obtain the fully discrete scheme of the variational formulation \eqref{eq:variational_form2}: for each time step
\begin{equation}
    \label{eq:variational_form3}
    \begin{cases}
       \text{Find } p_N\in \mathrm{W}_{N} \text{ such that}\\
        (\tilde{\mathcal{L}}p_N,\phi_j)=0,\quad \forall \phi_j \in \mathrm{V}_N.
    \end{cases}
\end{equation}

More precisely, setting
\begin{equation}
    \label{eq:Matrix2}
    \begin{aligned}
        &\hat{\mathbf{P}}^n=(\hat{u}_1(t^n),\hat{u}_2(t^n),...,\hat{u}_{2N+3}(t^n))^T,\\
        &\hat{s}_{jk}=(\psi_k,\phi_j),\quad \hat{S}=(\hat{s}_{jk})_{k=1,...,2N+3}\\
        &\hat{a}_{jk}=(\partial_v(v\psi_k),\phi_j),\quad \hat{A}=(\hat{a}_{jk})_{k=1,...,2N+3}\\
        &\hat{b}_{jk}=(\partial_{v}\psi_k,\phi_j),\quad \hat{B}=(\hat{b}_{jk})_{k=1,...,2N+3}\\
        &\hat{c}_{jk}=(\partial_{vv}\psi_k,\phi_j),\quad \hat{C}=(\hat{c}_{jk})_{k=1,...,2N+3},
    \end{aligned}
\end{equation}
the  variational formulation \eqref{eq:variational_form3} reduces to
\begin{equation}
\label{eq:system2}
    \left(\frac{\hat{S}}{\Delta t}-\hat{A}+bN(t^n)\hat{B}-a(N(t^n))\hat{C}\right)\hat{\mathbf{P}}^{n+1}=\frac{\hat{S}}{\Delta t}\hat{\mathbf{P}}^n.
\end{equation}

\subsection{Stability and the choice of test functions} \label{sec:stability}

The dynamic boundary conditions also give rise to challenges in choosing proper finite-dimensional test function spaces.  As we shall elaborate below, the construction of the trial function is so delicate that we can not simply choose the test functions only out of accuracy. Our goal is to find test functions that result in a stable evolution system in the discrete setting, and we hope the total mass is conserved with satisfactory accuracy.

%The conservation property is the basic property of the model, and stability is an essential factor in ensuring the success of the numerical scheme. Having proposed the fully discrete numerical scheme in Section \ref{sec:fully_discrete_scheme}, it is necessary to determine an appropriate test function space, as it affects the properties of the numerical solution. 

To this end, two propositions are introduced that relate the test functions to the conservation and stability of the semi-discrete scheme \eqref{eq:variational_form2} in the linear case. However, the spectral method is often not able to completely ensure the conservation of mass, therefore it is not serving as a rigid criterion. The stability of the numerical solution is instead analyzed through its long-term asymptotic behavior in the linear regime, which will be discussed in greater detail below. Following this, three different test function spaces are analyzed respectively.

When analyzing the impact of the test function space, we are to consider the semi-discrete system \eqref{eq:variational_form2}. Thanks to the definition in \eqref{eq:Matrix2}, the system can reduce to
\begin{equation}
    \label{eq:SDE_system2}
    \hat{S}\partial_t\hat{\mathbf{P}}=\left(\hat{A}-\left(bN(t)\right)\hat{B}+\left(a(N(t))\right)\hat{C}\right)\hat{\mathbf{P}},\quad \hat{\mathbf{P}}(0)=\hat{\mathbf{P}}^0,
\end{equation}
where $\hat{\mathbf{P}}=(\hat{u}_1(t),\hat{u}_2(t),...,\hat{u}_{2N+3}(t))^T,\,\hat{\mathbf{P}}^0=(\hat{u}_1(0),\hat{u}_2(0),...,\hat{u}_{2N+3}(0))^T$. For simplicity, we study the case of a linear equation that is $b=0$ and $a(N)=1$. Then the nonlinear system \eqref{eq:SDE_system2} becomes a linear system
\begin{equation}
    \label{eq:linear_system2}
    \hat{S}\partial_t \hat{\mathbf{P}}=(\hat{A}+\hat{C})\hat{\mathbf{P}}, \quad \hat{\mathbf{P}}(0)=\hat{\mathbf{P}}^0.
\end{equation}
Considering the unique solvability of ordinary differential equations, we assume that the matrices $\hat{S}$, $\hat{A}$, and $\hat{C}$ are square matrices of order $2N+3$ and the matrix $\hat{S}$ is invertible. Let $\hat{K}=\hat{S}^{-1}(\hat{A}+\hat{C})$, then the system \eqref{eq:linear_system2} can be rewritten as
\begin{equation}
\label{eq:linear_system3}
    \hat{\mathbf{P}}_t=\hat{K}\hat{\mathbf{P}}.
\end{equation}
Let $\hat{\mathbf{P}}^{\infty}=(\hat{u}_1^{\infty},\hat{u}_2^{\infty},...,\hat{u}_{2N+3}^{\infty})^T$ be the steady-state solution of the equation. It holds that
\begin{equation}
    \label{eq:steady}
    \hat{K}\hat{\mathbf{P}}^{\infty}=0.
\end{equation}
The steady-state equation \eqref{eq:steady} has a nonzero solution if and only if the matrix $\hat{K}$ has at least one zero eigenvalue. With a prescribed test function space, the properties of the scheme can be assessed by inspecting the elements of matrix $\hat{K}$, allowing us to fully characterize the system's behavior. The following propositions serve to elucidate this connection.

\begin{proposition}[mass conservation]\label{prop1}
    Consider the Fokker-Planck equation \eqref{eq:problem2} with $a=1, b=0$ and the semi-discrete scheme \eqref{eq:linear_system3} where the dimension of test function space $V_N$ is $2N+3$. The following relations hold:
\begin{enumerate}
    \item Matrix $\hat{K}$ has zero eigenvalue if and only if the test function space $\mathrm{V}_N$ contains constant functions.
    \item If the matrix $\hat{K}$ has zero eigenvalue, then the total mass of the numerical solution solved by system \eqref{eq:linear_system2} does not change with time. That is
    \begin{equation}
        \int_{V_{\min}}^{V_F}\partial_tp_N(v,t) dv=0,
    \end{equation}
    where $p_N(v,t)$ is defined in \eqref{eq:approximate_solution3}.
\end{enumerate}
\end{proposition}
\begin{proof}
\textbf{Proof of (1)}. 
If the test function space $\mathrm{V}_N$ contains constants, without loss of generality, let $\phi_j=1$. Using the definition in \eqref{eq:Matrix2}, $\forall \varphi_i \in \mathrm{W}_N$, we can derive that
\begin{equation}
    \label{eq:integral}
    \begin{aligned}
        &\int_{V_{\min}}^{V_F}\partial_v(v\varphi_i) dv=v\varphi_i |_{V_{\min}}^{V_R^-}+v\varphi_i |_{V_R^+}^{V_F}=0,\\
        &\int_{V_{\min}}^{V_F}\partial_{vv}\varphi_i dv=\partial_v\varphi_i |_{V_{\min}}^{V_R^-}+\partial_v\varphi_i |_{V_R^+}^{V_F}=0.
    \end{aligned}
\end{equation}
So the elements of the $j$th row of the matrices $A$ and $C$ are all zeros. Since $\hat{S}$ is invertible, $S^{-1}$ is a full-rank matrix. Then
\begin{equation}
    \text{rank}(\hat{K})=\text{rank}(\hat{S}^{-1}(\hat{A}+\hat{C}))=\text{rank}(\hat{A}+\hat{C})<2N+3.
\end{equation}
So $\hat{K}$ has a zero eigenvalue.

If matrix $\hat{K}$ has a zero eigenvalue, then matrix $(\hat{A}+\hat{C})$ has zero eigenvalue for $\hat{S}$ is invertible. Therefore, the matrix $(\hat{A}+\hat{C})$ can make the elements in the jth row all zero through the matrix transformation. Notice that, performing matrix row transformation on matrix $(\hat{A}+\hat{C})$ corresponds to replacing the test function in linear system \eqref{eq:linear_system2}  with the linear combination of the original test function. Without loss of generality, we assume that the elements in the jth row of matrix $(\hat{A}+\hat{C})$ are all zeros and the corresponding test function is $\phi$. Using the definition in \eqref{eq:Matrix2}, $\forall \varphi_i \in \mathrm{W}_N$,  it holds that
\begin{equation}
    \begin{aligned}
        \int_{V_{\min}}^{V_F}\partial_v(v\varphi_i) +\partial_{vv}\varphi_i dv=-\int_{V_{\min}}^{V_F} (v\varphi_i+\partial_v\varphi_i)\partial_v\phi dv=0.
    \end{aligned}
\end{equation}
Since the above formula holds for all $\varphi_i \in \mathrm{W}_N$, so $\partial_v\phi=0$, that is $\phi= \text{constant}$.

\textbf{Proof of (2)}.
From (1), we know that when the matrix $\hat{K}$ has a zero eigenvalue,  the constant function $C_1 \in V_N$. Substituting $\phi_j=1$ into \eqref{eq:variational_form2} with $a=1,b=0$,
\begin{equation}
    \int_{V_{\min}}^{V_F} \partial_t p_N dv-\int_{V_{\min}}^{V_F}(\partial(vp_N)+\partial_{vv}p_N) dv\overset{\eqref{eq:integral}}{=}\partial_t \int_{V_{\min}}^{V_F} p_N dv=0.
\end{equation}
So the total mass does not change over time.
\end{proof}

\begin{proposition}[Stability]\label{prop2}
    Consider the Fokker-Planck equation \eqref{eq:problem2} with $a=1, b=0$ and the semi-discrete scheme \eqref{eq:linear_system3} . A necessary condition for the stability of the method is that all the eigenvalues of the matrix $\hat{K}$ are non-positive.
\end{proposition}
Note that a modified stability criterion is proposed here because the traditional stability conclusion cannot be applied due to the complexity of the equation. In the linear case, the equation has a unique steady state \cite{caceres2011analysis}, and the solution of the equation will converge exponentially to the steady state, so the discretized kinetic equation can only have non-positive eigenvalues. When there are positive eigenvalues, it means that the numerical scheme is unstable.

Following the theoretical analysis, we can now discuss the specific test function space. Our goal is to select suitable test function spaces such that the constructed numerical scheme is stable and preserves the original properties of the Fokker-Planck equation \eqref{eq:problem2} to the greatest extent, such as mass conservation. The Galerkin method is widely used in spectral methods \cite{shen2011spectral}. Consequently, Legendre-Galerkin Method is proposed below.
\paragraph{Legendre-Galerkin Method (LGM)}
The test function space is chosen to be the same as the trial function space. Applied to the semi-discrete method \eqref{eq:variational_form2} or its fully discrete version \eqref{eq:variational_form3} as 
\begin{equation}
    \mathrm{V}_N=\mathrm{W}_N,
\end{equation}
where $\mathrm{V}_N$ is test function space, and $\mathrm{W}_N$ is trial function space defined in \eqref{eq:trial_function}, we obtain a Legendre-Galerkin Method (LGM for short) for the model \eqref{eq:problem1}. 

The LGM method is numerically stable but total mass is not well conserved in dynamics. When constructing the trial function space, some low-order polynomials, especially constants, are discarded in order to satisfy the boundary conditions. For the LGM, the test function space does not contain constants, which fails to ensure mass conservation as stated in Proposition \ref{prop1}. Hence, this method can be used for  finite-time simulations, yet it is unsuitable for capturing long-time behavior or multiscale problems.

 According to Proposition \ref{prop1}, to improve the mass conservation property of the LGM, it seems that we may replace one of the basis functions in  the test function space with the constant function $1$. Say, we may consider the modified test function space
\begin{equation} \label{tildeV}
    \tilde{\mathrm{V}}_N=\mathrm{W}_N-\{\psi_k\} +\{1\},
\end{equation}
where $\{\psi_k\}(k=1,...,2N+3)$ is the basis function of the $\mathrm{W}_N$ space.

In this case, the mass of the numerical solution appears invariant. However, as shown in Figure \ref{fig:stable}, the matrix $\hat{K}$ in \eqref{eq:linear_system3} has positive eigenvalues for some $N$, which makes the method unstable, which agrees with Proposition \ref{prop2}. In fact, Figure \ref{fig:stable} shows that the maximum eigenvalue of the matrix $\hat{K}$ is significantly positive large  for when $N$ is odd and when the modified test function space $\tilde{\mathrm{V}}_N$ is used. Hence, we need to resort to other strategies for enhancing mass conservation. 
\begin{figure}[!htb]
    \centering
        \begin{minipage}[c]{0.49\textwidth}
            \centering
            \includegraphics[width=1\textwidth]{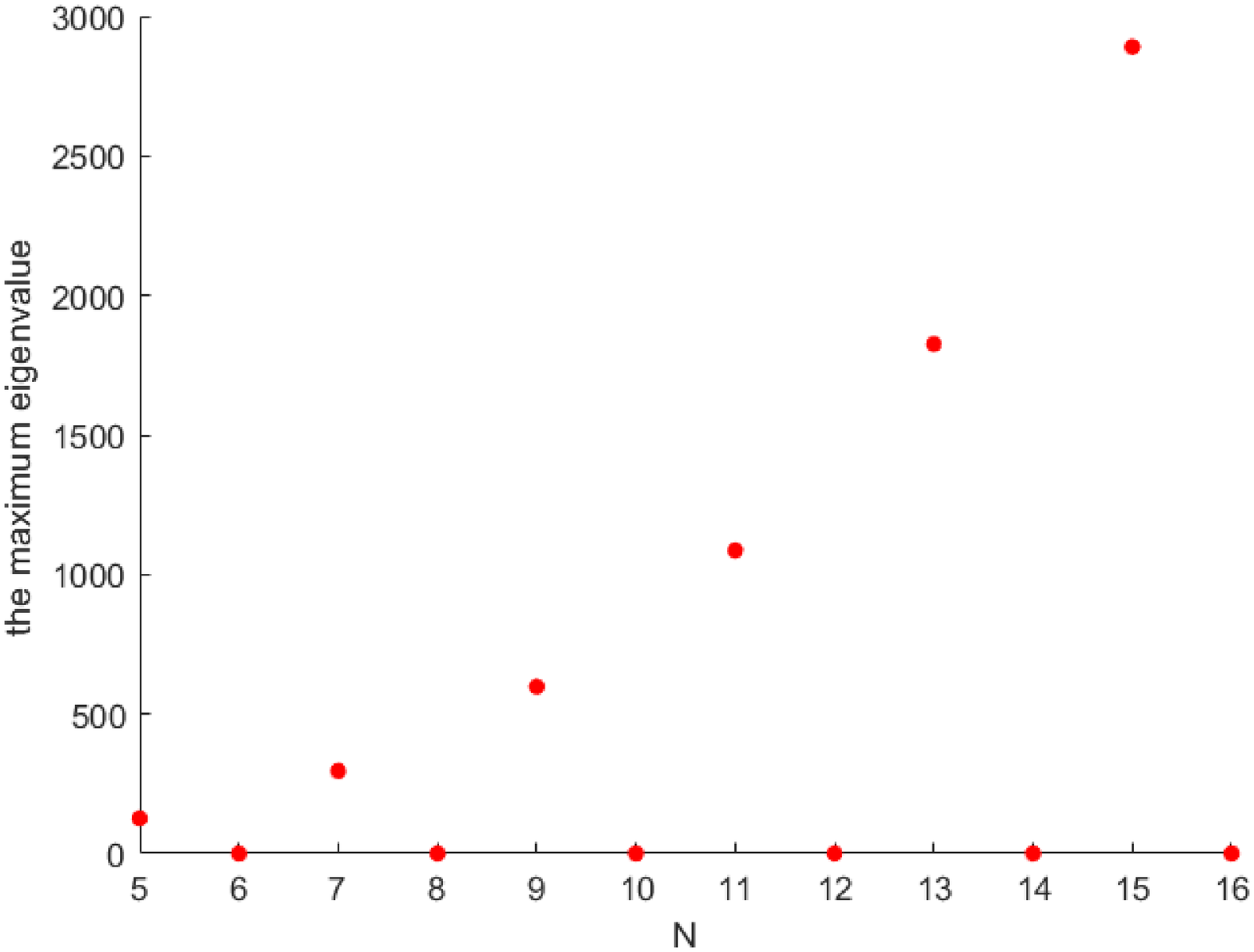}
        \end{minipage}
        \begin{minipage}[c]{0.49\textwidth}
            \centering
            \includegraphics[width=1\textwidth]{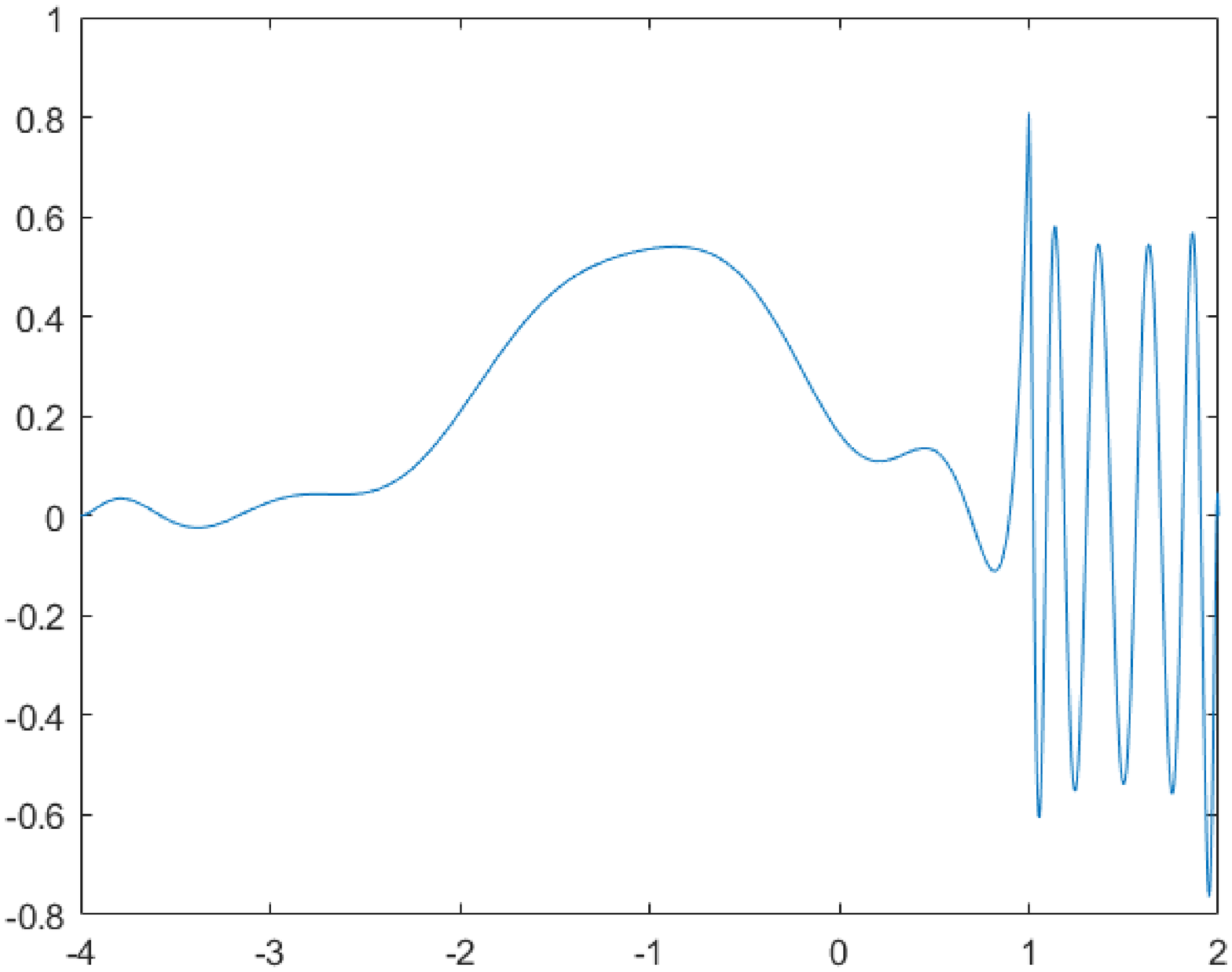}
        \end{minipage}
        \caption{When the test function space is $\tilde{\mathrm{V}}_N$ \eqref{tildeV}, the numerical method might be unstable. Left: The maximum eigenvalue of matrix $\hat{K}$ at different $N$. Right: A typical unstable solution. Equation parameters $a=1, b=0$ with Gaussian initial condition $v_0=-1, \sigma_0^2=0.5$ and $N=11,\Delta t=0.001$.}
        \label{fig:stable}
\end{figure}

\paragraph{Modified Petrov-Galerkin Method (MPGM)} We propose an alternative formulation of the test function space by extending the test function space with one additional basis function $1$. As a consequence, the dimension of the test function space is larger than that of the trial function space, which results in an overdetermined system, and we solve such a system using the Least-Squares method. 

More precisely,  for the semi-discrete method \eqref{eq:variational_form2} or its fully discrete version \eqref{eq:approximate_solution3}, constants are added to  form an augmented test function space
\begin{equation}
    \mathrm{V}_N=\mathrm{W}_N +\{1\}.
\end{equation}
where $\mathrm{W}_N$ is trial function space defined in \eqref{eq:trial_function}, and we thus obtain the modified Petrov-Galerkin Method (MPGM for short).

Note that, the dimension of the test function space is higher than that of the trial function space by $1$. Multiplying \eqref{eq:linear_system2} from the left by $\hat{S}^T$ , the least square solution satisfies
\begin{equation}
    \hat{S}^T\hat{S}\partial_t \hat{\mathbf{P}}=\hat{S}^T(\hat{A}+\hat{C})\hat{\mathbf{P}}.
\end{equation}
The matrix $\hat{K}$ in \eqref{eq:linear_system3} can be written as
\begin{equation}
    \hat{K}=(\hat{S}^T\hat{S})^{-1}(\hat{S}^T(\hat{A}+\hat{C})).
\end{equation}
The numerical solution is not completely mass-conserving due to the use of the least-squares method. But compared with the LGM, the mass of the numerical solution of the MPGM changes very little over time, as shown in Figure \ref{mass}. With extensive tests, the matrix $\hat{K}$ has no positive eigenvalues, and the MPGM is numerically stable.
\begin{figure}[!htb]
    \centering
        \begin{minipage}[c]{0.49\textwidth}
            \centering
            \includegraphics[width=7cm]{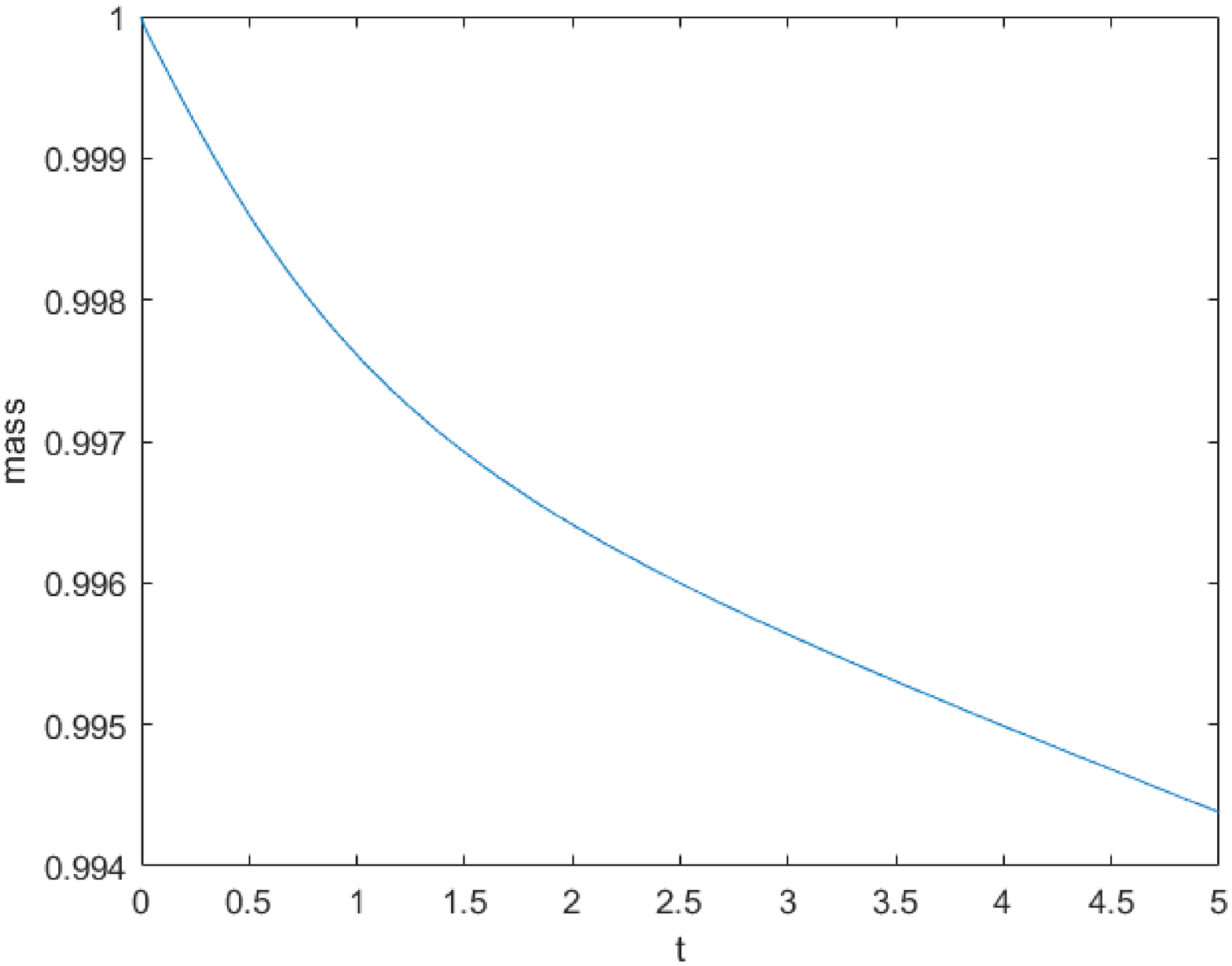}
        \end{minipage}
        \begin{minipage}[c]{0.49\textwidth}
            \centering
            \includegraphics[width=7cm]{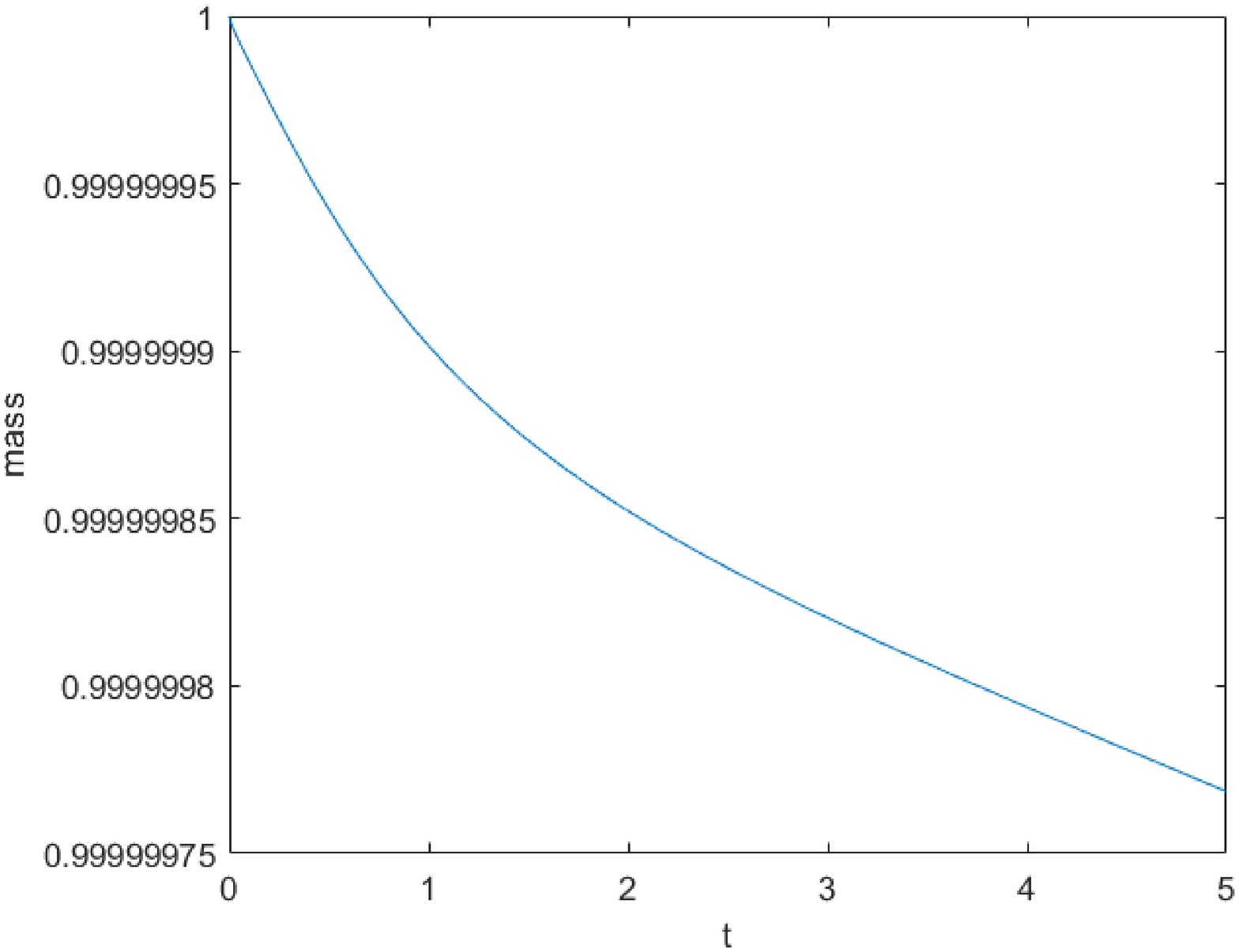}
        \end{minipage}
        \caption{Equation parameters $a=1, b=0$ with Gaussian initial condition $v_0=-1, \sigma_0^2=0.5$ and $N=10,\Delta t=0.001,T_\text{max}=5$. Left: Variation of total mass with time by the LGM. Right: Variation of total mass with time by the MPGM.}
        \label{mass}
\end{figure}

In conclusion, we have proposed two methods, i.e. the LGM and the MPGM, for the model problem \eqref{eq:problem1},  each possessing different advantages and therefore should be used in a flexible manner. The MPGM is preferred for simulating long-time behavior and testing the asymptotic preserving properties of the model, as the mass of the numerical solution from the LGM is significantly diminished in a long time. On the other hand, the LGM can be utilized for $\mathcal O(1)$ time simulations with verifiable order of convergence, and it does not involve the error due to the least square approximation.

\section{NNLIF with learning rules}\label{sec:lr}
In this Section, we consider the NNLIF model with a learning rule which is an extension of the Fokker-Planck equation \eqref{eq:problem1}, involving synaptic weights and the Hebbian learning rule. This is a novel and intriguing model and the dynamics of the membrane potential $v$ and the synaptic weight $w$ are on different time scales, making numerical simulation far more challenging. In order to better understand this model and verify the generality of the method proposed in Section \ref{sec:scheme}, we further explore this model from a numerical perspective. 

\subsection{Model introduction}

Compared with the simplest form of NNLIF model, the NNLIF model with learning rules introduces a new variable, the synaptic weight $w$, which is also the connectivity of the network $b$ mentioned in \eqref{eq:ha}. Furthermore, an external input function $I(w,t)$ is added to the drift coefficient $h$
\begin{equation}
    h(w,N(t))=-v+I(w,t)+w\sigma(N(t)).
\end{equation}
The function $\sigma(\cdot)$ represents the response of the network to the total activity, usually taking $\sigma(N)=N$. Then the Fokker-Planck equation without learning rules can be written as 
\begin{equation}
    \label{eq:problem3}
    \begin{cases}
        \partial_{t}p+\partial_{v}((-v+I(w,t)+w\sigma(\bar{N}(t)))p)-a\partial_{v v}p=0,\qquad v\in(-\infty,V_F]/\{V_R\},\\
        p(v,w,0)=p^0(v,w),\qquad p(-\infty,w,t)=p(V_F,w,t)=0,\\
        p(V^-_R,w,t)=p(V^+_R,w,t),\quad \partial _vp(V^-_R,w,t)=\partial _vp(V^+_R,w,t)+\frac{N(w,t)}{a},\\
    \end{cases}
\end{equation}
where, $p(v,w,t)$ describes the probability of finding a neuron at voltage $v$, synaptic weight $w$ and given time $t$. The diffusion coefficient $a$ is the same as in \eqref{eq:ha}. The subnetwork activity $N(w, t)$ and total activity
$\bar{N}(t)$ are defined as
\begin{equation}
    N(w,t)=-a\frac{\partial p}{\partial v}(V_F,w,t)\geq 0,\quad \bar{N}(t)=\int_{-\infty}^{+\infty}N(w,t)dw.
\end{equation}
Then we define the probability density of finding a neuron at synaptic weight $w$ and given time $t$ by 
\begin{equation}
    H(w, t)=\int_{-\infty}^{V_{F}} p(v, w, t) d v, \quad \int_{-\infty}^{\infty} H(w, t) d w=1 .
\end{equation}

In this case of no learning rule, the function $H(w,t)$ is time-independent because the distribution of synaptic weights in \eqref{eq:problem3} is fixed. The input signal $I(w,t)$ can be reflected by an output signal related to network activity $N(w,t)$. Next, we employ the learning rules of \cite{perthame2017distributed} to modulate the distribution of synaptic weights $H$, enabling the network to discriminate specific input signals $I$ by choosing an apposite synaptic weight distribution $H$ that is adapted to the signal $I$.

In \cite{perthame2017distributed}, the authors choose learning rules inspired by the seminal Hebbian rule and assume synaptic weights described with a single parameter $w$ and the subnetworks interact only via the total rate $\bar{N}$. They elucidate that all subnetworks parameterized by $w$ can vary their intrinsic synaptic weights $w$ according to a function $\Phi$ that is based on the intrinsic activity $N(w)$ of the network and the total activity of the network $\bar{N}$. Then, they give the generalization choice of Hebbian rule
\begin{equation}
    \Phi(N(w), \bar{N})=\bar{N} N(w) K(w),
\end{equation}
where $K(\cdot)$ represents the learning strength of the subnetwork with synaptic weight $w$. Adding the above choice of learning rule, the Fokker-Planck equation with learning rules is given by
\begin{equation}
    \label{eq:problem40}
    \frac{\partial p}{\partial t}+\frac{\partial}{\partial v}[(-v+I(w,t)+w \sigma(\bar{N}(t))) p]+\varepsilon \frac{\partial}{\partial w}[(\Phi-w) p]-a \frac{\partial^{2} p}{\partial v^{2}}=N(w, t) \delta\left(v-V_{R}\right).
\end{equation}
In order to better apply the numerical scheme and study the learning behavior of the model, we consider the equation \eqref{eq:problem40} for time rescaling $t \rightarrow t / \varepsilon$ and convert $\delta$-function to dynamic boundary condition such as:
\begin{equation}
    \label{eq:problem4}
    \begin{cases}
        \displaystyle
        \frac{\partial p}{\partial t}+\frac{\partial}{\partial w}[(\bar{N}(t)N(w,t)K(w)-w)p]
        =\frac{1}{\varepsilon}\left\{a\frac{\partial^2p}{\partial v^2}-\frac{\partial}{\partial v}[(-v+I(w,t)+w\sigma(\bar{N}(t)))p]\right\},\\
        p(v,w,0)=p^0(v,w),p(V_F,w,t)=p(-\infty,w,t)=p(v,\pm \infty,t)=0,\\
        p(V_R^-,w,t)=p(V_R^+,w,t),\qquad \frac{\partial}{\partial v}p(V^-_R,w,t)=\frac{\partial}{\partial v}p(V^+_R,w,t)+\frac{N(w,t)}{a}.
    \end{cases}
\end{equation}

Here, $p^0(v,w)$ is initial condition and the probability density
function p(v, t) should satisfy the condition of conservation of mass
\begin{equation}
    \int_{-\infty}^{\infty} \int_{-\infty}^{V_{F}} p(v, w, t) d v d w=\int_{-\infty}^{\infty} \int_{-\infty}^{V_{F}} p^{0}(v, w) d v d w=1.
\end{equation}

Despite some research on model \eqref{eq:problem4} as indicated by the theoretical properties presented in \cite{perthame2017distributed} and the numerical analysis and experiments in \cite{he2022structure}, it is still a relatively new model with limited established knowledge. In this paper, the numerical method proposed in Section \ref{sec:scheme} is used to further investigate the learning behaviors of this model numerically.

\subsection{Numerical scheme}
Now, we describe the numerical scheme for \eqref{eq:problem4}. We choose the calculation interval as $[V_{\text{min}},V_F]\times [W_{\text{min}},W_{\text{max}}]\times [0,T_{\text{max}}]$ and suppose the density function is practically negligible out of this region.
We use spectral methods for v-wise discretization and Differential method for w-wise and t-wise discretization. So we divide the interval $[W_{\text{min}},W_{\text{max}}],[0,T_{\text{max}}]$ into $n_w,n_t$ equal sub-intervals with size
\begin{equation}
    \Delta w=\frac{W_{\text{max}}-W_{\text{min}}}{n_w},\Delta t=\frac{T_{\text{max}}}{n_t}.
\end{equation}
Then the grid points can be represented as follows
\begin{equation}
    \begin{aligned}
        &w_{j}=W_{\text{min} }+j \Delta w, & j=0,1,2, \cdots, n_{w} \\
        &t^{n}=n \Delta t, & n=0,1,2, \cdots, n_{t}
    \end{aligned}
\end{equation}
For the v-direction discretization, we take the same scheme as in Section \ref{sec:fully_discrete_scheme}. The approximate solution is expended as 
\begin{equation}
    \label{eq:approximate_solution4}
    p_N(v,w,t)=\sum_{k=1}^{2N+3}\hat{u}_k(w,t)\psi_k(v).
\end{equation}
The initial condition for the expansion coefficients $\{\hat{u}_{k}(w,0)\}_{k=0}^{2N+3}$ can be obtained by the least square approximation,
\begin{equation}
    \label{eq:initial_vector3}
    \int_{V_{\min}}^{V_F} \sum_{k=1}^{2N+3}\hat{u}_k(w_j,0)\psi_k(v) \phi_i dv=\int_{V_{\min}}^{V_F} p^0(w_j,v)\phi_i dv, \quad j=0,1,2, \cdots, n_{w} \quad \forall \phi_i \in \mathrm{V}_N.
\end{equation}
From the properties of the basis functions \eqref{eq:lambda}, subnetwork activity $N(w,t)$ can be expressed as
\begin{equation}
    N^n_j=N(w_j,t^n)=-a\hat{u}_{2N+3}(w_j,t^n).
\end{equation}
And we apply the simplest rectangular numerical integration rule to discretize the total activity $\bar{N}(t)$
\begin{equation}
    \bar{N}^n=\Delta w \sum_{j=0}^{n_w}N^n_j.
\end{equation}
For the w-direction discretization, we inherit the idea form \cite{he2022structure} which takes the following explicit flux construction adapted from Godunov's Method 
 \begin{equation}
    \Phi_{i, j+\frac{1}{2}}^{n}=
        \begin{cases}
            \begin{cases}
                \min \left\{\Phi_{i, j}^{n}, \Phi_{i, j+1}^{n}\right\} \qquad &\hat{P}_{i, j}^{n} \leq \hat{P}_{i, j+1}^{n} \\
                \max \left\{\Phi_{i, j}^{n}, \Phi_{i, j+1}^{n}\right\}  &\hat{P}_{i, j}^{n}>\hat{P}_{i, j+1}^{n} \\
            \end{cases}&j=0, \cdots, n_{w}-1\\
            0  &j=-1, n_{w}
        \end{cases}
    \end{equation}
where
\begin{equation}
    \Phi_{i, j}^{n}=\left(\bar{N}^{n} N_{j}^{n} K\left(w_{j}\right)-w_{j}\right) \hat{P}_{i, j}^{n} \quad \text { for } \quad j=0, \cdots, n_{w}.
\end{equation}
$\hat{P}_{i,j}^n$ is the coefficients of the basis functions in \eqref{eq:approximate_solution4}
\begin{equation}
    \hat{P}_{i,j}^n=\hat{u}_i(w_j,t^n).
\end{equation}
Define 
\begin{equation}
\begin{aligned}
    &p_{N,j}^{n}=\sum_{k=1}^{2N+3}\hat{u}_k(w_j,t^n)\psi_k(v),\\
    &q_{N,j+\frac{1}{2}}^n=\sum_{k=1}^{2N+3} \Phi_{k,j+\frac{1}{2}}^{n}\psi_k(v).
\end{aligned}
\end{equation}
After using a semi-implicit method for time discretization, we obtain the fully discrete scheme as follows:
\begin{equation}
    \frac{p_{N,j}^{n+1}-p_{N,j}^{n}}{\Delta t}+\frac{q_{N,j+\frac{1}{2}}^n-q_{N,j-\frac{1}{2}}^n}{\Delta w}=\frac{1}{\varepsilon}\left\{a\frac{\partial^2p_{N,j}^{n+1}}{\partial v^2}-\frac{\partial}{\partial v}\left[(-v+I(w_j)+w_j\sigma(\bar{N}(t^n)))p_{N,j}^{n+1}\right]\right\}.
\end{equation}
When the test function space $\mathrm{V}_N$ is given, the coefficients of the approximate solution \eqref{eq:approximate_solution4} for each $t$ and $w$ step can be obtained by the following linear system
\begin{equation}
    \begin{aligned}
        &\frac{\hat{S}(\hat{\mathbf{P}}^{n+1}_{j}-\hat{\mathbf{P}}^n_{j})}{\Delta t}+\frac{\hat{S}(\mathbf{\Phi}_{j+\frac{1}{2}}^n-\mathbf{\Phi}_{j-\frac{1}{2}}^n)}{\Delta w}\\
        +&\frac{1}{\varepsilon}\left\{-\hat{A}\hat{\mathbf{P}}^{n+1}_{j}+\left(I(w_{j},t^n)+w_{j}\sigma(\bar{N}(t^n))\right)\hat{B}\hat{\mathbf{P}}^{n+1}_{j}-a\hat{C}\hat{\mathbf{P}}^{n+1}_{j}\right\}=0,
    \end{aligned}
\end{equation}
where 
\begin{equation}
    \begin{aligned}
    &\hat{\mathbf{P}}^n_j=\left(\hat{u}_1(w_j,t^n),\hat{u}_2(w_j,t^n),...,\hat{u}_{2N+3}(w_j,t^n)\right)^T,\\
        &\mathbf{\Phi}_{j+\frac{1}{2}}^n=(\Phi_{1, j+\frac{1}{2}}^{n},\Phi_{2, j+\frac{1}{2}}^{n},...,\Phi_{2N+3, j+\frac{1}{2}}^{n})^T,
    \end{aligned}
\end{equation}
 and the matrix $\hat{S},\hat{A},\hat{B},\hat{C}$ are defined in \eqref{eq:Matrix2}.

This numerical scheme is conserved naturally in the $w$ direction, however, strict conservation of mass in the $v$ direction is not achieved when the test function space is selected based on Section \ref{sec:stability}. When $\varepsilon$ is small enough, the asymptotic preserving properties of the model can only be verified through the use of MPGM.

\section{Numerical test} \label{sec:numerical_test}

In this section, we give  numerical tests to verify the properties of the proposed schemes and demonstrate some explorations of the model. Numerical solutions for the initial three subsections are obtained by LGM; results for the MPGM approach are similar except for Section \ref{sec:Convergence}, which are thus omitted, and numerical solutions for Section \ref{sec:Learning_testing} are obtained by MPGM, as variations in the time scale require the scheme to be asymptotic preserving.

The tests are structured as follows. In Section \ref{sec:Convergence}, the convergence order of the method is tested in both the NNLIF model and the NNLIF model with learning rules. In Section \ref{sec:time_saving}, we validate the efficiency of the spectral method by comparing it to existing methods. In Section \ref{sec:blow_up}, we test a few properties of the NNLIF model. In Section \ref{sec:Learning_testing}, we test the learning and discrimination abilities of NNLIF model with learning rules for the periodic input function.

\subsection{Order of accuracy}\label{sec:Convergence}
In this part, we test the order of accuracy of the proposed scheme based on the NNLIF model and the NNLIF model with learning rules. Since the exact solution is unavailable, we choose the numerical solution $p_e$ of the finite difference method \cite{hu2021structure} with sufficient accuracy to replace the exact solution.

For NNLIF model \eqref{eq:problem2}, we choose $V_F=2,V_R=1,V_{\text{min}}=-4,a=1, b=3$ and the Gaussian distribution
\begin{equation}
    p_G(v)=\frac{1}{\sqrt{2 \pi} \sigma_{0} M_0} e^{-\frac{\left(v-v_{0}\right)^{2}}{2 \sigma_{0}^{2}}},
\end{equation}
as the initial condition with $v_0=-1$ and $\sigma_0^2=0.5$, $M_0$ is a normalization factor such that
\begin{equation}
    \int_{V_{\text{min}}}^{V_F} p_G(v) dv=1.
\end{equation}
 The numerical solution is computed till time $t=0.2$. Errors in both $L^{\infty}$ and $L^2$ norm are examined with fixed $N=12$ and different $\Delta t$ in Table \ref{convergence1}. It should be noted that the number of basis functions is not $N$, but rather $2N+3$, as shown in equation \eqref{eq:approximate_solution3}. 
\begin{table}[!htb]
	\centering
	\begin{tabularx}{10cm}{ccccc}
	\toprule
	$\Delta t$ & $\left\| p_N-p_e \right\|_{L^{\infty}}$& {$O_{\tau,L^{\infty}}$}& $\left\| p_N-p_e \right\|_{L^2}$& {$O_{\tau,L^2}$}\\ 
	\midrule
	0.04 & 3.880E-03 &0.9520 & 1.868E-03 &0.9508 \\
	0.02 & 2.005E-03 & 0.9792 & 9.662E-04 &0.9617 \\
	0.01 & 1.017E-03 & 0.9926 &4.961E-06&0.9287 \\
    0.005 & 5.111E-04 & - &2.606E-04& - \\
	\bottomrule
    \end{tabularx}
    \caption{Error and order of accuracy of the proposed numerical scheme for NNLIF model with different temporal sizes. The parameter $N$ is fixed as $N=12$.}
    \label{convergence1}
\end{table}

For the order of accuracy in the $v$ direction, we choose the time step size $\Delta t=10^{-5}$. Errors in the $L^2$ norm are examined with different $N$. The logarithm of the error versus $N$ is plotted in Figure \ref{convergence2}. We remark that when testing the order of spatial convergence, the results present a zig-zag decreasing profile as $N$ increases, which is a common phenomenon for spectral methods. We thus plot the errors for odd and even numbers of $N$, respectively. For each scenario, we clearly observe the spectral convergence as the number of spatial basis functions increases.

\begin{figure}[!htb]
    \centering
        \begin{minipage}[c]{0.49\textwidth}
            \centering
            \includegraphics[width=7cm]{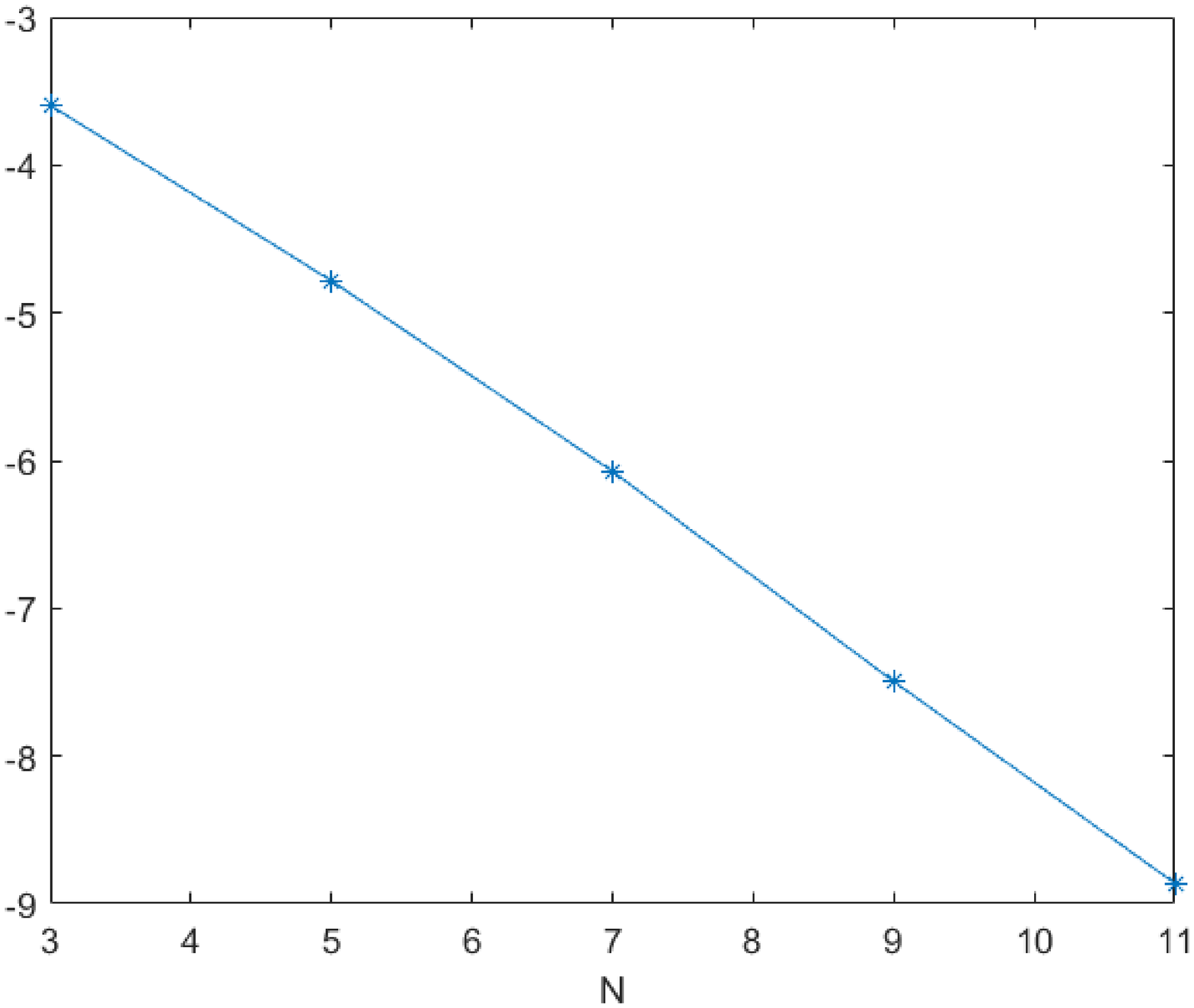}
        \end{minipage}
        \begin{minipage}[c]{0.49\textwidth}
            \centering
            \includegraphics[width=7cm]{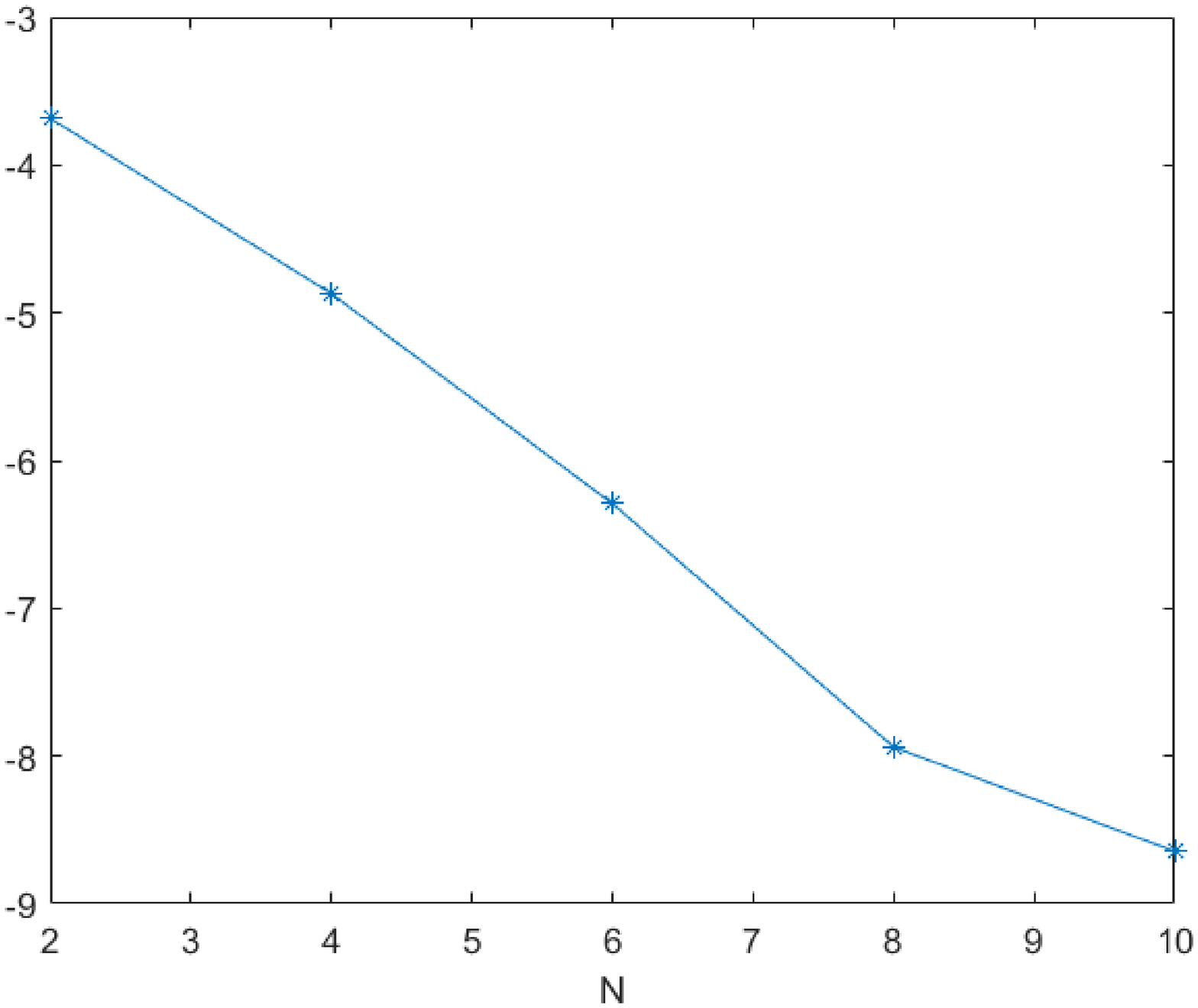}
        \end{minipage}
		\caption{Logarithm of the error of the proposed numerical scheme for NNLIF model with learning rules with different $N$. The temporal size is fixed as $\Delta t=10^{-5}$. Left: $N$ is odd; Right: $N$ is even.}
  \label{convergence2}
\end{figure}

For NNLIF with learning rules model \eqref{eq:problem4}, we choose $V_F=2,V_R=1,V_{\text{min}}=-4,a=1,\varepsilon=0.5,W_{\text{min}}=-1.1,W_{\text{max}}=0.1,\sigma(\bar{N})=\bar{N}, I(w)=0$ and initial condition
\begin{equation}
    p_{\text{init}}=\begin{cases}
        \frac{1}{\sqrt{2 \pi} \sigma_{0} } e^{-\frac{\left(v-v_{0}\right)^{2}}{2 \sigma_{0}^{2}}}\text{sin}^2(\pi w) \qquad &-1<w<0,\\
        0 &\text{otherwise},
    \end{cases}
\end{equation}
with $v_0=-1$ and $\sigma_0^2=0.5$.

The numerical solution is computed till time $t=0.1$. For $t$ direction and $w$ direction, we fix $\frac{\Delta w}{\Delta t}=1, N=16$. Considering that both the $t$ direction and the $w$ direction are theoretically first-order accurate, as well as the stability factor, it is reasonable to jointly test the order of accuracy. Errors in both $L^{1}$ and $L^2$ norm are examined with different $\Delta t$ and $\Delta w$ in Table \ref{convergence3}. For $v$ direction, we fix ${\Delta w}={\Delta t}=10^{-5}$. Errors in the $L^2$ norm are examined with different $N$. The logarithm of the error versus $N$ is plotted in Figure \ref{convergence4}.

\begin{table}[!htb]
	\centering
	\begin{tabularx}{10cm}{cccccc}
	\toprule
	$\Delta t$ &$\Delta w$ & $\left\| p_N-p_e \right\|_{L^{1}}$& $O_{\tau,L^{1}}$& $\left\| p_N-p_e \right\|_{L^2}$& {$O_{\tau,L^2}$}\\ 
	\midrule
	0.02 &0.02 & 1.599E-03 &1.04 & 3.234E-03 &1.07 \\
	0.01 &0.01 & 7.755E-04 & 0.99 & 1.536E-03 &0.97 \\
	0.005 &0.005 &3.893E-04 & 1.01 &7.812E-04&1.05 \\
    0.0025 & 0.0025 & 1.926E-04 & - &3.757E-04& - \\
	\bottomrule
    \end{tabularx}
    \caption{Error and order of accuracy of the proposed numerical scheme for NNLIF model with learning rules with different ${\Delta w}$ and ${\Delta t}$. The parameter $N$ is fixed as $N=16$.}
    \label{convergence3}
\end{table}
\begin{figure}[!htb]
    \centering
        \begin{minipage}[c]{0.49\textwidth}
            \centering
            \includegraphics[width=7cm]{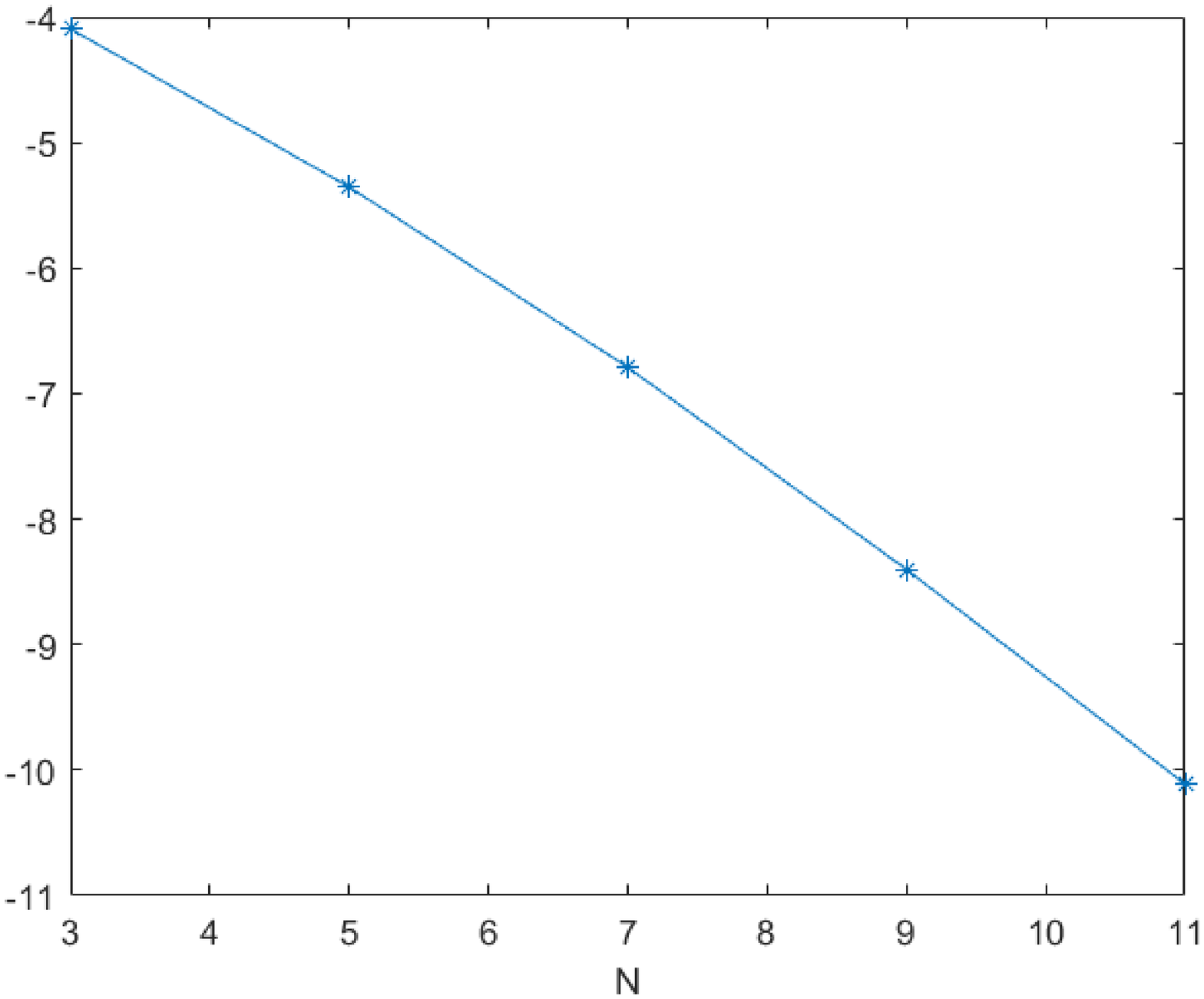}
        \end{minipage}
        \begin{minipage}[c]{0.49\textwidth}
            \centering
            \includegraphics[width=7cm]{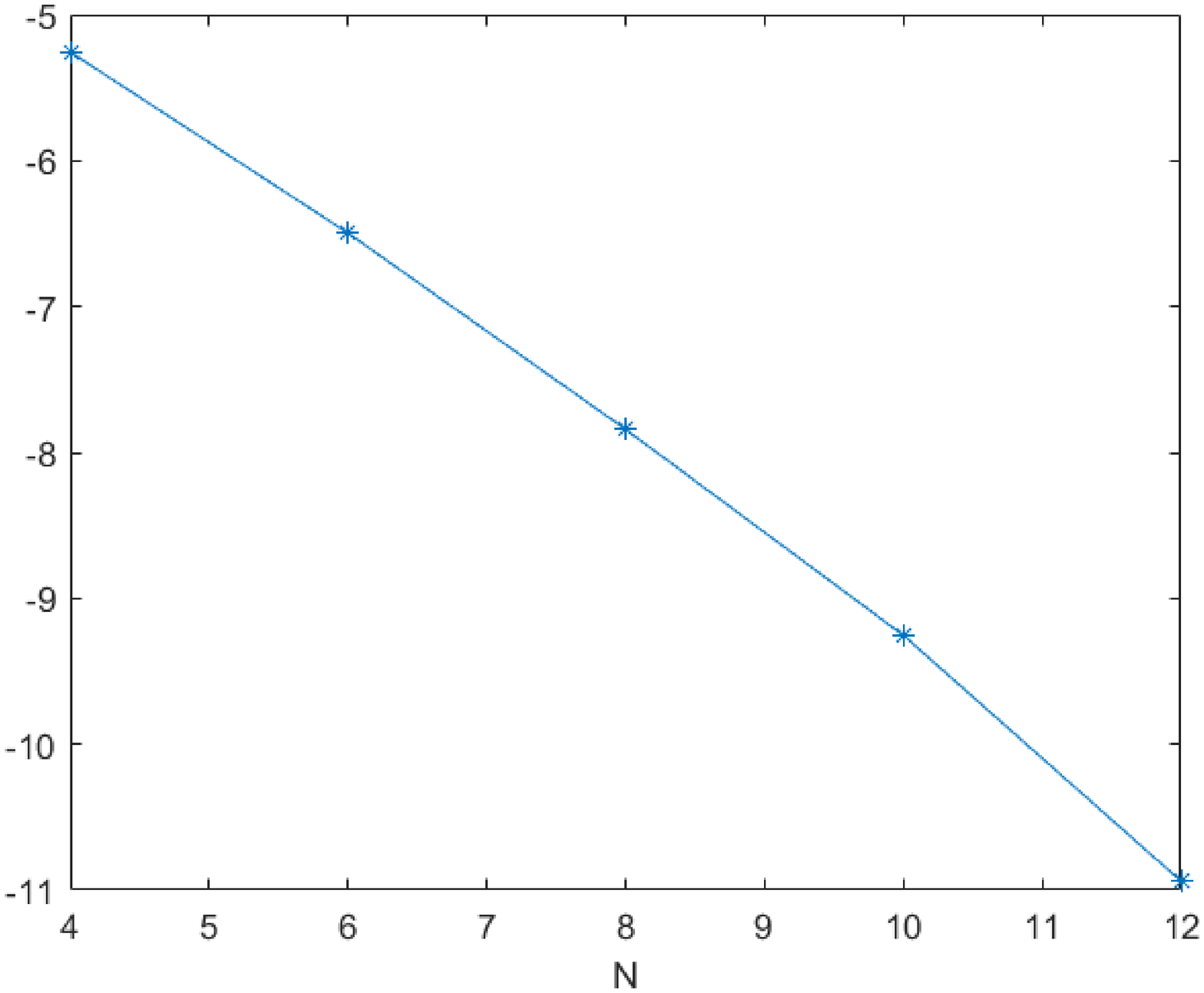}
        \end{minipage}
		\caption{Logarithm of the error of the proposed numerical scheme for NNLIF model with learning rules with different $N$. The temporal size is fixed as $\Delta t=10^{-5}$. Left: $N$ is odd; Right: $N$ is even}
  \label{convergence4}
\end{figure}

The results indicate that the scheme shows first-order accuracy in time and exponential convergence in space for the NNLIF model; first-order accuracy in the $w$, $t$ direction and exponential convergence in the v direction for the NNLIF model with learning rules.

\subsection{Simulation time comparison}\label{sec:time_saving}
In this part, we compare the CPU time between the proposed spectral method and the finite difference method \cite{hu2021structure}, to show that our scheme has a significant computational time advantage with the same level of accuracy.

We choose NNLIF model with parameters $a=1,b=0.5,\Delta t=5\times 10^{-4}$ and the Gaussian initial condition with $v_0=0,\sigma_0^2=0.25$. The numerical solution is computed till time $t=0.5$. The results of the spectral method and the finite difference method are shown in Table \ref{error1} and Table \ref{error2}.
\begin{table}[!htb]
	\centering
	\begin{tabularx}{8cm}{cccc}
	\toprule
	$N$& $\left\| \cdot \right\|_{\infty}$&$\left\| \cdot \right\|_{1}$& CPU Time (s) \\ 
	\midrule
	5 & 5.15e-02 & 1.58e-02& 0.026 \\
    10 & 3.33e-03 & 2.85e-04 & 0.030  \\
    15 & 9.71e-05 & 2.51e-05& 0.053 \\
    20 & 1.30e-06 & 3.76e-07 & 0.071  \\
	\bottomrule
    \end{tabularx}
    \caption{Errors using the spectral method with different numbers of basis functions.}
    \label{error1}
\end{table}

\begin{table}[!htb]
	\centering
	\begin{tabularx}{8cm}{cccc}
	\toprule
	$h$& $\left\| \cdot \right\|_{\infty}$&$\left\| \cdot \right\|_{1}$& CPU Time (s) \\ 
	\midrule
    ${1/4}$ & 3.01e-03 & 7.09e-04& 0.031 \\
    ${1/8}$ & 9.69e-04 & 2.18e-04& 0.073 \\
    ${1/16}$ & 2.79e-04 & 6.18e-05 & 0.157  \\
    ${1/32}$ & 7.54e-05 & 1.64e-05& 0.348 \\
    ${1/64}$ & 1.97e-05 & 4.21e-06 & 2.801  \\
    ${1/128}$ & 4.41e-06 & 1.12e-06 & 11.971  \\
	\bottomrule
    \end{tabularx}
    \caption{Errors using the finite difference method  with different spatial grid sizes}
    \label{error2}
\end{table}
These tables clearly indicate that to achieve the same level of accuracy, the spectral method is more efficient in terms of the simulation time, and the advantage is more noticeable when the accuracy level is higher.

\subsection{Global solution and blow-up in NNLIF model}\label{sec:blow_up}
\subsubsection{Blow up}

In \cite{caceres2011analysis}, the authors find the solution may blow up in finite time with the suitable initial conditions for the excitatory network. They show that whenever the value of $b>0$ is, if the initial data is concentrated enough around $v=V_F$, then the defined weak solution in Definition 2.1 of \cite{caceres2011analysis} does not exist for all times. Figure \ref{fig:blowup1} and Figure \ref{fig:blowup2} show this phenomenon. It can be seen that when the blow-up phenomenon is about to occur, the density function $p(v,t)$ is increasingly concentrated and sharp at reset point $V_R$ and the firing rate $N(t)$ is growing rapidly. 
\begin{figure}[!htb]
    \centering
        \begin{minipage}[c]{0.49\textwidth}
            \centering
            \includegraphics[width=7cm]{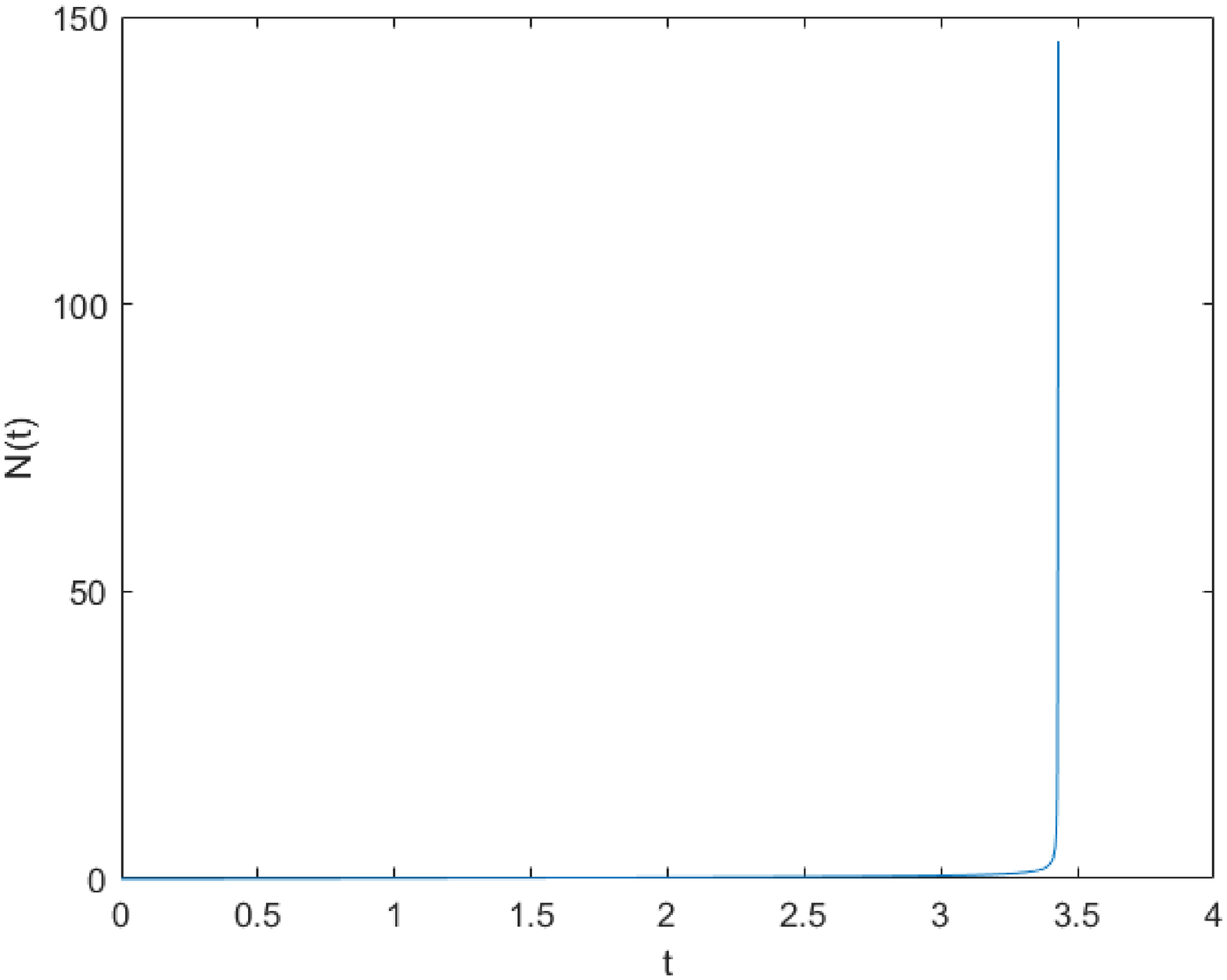}
        \end{minipage}
        \begin{minipage}[c]{0.49\textwidth}
            \centering
            \includegraphics[width=7cm]{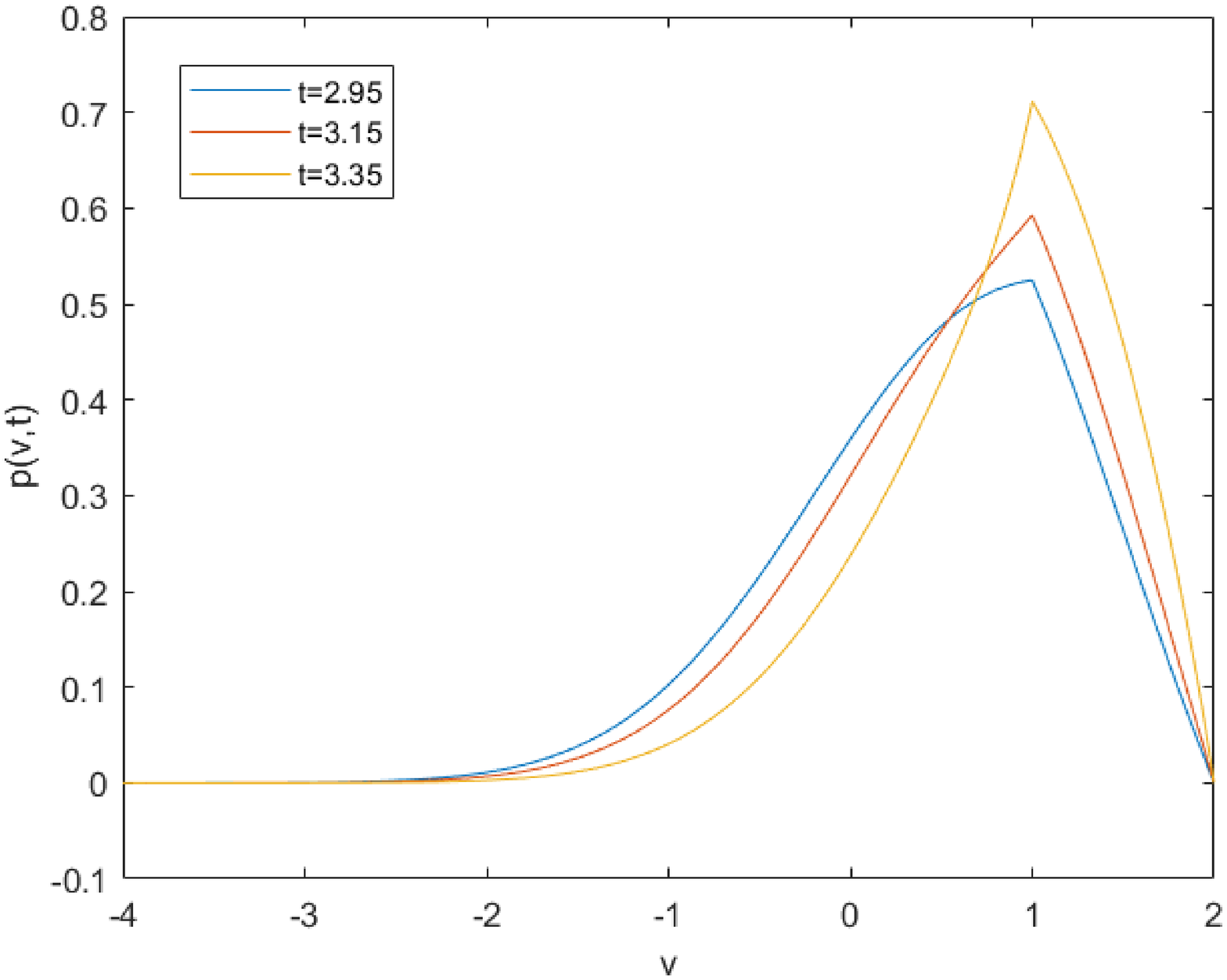}
        \end{minipage}
        \caption{Equation parameters $a=1, b=3$ with Gaussian initial condition $v_0=-1, \sigma_0^2=0.5$. Left: evolution of firing rate $N(t)$. Right: density function $p(v, t)$ at $t=2.95,3.15,3.35$.}
        \label{fig:blowup1}
\end{figure}
\begin{figure}[!htb]
        \begin{minipage}[c]{0.49\textwidth}
            \centering
            \includegraphics[width=7cm]{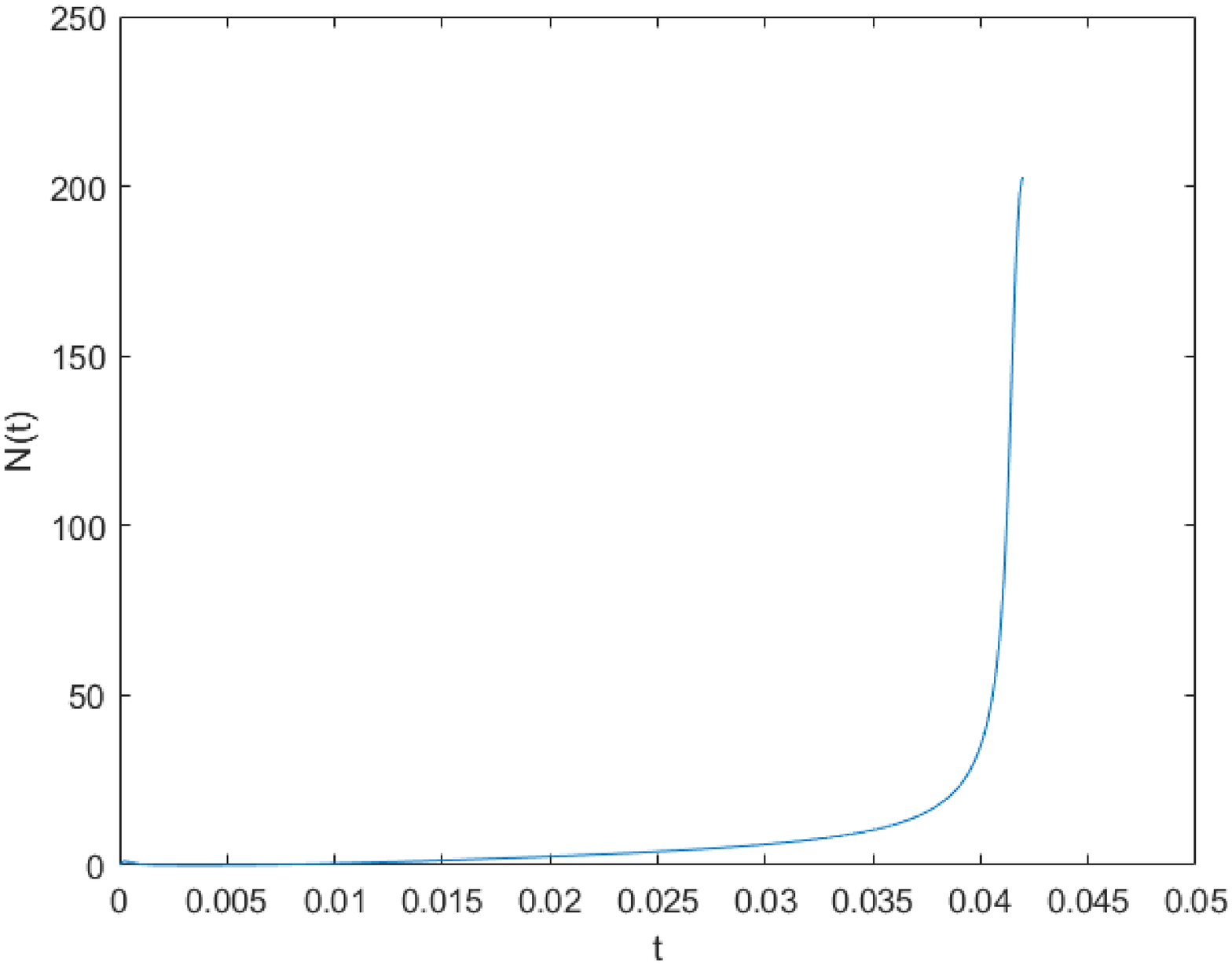}
        \end{minipage}
        \begin{minipage}[c]{0.49\textwidth}
            \centering
            \includegraphics[width=7cm]{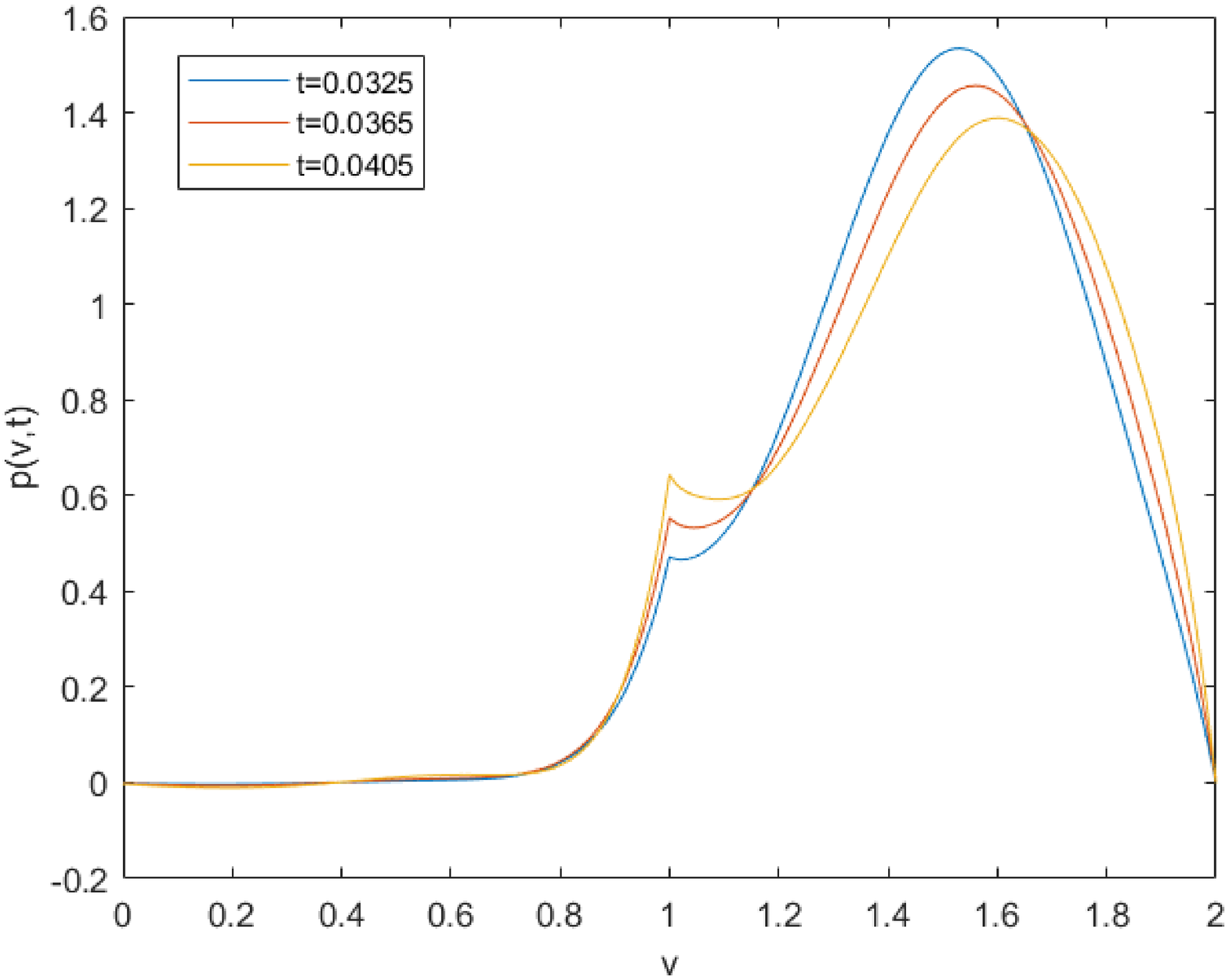}
        \end{minipage}
        \caption{Equation parameters $a=1, b=1.5$ with Gaussian initial condition $v_0=1.5, \sigma_0^2=0.005$.Left: evolution of firing rate $N(t)$. Right: density function $p(v, t)$ at $t=0.0325,0.0365,0.0405$.}
        \label{fig:blowup2}
\end{figure}

For spectral methods, the approximate solution of the density function is dependent on the coefficients of the basis functions. We aim to further investigate how the coefficients change when the blow-up phenomenon is about to occur. We choose $a=1, b=1.5$ in equation and $N=20,\Delta t=10^{-5}$.
\begin{figure}[!htb]
    \centering
        \begin{minipage}[c]{0.49\textwidth}
            \centering
            \includegraphics[width=7cm]{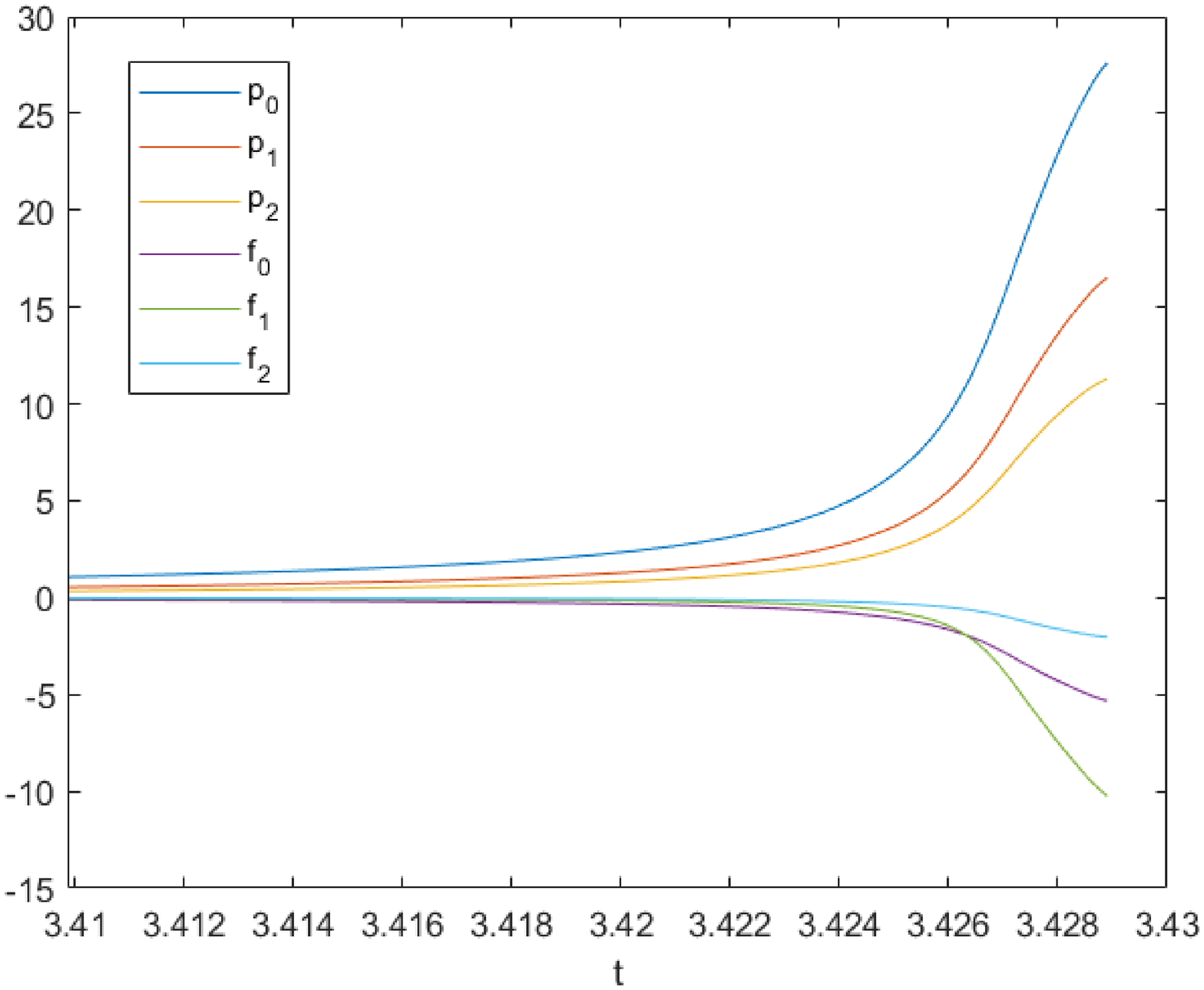}
        \end{minipage}
        \begin{minipage}[c]{0.49\textwidth}
            \centering
            \includegraphics[width=7cm]{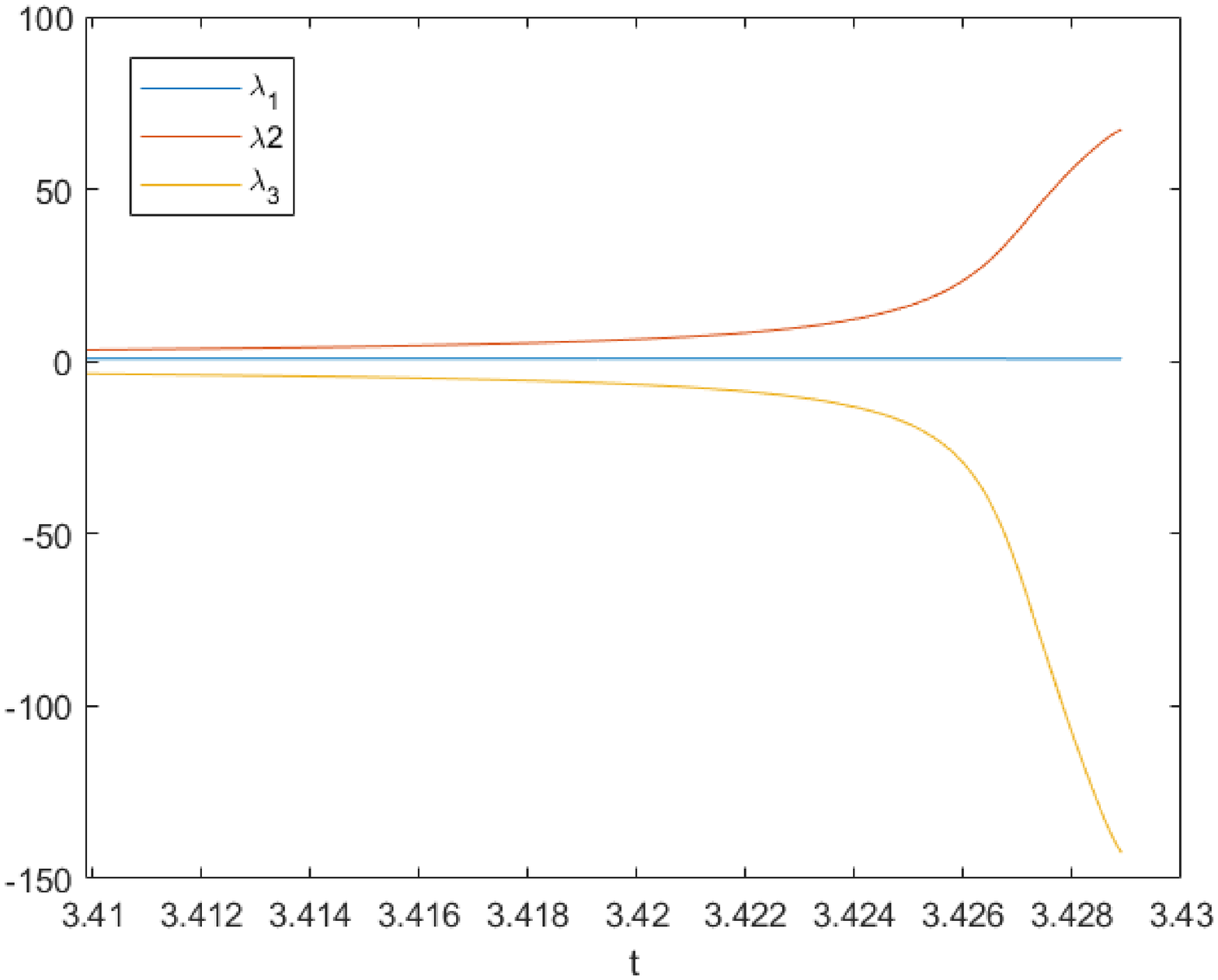}
        \end{minipage}
    \caption{Changes of the coefficients of the first few terms in the expansion formula \eqref{eq:approximate_solution3} during blow up. Left: evolution of the coefficients $\{p_k,f_k\}_{k=0}^2$. Right:evolution of the coefficients $\{\lambda_k\}_{k=1}^3$.  }
    \label{fig:coeff}
\end{figure}

Recall that
\begin{equation}
    \lambda_3(t)=\partial_vp(V_F,t)=-\frac{N(t)}{a},\qquad \partial_vp(V_R^+,t)=\lambda_2(t)+\lambda_3(t).
\end{equation}
Therefore, $\lambda_2$ and $\lambda_3$ are directly influenced by the firing rate. Due to the use of global basis functions, as the firing rate $N(t)$ increases, all the basis functions are affected. In response to the change of $\lambda_3$, $\lambda_2$ and the coefficients of the basis functions in $\mathrm{W}_2$ change accordingly, respectively controlling the derivative value on both sides of point $V_R$ and the function value in the interval. Figure \ref{fig:coeff} show the change of coefficients $\{p_k,f_k\}_{k=0}^2$, $\{\lambda_k\}_{k=1}^3$ in \eqref{eq:approximate_solution3} as time involves. It can be seen from the figure that the changes in $\lambda_2$ and $\lambda_3$ are most obvious, while the coefficients of all basis functions in $\mathrm{W}_2$ space are affected but the changes are relatively small.

\subsubsection{Relative entropy}

As we have mentioned, since little is known about the properties of the solutions of the Fokker-Planck equation \eqref{eq:problem1}, there  is a lack of complete understanding of the long-time asymptotic behavior in the continuous case. In \cite{caceres2011analysis}, they studied relative entropy theory for linear problem $a_1=b=0$, which implies exponential convergence to equilibrium. The relative entropy is given by
\begin{equation}
    I_e=\int_{-\infty}^{V_{F}} G\left(\frac{p(v, t)}{p^{\infty}(v)}\right) p_{\infty}(v) d v,
\end{equation}
which can be shown to be decreasing in time, where $G(\cdot)$ is a smooth convex function and $p^{\infty}(v)$ represents the stationary solution. In this part, we numerically verify the relative entropy theory. The numerical relative entropy is given by
\begin{equation}
    S(t)=\int_{V_L}^{V_{F}} G\left(\frac{p_N(v, t)}{p^{\infty}(v)}\right) p_{\infty}(v) d v.
\end{equation}

We consider nonlinear cases with $a_0=1,a_1=0,b=-0.5$ and $a_0=1,a_1=0.1,b=0$. We choose the numerical solution of a sufficiently long time as the stationary solution $p^{\infty}(v)$ and the Gaussian initial condition $v_0=-1, \sigma_0^2=0.5$. Figure \ref{fig:relative_entropy2} \ref{fig:relative_entropy3} show the time evolution of the firing rate and the numerical relative entropy for these cases.
\begin{figure}[!htb]
    \centering
    \begin{minipage}[c]{0.49\textwidth}
        \centering
        \includegraphics[width=1\textwidth]{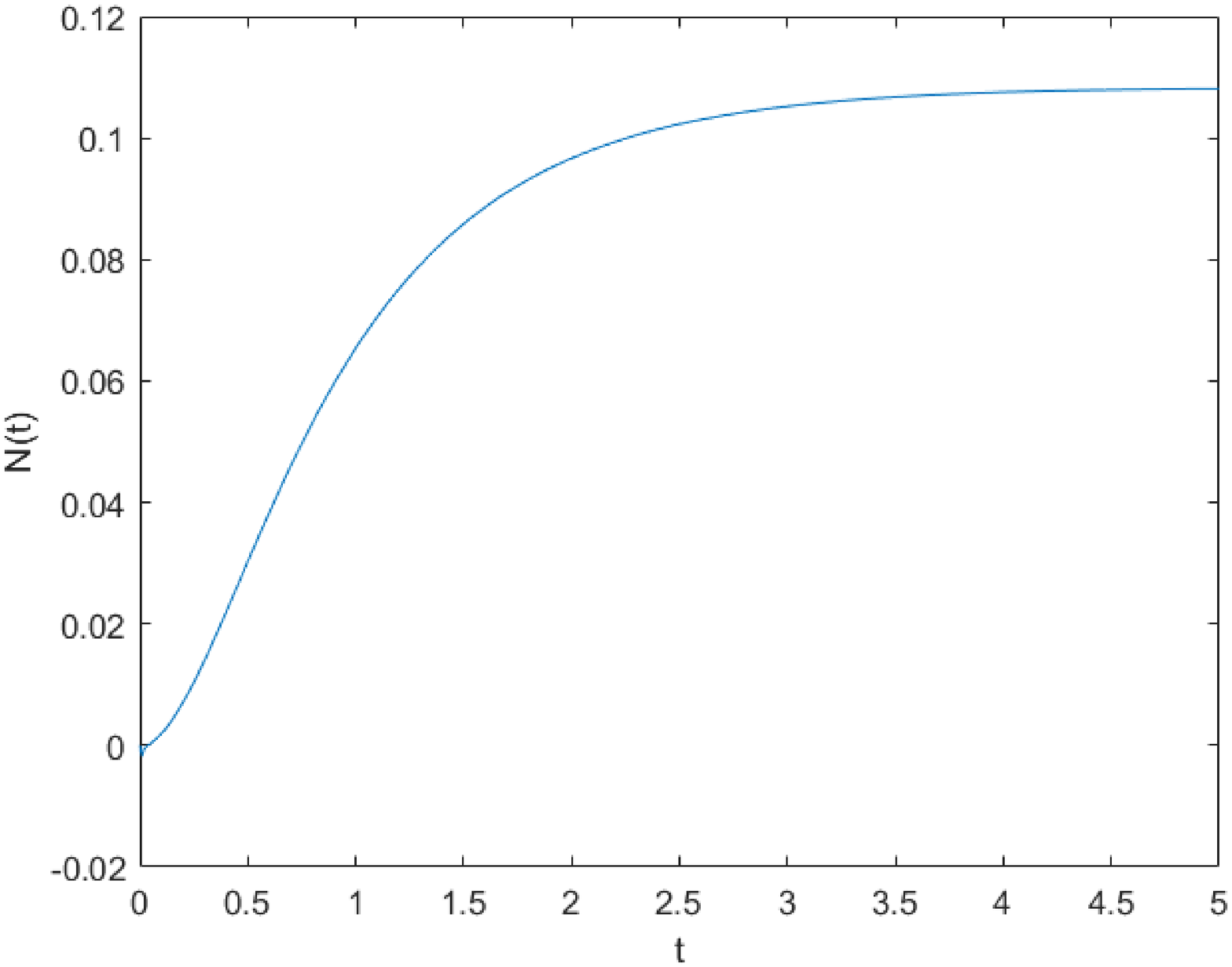}
    \end{minipage}
    \begin{minipage}[c]{0.49\textwidth}
        \centering
        \includegraphics[width=1\textwidth]{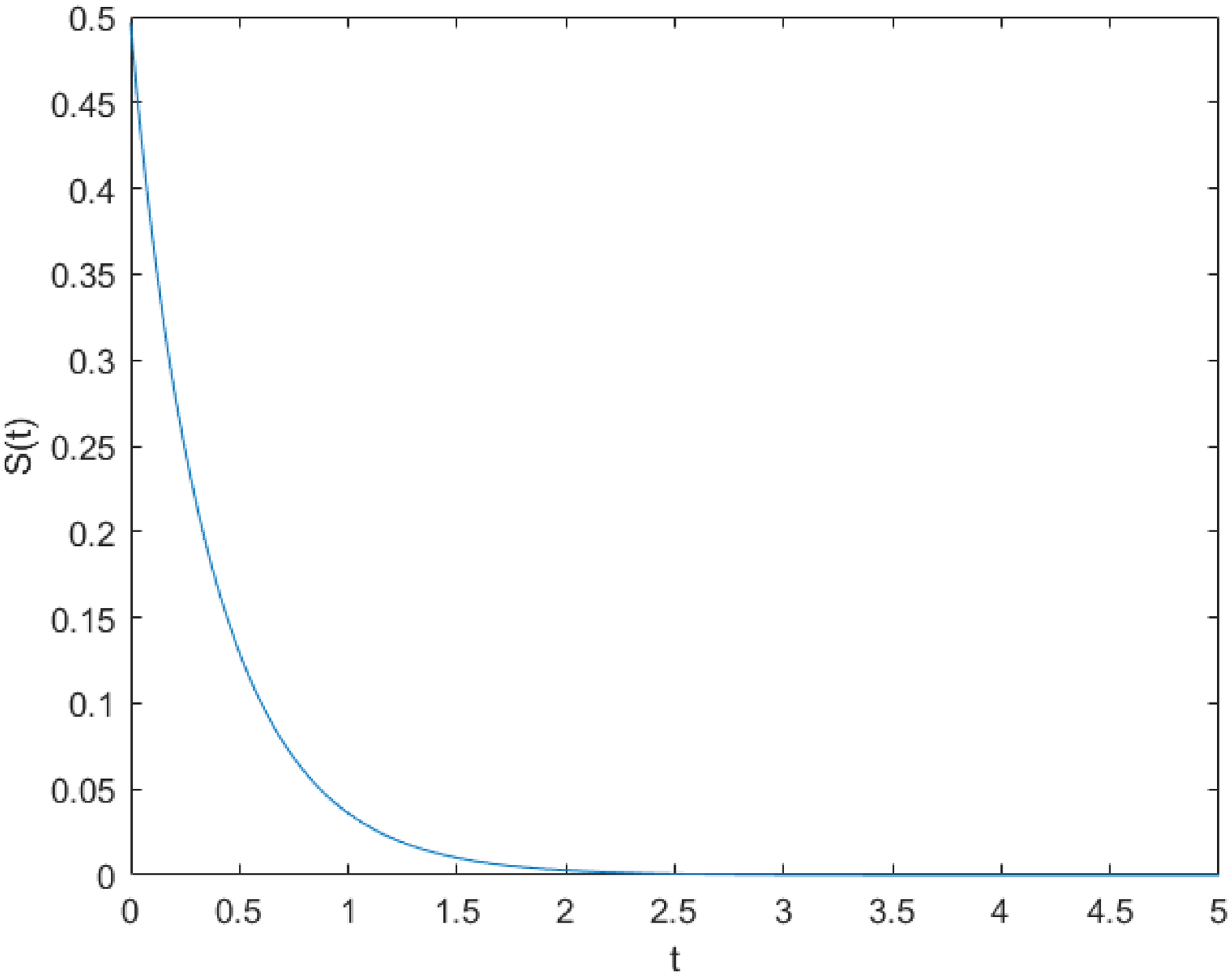}
    \end{minipage}
    \caption{Equation parameters $a=1,b=-0.5$ with Gaussian initial condition $v_0=-1, \sigma_0^2=0.5$. Left: evolution of firing rate $N(t)$. Right: evolution of relative entropy $S(t)$ with $G(x)=\frac{(x-1)^2}{2}$.}
     \label{fig:relative_entropy2}
\end{figure}
\begin{figure}[!htb]
    \centering
    \begin{minipage}[c]{0.49\textwidth}
        \centering
        \includegraphics[width=1\textwidth]{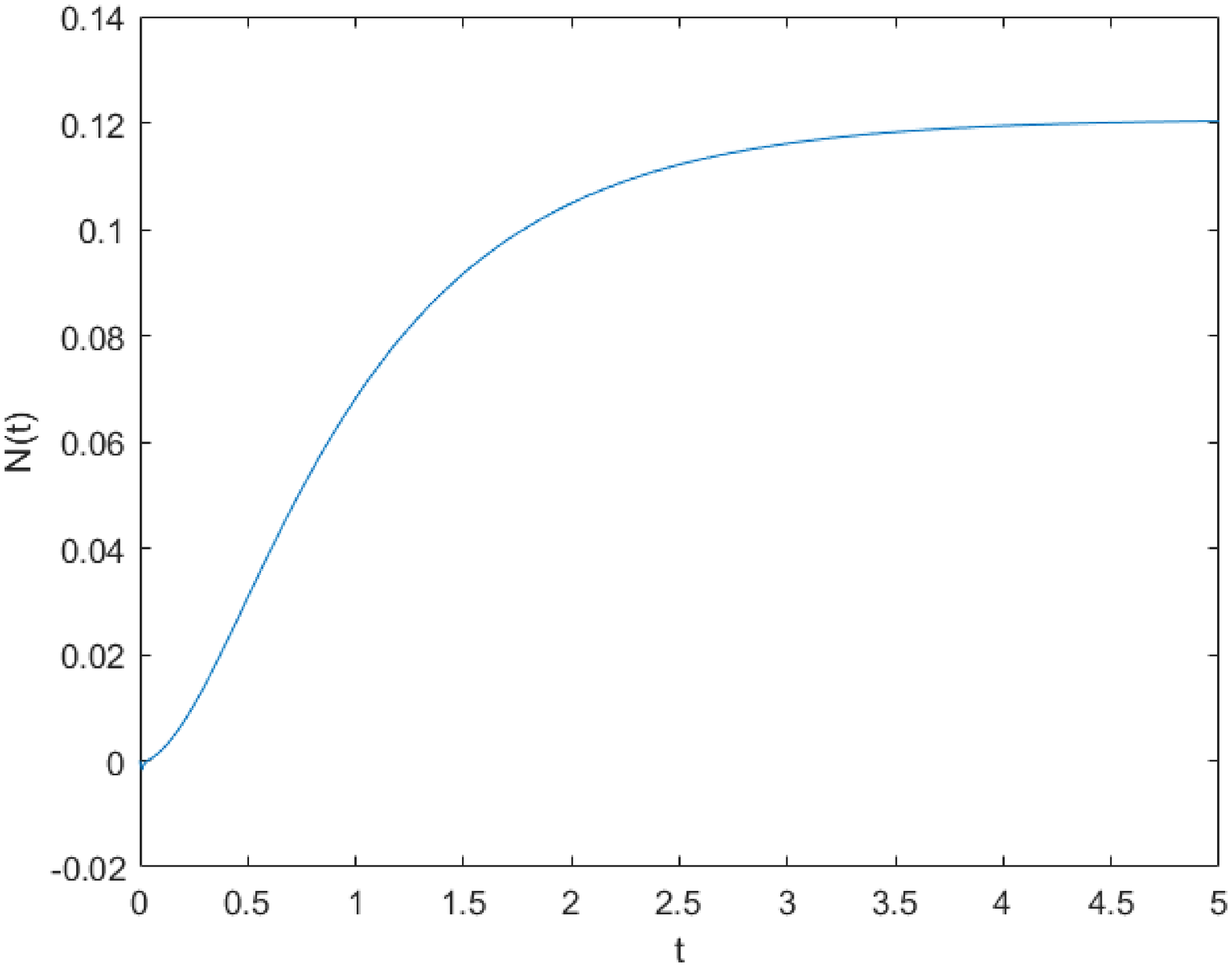}
    \end{minipage}
    \begin{minipage}[c]{0.49\textwidth}
        \centering
        \includegraphics[width=1\textwidth]{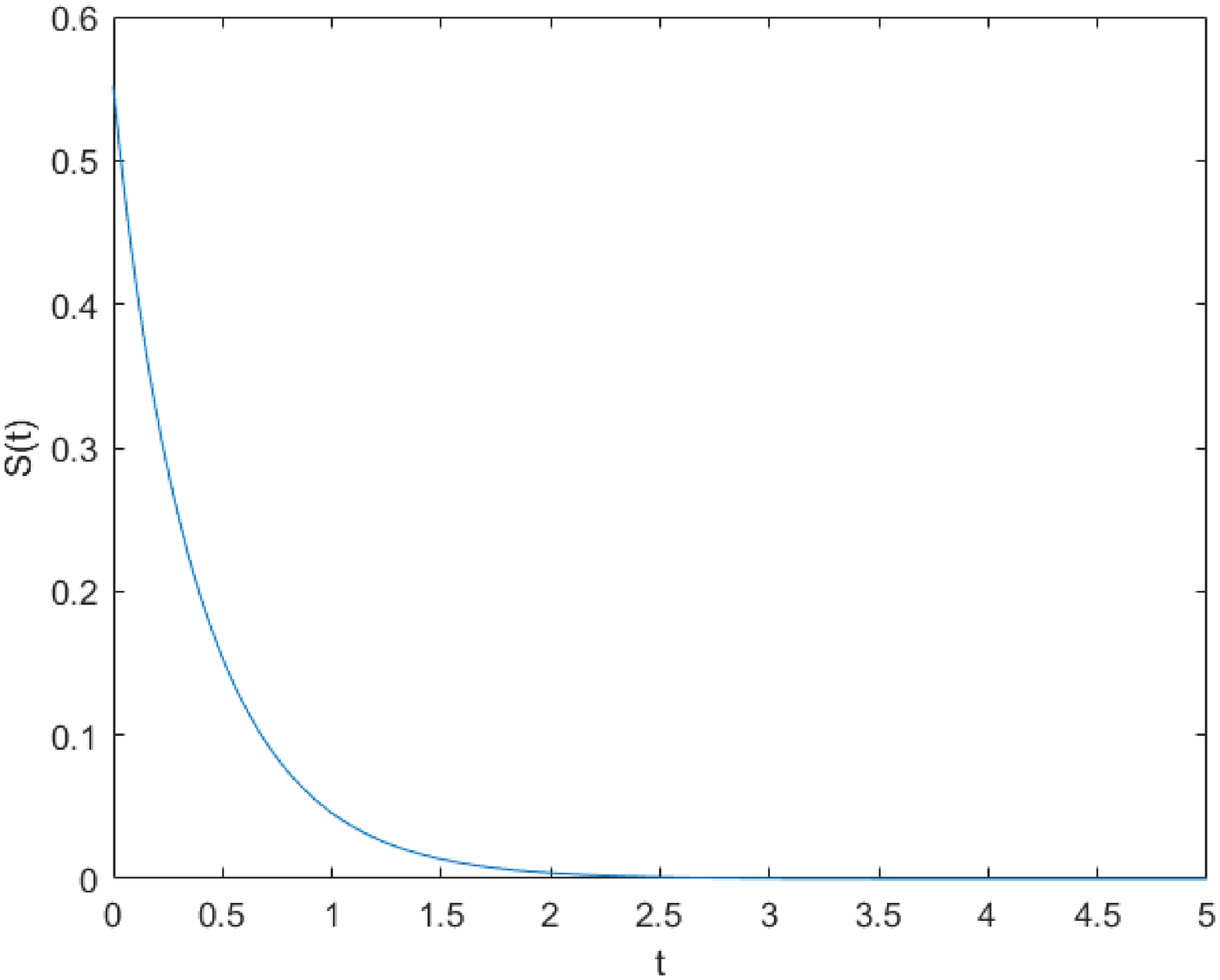}
    \end{minipage}
    \caption{Equation parameters $a_0=1,a_1=0.1,b=0$ with Gaussian initial condition $v_0=-1, \sigma_0^2=0.5$. Left: evolution of firing rate $N(t)$. Right: evolution of relative entropy $S(t)$ with $G(x)=\frac{(x-1)^2}{2}$.}
     \label{fig:relative_entropy3}
\end{figure}

As shown in \cite{caceres2011analysis}, there may be two stationary solutions for the system of $b>0$. For example, when $ a(N(t)) = 1$ and $b =1.5$, there are two different steady states whose firing rates are $N^{\infty}=2.319$ and $N^{\infty}=0.1924$. Given the firing rate $N^{\infty}$, the expression of $p^{\infty}(v)$is given by
\begin{equation}
    p^{\infty}(v)=\frac{N^{\infty}}{a\left(N^{\infty}\right)} e^{-\frac{h\left(v, N^{\infty}\right)^{2}}{2 a\left(N^{\infty}\right)}} \int_{\max \left\{v, V_{R}\right\}}^{V_{F}} e^{\frac{h\left(\omega, N^{\infty}\right)^{2}}{2 a\left(N^{\infty}\right)}} d \omega ,
\end{equation}
which is the stationary solution when we calculate the relative entropy for multiple steady-state problems. The results are shown in Figure \ref{fig:relative_entropy4}, where the steady state with a larger firing rate $N^{\infty}=2.319$ is unstable while the stationary solution with a lower firing rate $N^{\infty}=0.1915$ is stable. We see that the relative entropy decreases with time for the stable state, while the other one does not.
\begin{figure}[!htb]
    \centering
    \begin{minipage}[c]{0.49\textwidth}
        \centering
        \includegraphics[width=1\textwidth]{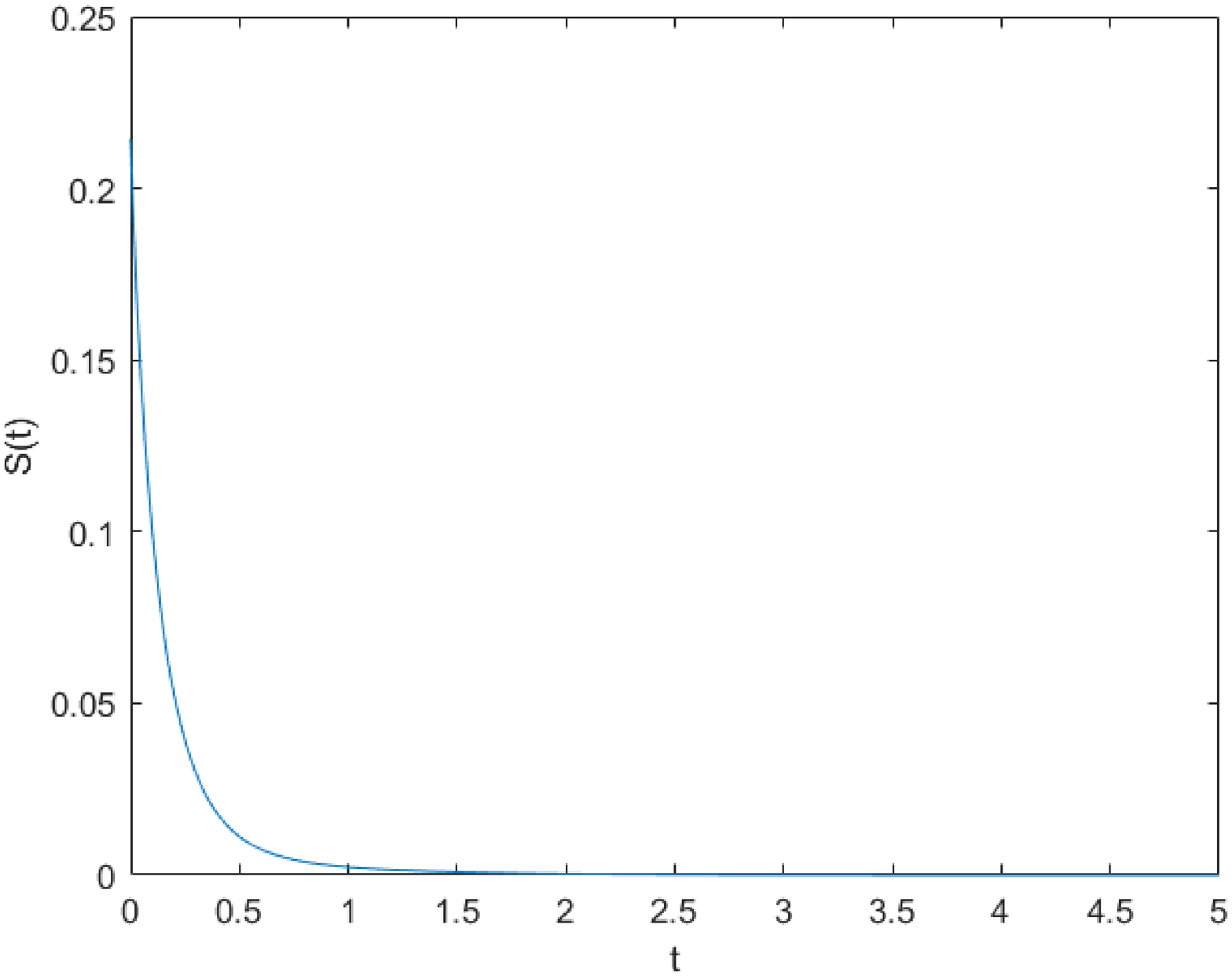}
    \end{minipage}
    \begin{minipage}[c]{0.49\textwidth}
        \centering
        \includegraphics[width=1\textwidth]{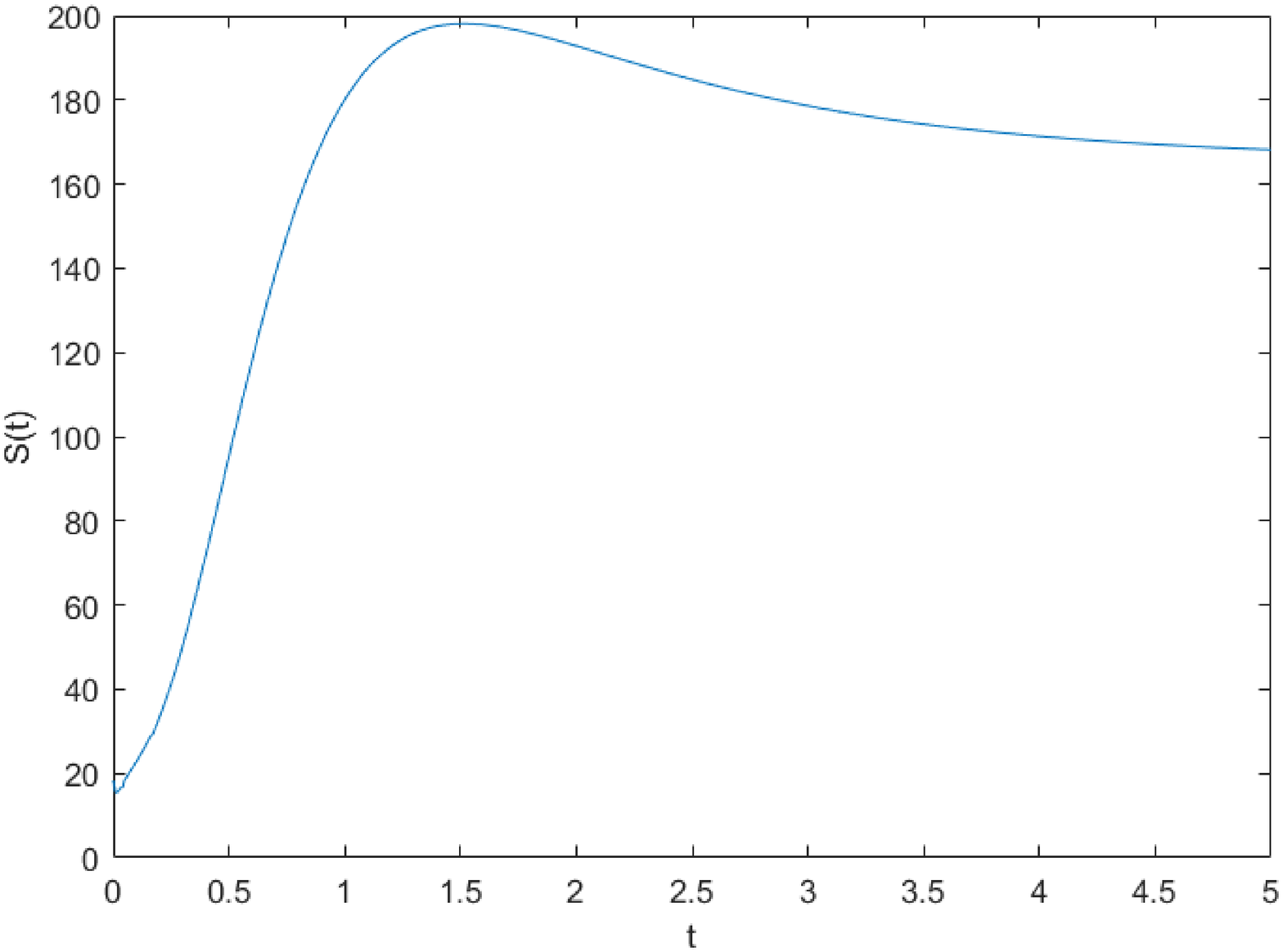}
    \end{minipage}
    \caption{Equation parameters $a=1,b=1.5$  with Gaussian initial condition $v_0=-1, \sigma_0^2=0.5$. In this case, the model has two stationary states with firing rates $N^{\infty}=0.1924$ and $N^{\infty}=2.319$. Left: evolution of relative entropy $S(t)$ with $G(x)=\frac{(x-1)^2}{2}$ for stable state with $N^{\infty}=0.1924$. Right: evolution of relative entropy $S(t)$ with $G(x)=\frac{(x-1)^2}{2}$ for unstable state with $N^{\infty}=2.319$ }
     \label{fig:relative_entropy4}
\end{figure}

\subsection{Learning and testing in NNLIF model with learning rules}\label{sec:Learning_testing}
In this part, we consider the learning and discrimination abilities in NNLIF model with learning rules. In \cite{perthame2017distributed}, the authors proposed a two-phase test to illustrate the discrimination property:\medskip

\textbf{Learning phase}
\smallskip

1. An heterogeneous input $I(w)$ is presented to the system, when the learning process is active. The initial data is supported on inhibitory weights and the learning rule is determined for the present weights by $-N(w)\bar{N}$ by taking $K(w) = -1$ if $w \leq 0$.

2. After some time, the synaptic weight distribution $H(w, t)$ converges to an equilibrium distribution $H^*_
I(w)$, which depends on $I$.\medskip

\textbf{Testing phase}
\smallskip

1. The learning process is now switched off, i.e. there is no w-direction convection, and a new input $J(w)$ is presented to the system.

2. After some time, the solution $p_J (v,w, t)$ reaches an equilibrium $p^*_J (v,w)$, which is characterized  by the output signal $N^*_J(w)$ which is the neural activity distribution across the heterogeneous populations.\medskip

Some numerical explorations of the learning behavior and discriminative properties of the model have been done in \cite{perthame2017distributed}\cite{he2022structure}. When the learning phase is over, in addition to the synaptic weight distribution $H(w, t)$, the equilibrium state $N_I(w)$ of the sub-network activity $N(w,t)$ can also be obtained, which we call the \textbf{prediction signal}. In the previous work on the time-independent input function $I(w)$ for the learning phase \cite{perthame2017distributed}\cite{he2022structure}, the prediction signal $N_I(w)$ is like a triangle depending on the input function $I(w)$ of the learning phase.  After the testing phase when the learning input $I(w)$ and testing input $J(w)$ are the same, the output signal $N_J^*(w)$ is like a triangle that matches the prediction signal $N_I(w)$; but when $I(w)$ and $J(w)$ are different, the output signal is not in a regular shape.

They explore learning and discriminative power in the model only if the input function is constant in time. In our work, we plan to explore how the model would react to a time-varying input signal through numerical experiments, and there have been studies in the field of neuroscience surrounding time-varying input \cite{isidori1990output}. Especially, we consider input functions that are time-periodic and explore the effect of oscillation periods on the learning ability of the model. To this aim, we have designed $4$ sets of experiments, progressively revealing the nature of its learning behavior.

\paragraph{Test 1. Synchronizing with oscillating inputs.}
We choose the testing input functions
\begin{equation}
    \begin{aligned}
        I_{1}&=\pi^{-\frac{1}{4}} e^{-\frac{1}{2}(10 w+5)^{2}}+2 \\
        I_{2}&=\pi^{-\frac{1}{4}} \sqrt{2}(10 w+5) e^{-\frac{1}{2}(10 w+5)^{2}}+2,
    \end{aligned}
\end{equation}
and the learning input function is periodically switching between those two
\begin{equation}
    \label{input}
    I(w,t)=a(t)I_1(w)+b(t)I_2(w),
\end{equation}
where
\begin{equation}
    \label{eq:input_coff}
    \begin{aligned}
        a(t)&=\frac{1+\cos(\frac{2\pi t}{D})}{2},\\
        b(t)&=1-a(t).
    \end{aligned}
\end{equation}

For other parameters, we choose $V_F=2,V_R=1,V_{\text{min}}=1,a=1,\varepsilon=0.1,W_{\text{min}}=-1.1,W_{\text{max}}=0.1,T_{\text{max}}=4,\sigma(\bar{N})=\bar{N},\Delta t=2.5\times 10^{-4},\Delta w=0.01$ and the initial condition 
\begin{equation}
    p_{\text{init}}=\begin{cases}
        \text{sin}^2(\pi v)\text{sin}^2(\pi w) \qquad &-1<w<0\text{ and }-1<v<1,\\
        0 &\text{otherwise}.
    \end{cases}
\end{equation}

In the learning phase, the input function changes periodically in time; the smaller the period is, the greater the rate of change of the input function is. The total network activity $\bar{N}$ is an intuitive response to the input function, so we first observe the change in the total network activity. First, we choose period $D=1,0.5,0.2$. 
\begin{figure}[!htb]
    \centering
        \begin{minipage}[c]{0.3\textwidth}
            \centering
            \includegraphics[width=1\textwidth]{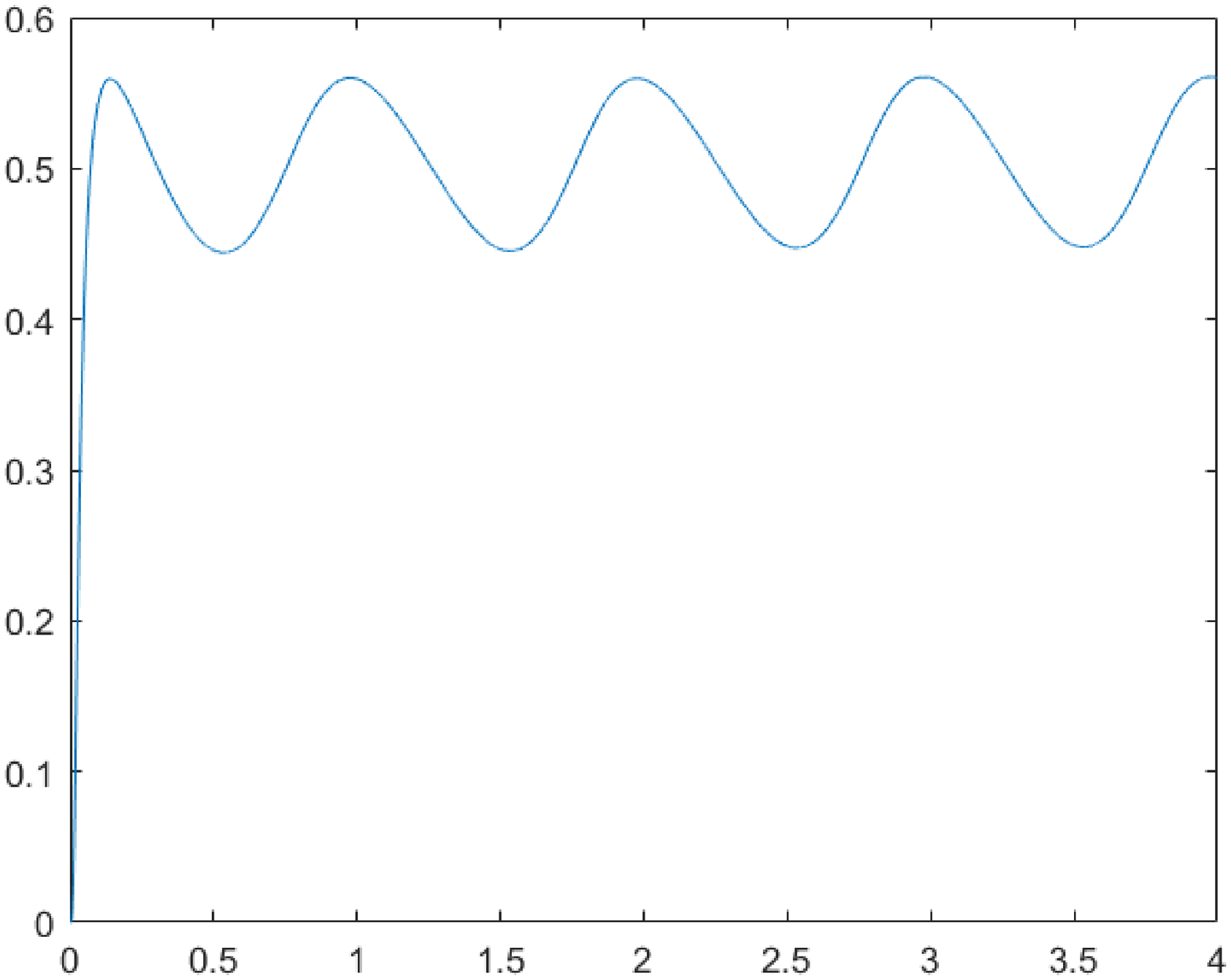}
        \end{minipage}
        \begin{minipage}[c]{0.3\textwidth}
            \centering
            \includegraphics[width=1\textwidth]{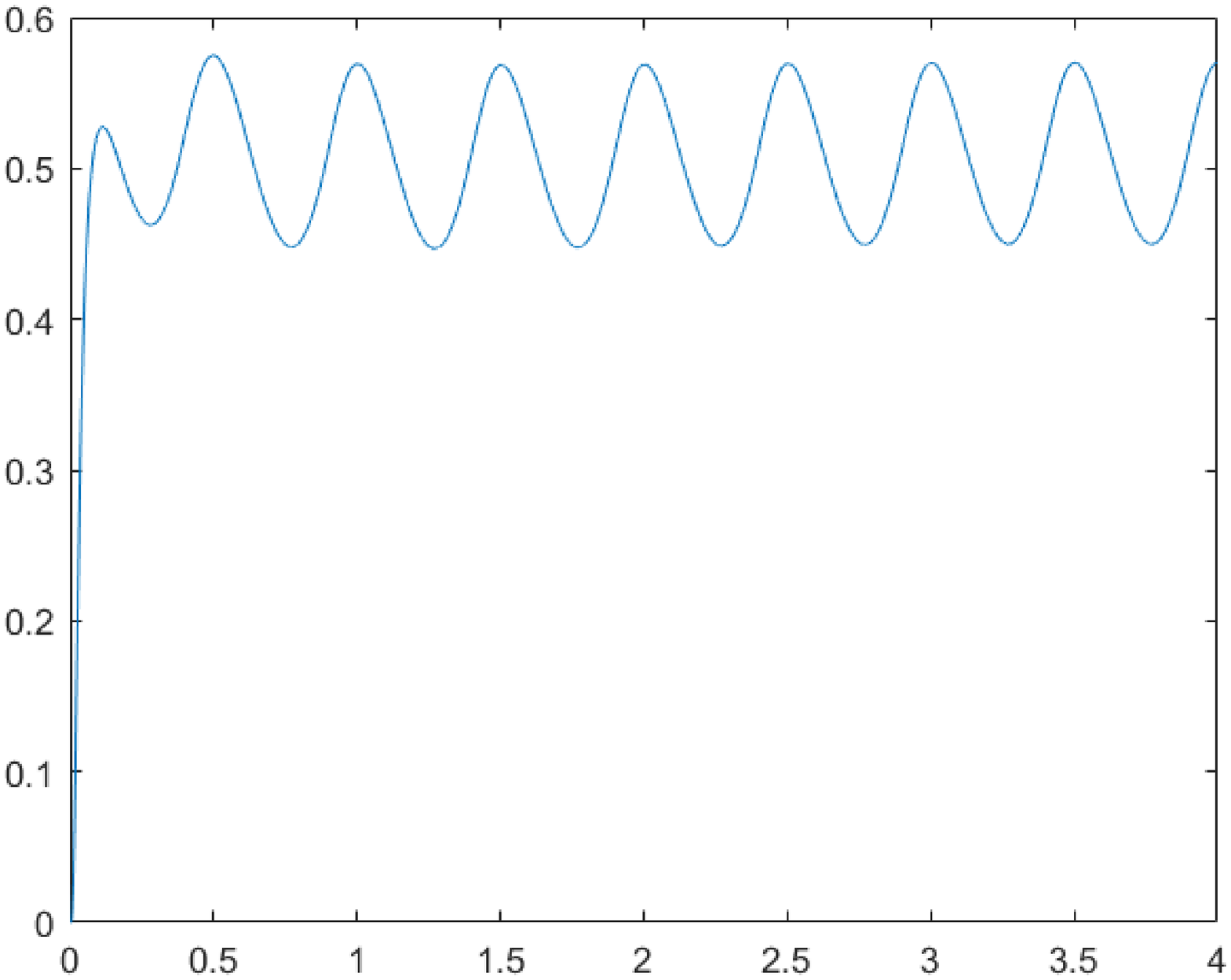}
        \end{minipage}
        \begin{minipage}[c]{0.3\textwidth}
            \centering
            \includegraphics[width=1\textwidth]{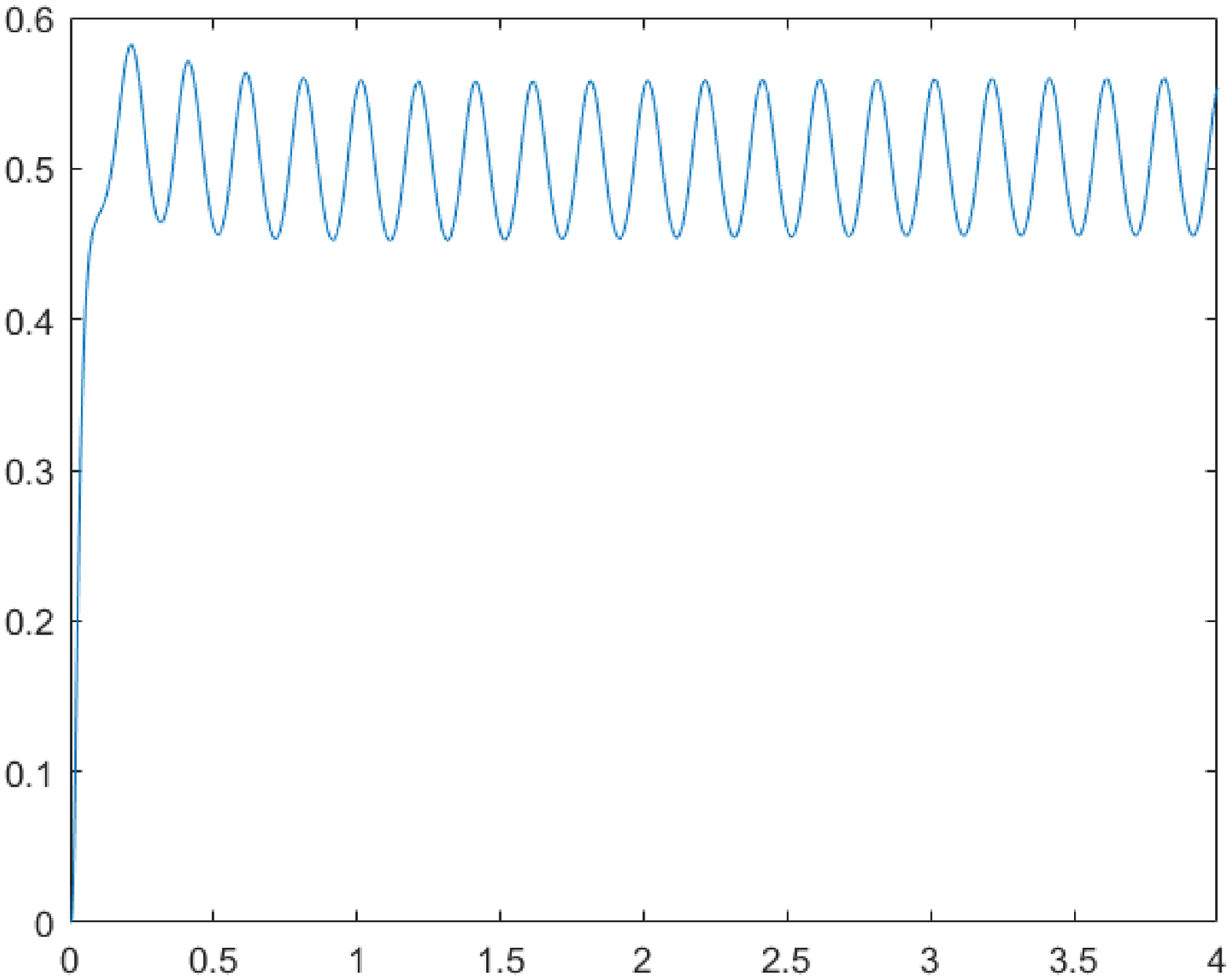}
        \end{minipage}
    \caption{Equation parameters $a=1$ and $\varepsilon=0.1$. The evolution of total firing rate $\bar{N}$. Left: the input function period $D=1$. Middle: the input function period $D=0.5$. Right: the input function period $D=0.2$.}
    \label{fig:Nbar}
\end{figure}

Figure \ref{fig:Nbar} shows the evolution of the total firing rate at different periods. As we expected, except for the initial transient evolutionary phase, the total activity of the network changes periodically over time and its period is the same as the input function.

\paragraph{Test 2. Adapting to fast oscillating inputs.}
Since the prediction signal is determined by the learning input function and reflects the model's learning of the learning input function $I(w,t)$, observing the prediction signal in different periods helps us explore the learning behavior of the model. We compare numerical results for different periods $D=4,0.4,0.2,0.1,0.01$.  In this case, the last input function learned by the model is $I(w,t_\text{max})=I_1$.

\begin{figure}[!htb]
    \centering
        \begin{minipage}[c]{0.3\textwidth}
            \centering
            \includegraphics[width=1\textwidth]{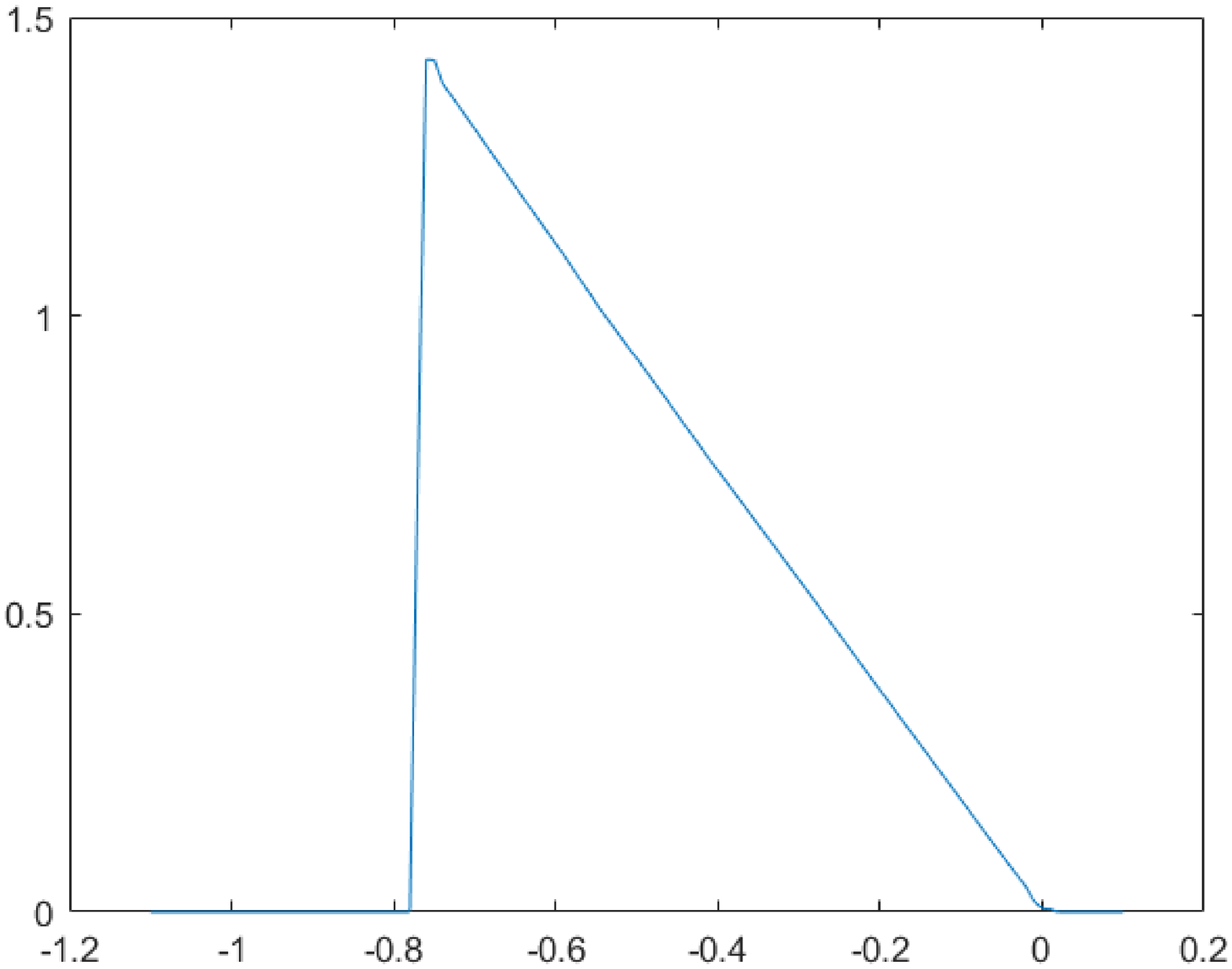}
        \end{minipage}
        \begin{minipage}[c]{0.3\textwidth}
            \centering
            \includegraphics[width=1\textwidth]{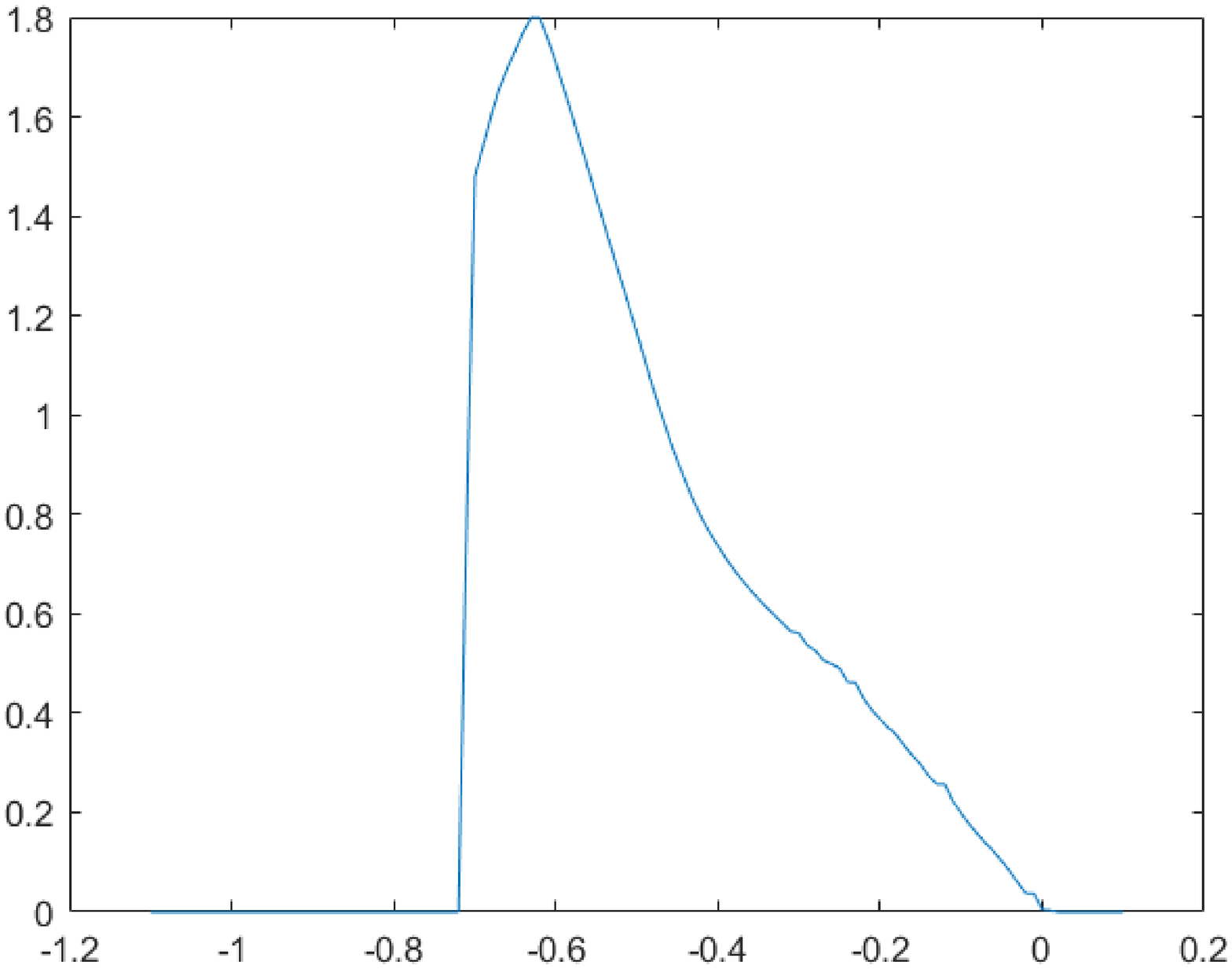}
        \end{minipage}
        \begin{minipage}[c]{0.3\textwidth}
            \centering
            \includegraphics[width=1\textwidth]{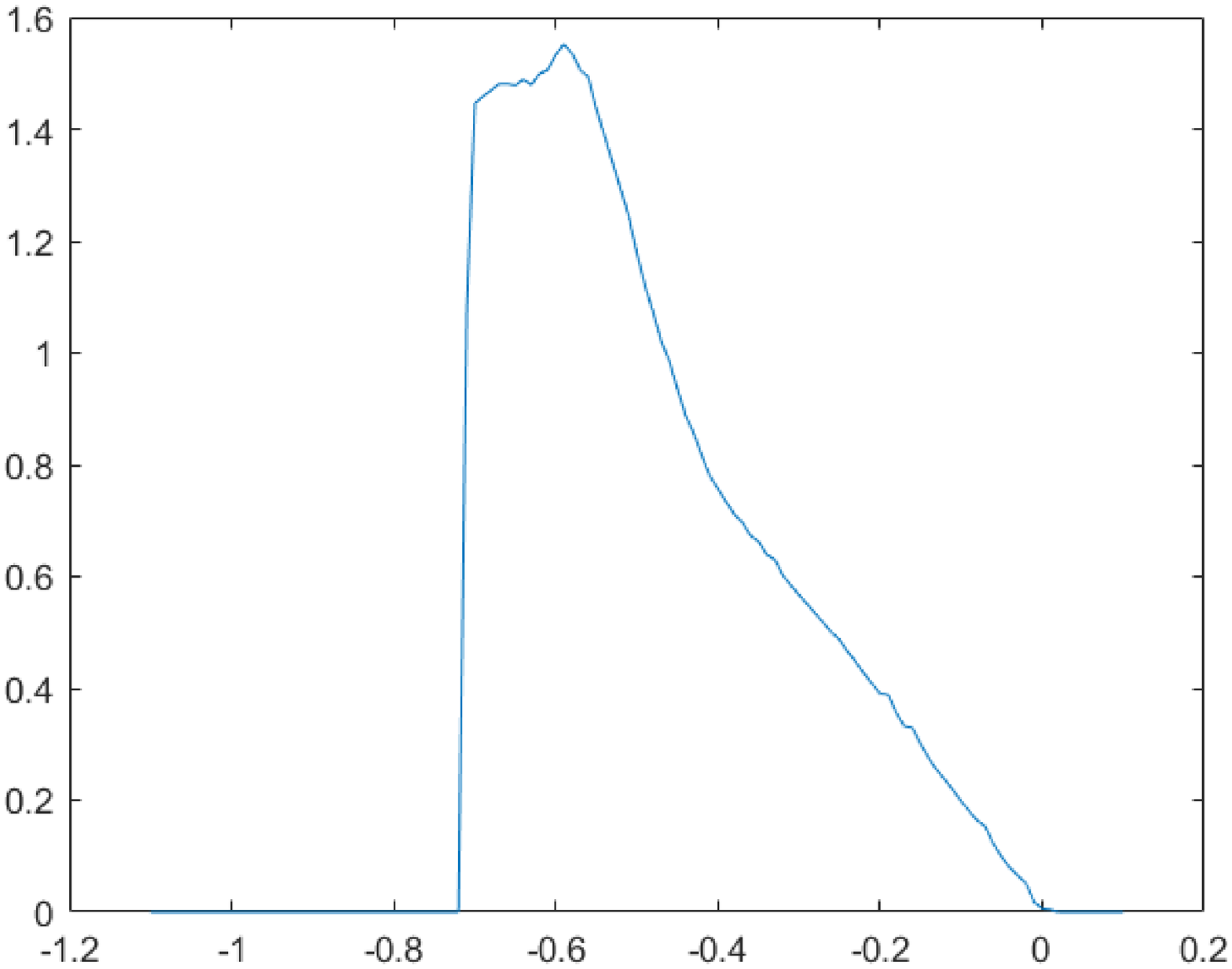}
        \end{minipage}
        \begin{minipage}[c]{0.3\textwidth}
            \centering
            \includegraphics[width=1\textwidth]{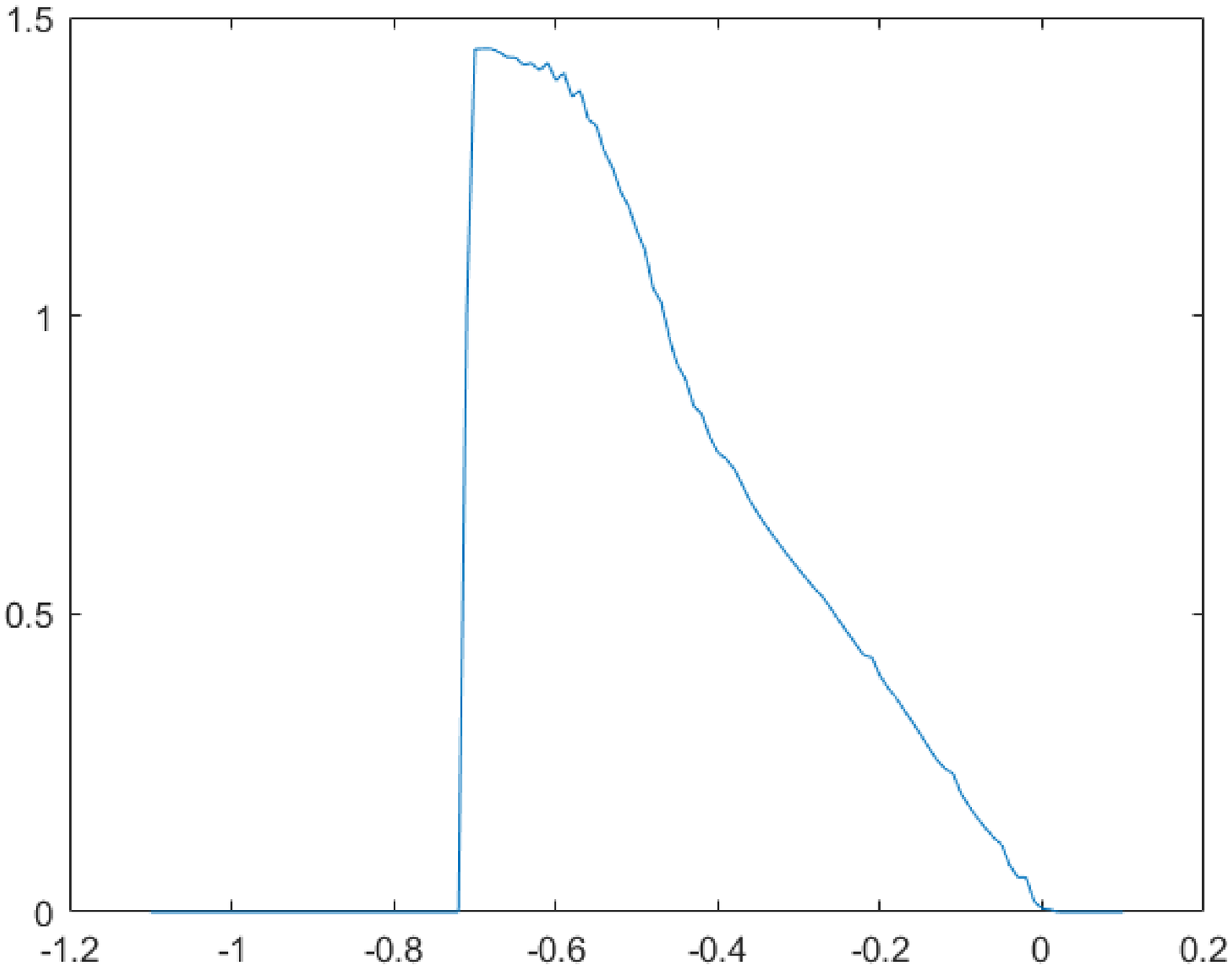}
        \end{minipage}
        \begin{minipage}[c]{0.3\textwidth}
            \centering
            \includegraphics[width=1\textwidth]{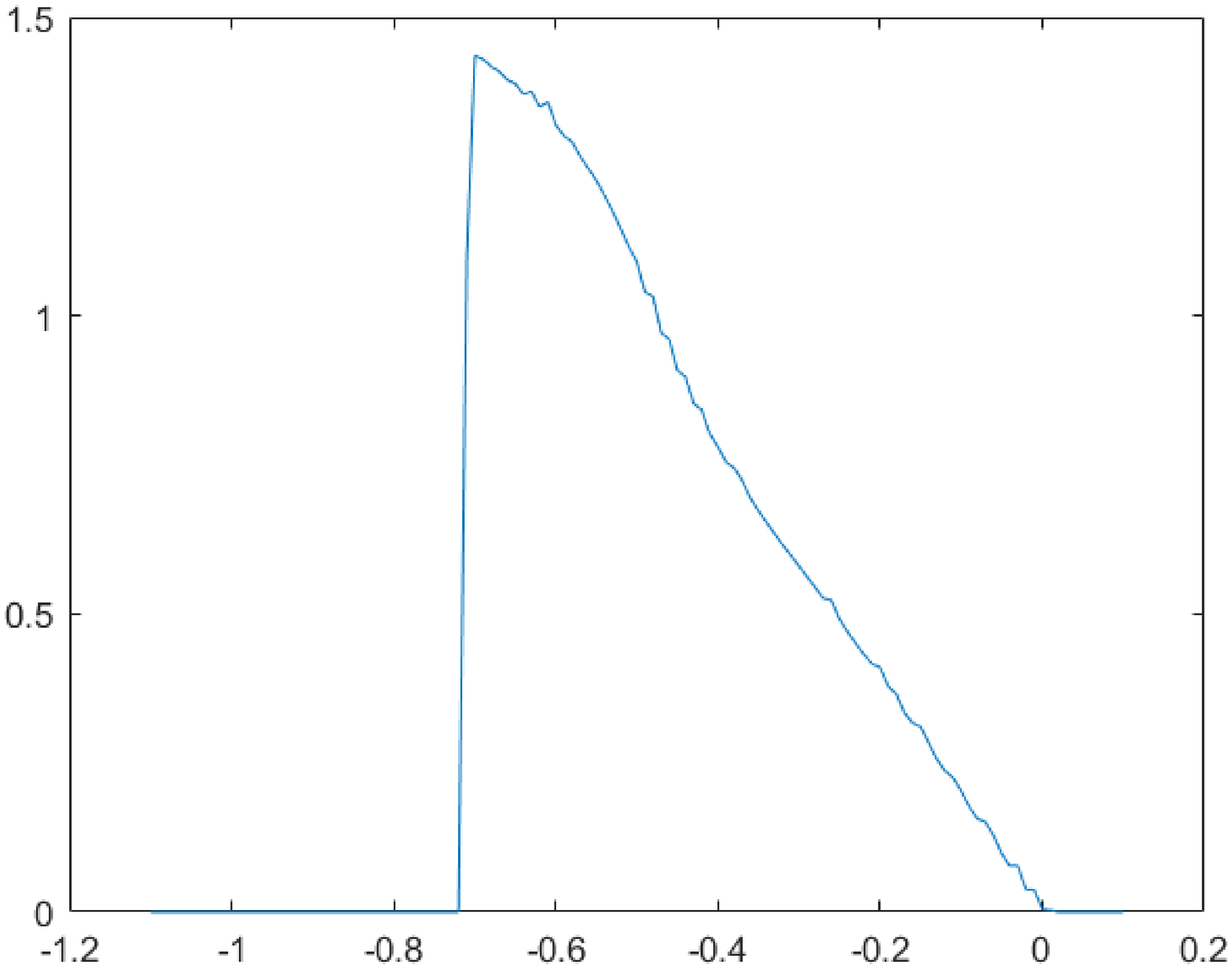}
        \end{minipage}
        \begin{minipage}[c]{0.3\textwidth}
            \centering
            \includegraphics[width=1\textwidth]{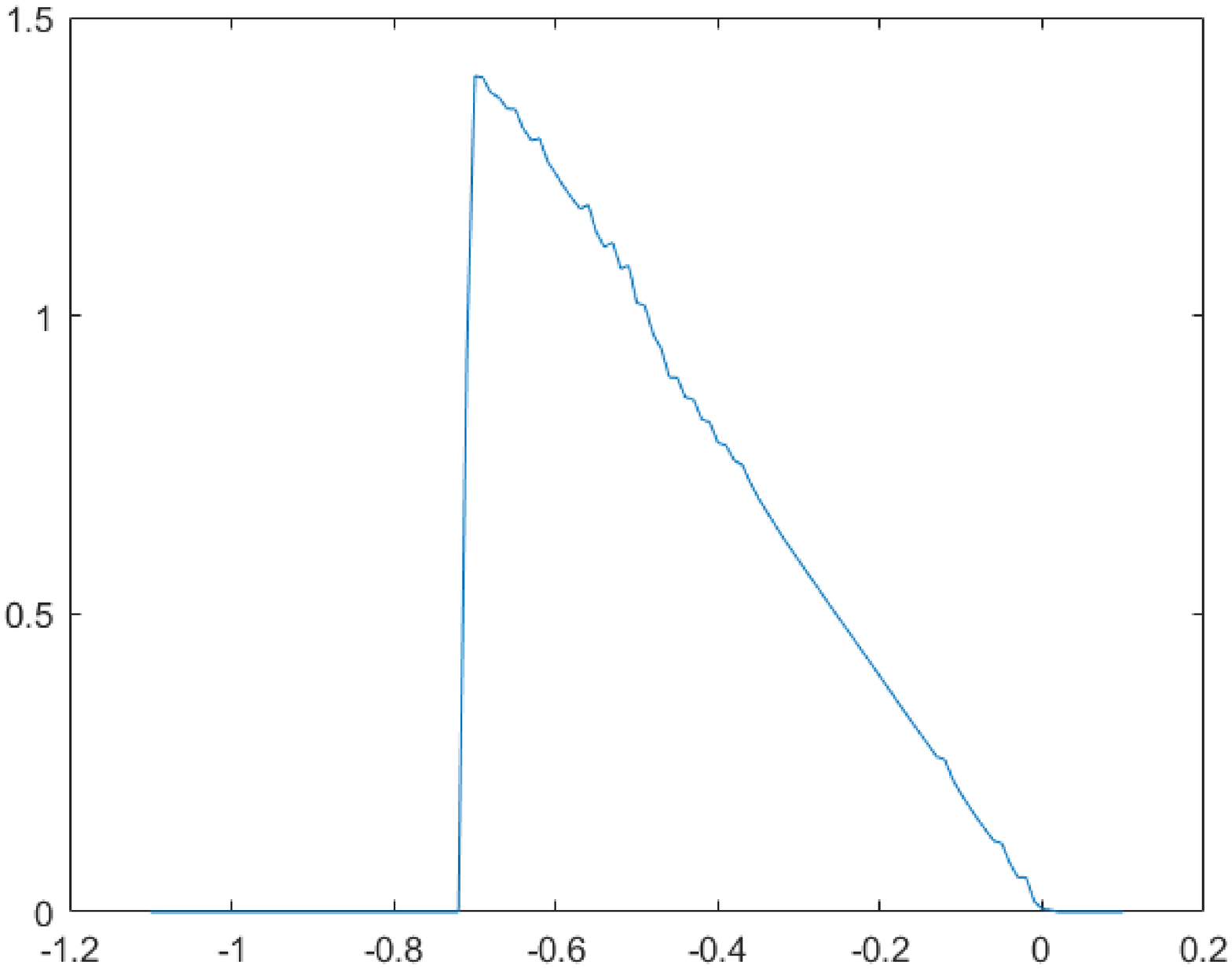}
        \end{minipage}
    \caption{Equation parameters $a=1$ and $\varepsilon=0.1$. The prediction signal at $t=4$ with different input function periods. Top: the input function period $D=4,0.4,0.2$ from left to right. Bottom: the input function period $0.1,0.05,0.01$ from left to right.}
    \label{fig:prediction}
\end{figure}

Figure \ref{fig:prediction} shows the prediction signal at different periods. When the period is large, the prediction signal is like a triangle. As the period gets smaller, the shape of the prediction signal is getting more and more irregular. However, as the period is getting further smaller, the shape of the prediction signal is becoming triangular again. In previous experiments \cite{perthame2017distributed}\cite{he2022structure}, for the time-independent learning input signal $I(w)$, the test signal always resembles a triangle. So we speculate from Figure \ref{fig:prediction} that for sufficiently large or sufficiently small periods, the predicted signal looks like a triangle, and the model has effectively learned a signal of a certain form.

\paragraph{Test 3. Learning from oscillating inputs.}

In order to verify the above conjecture, we choose a relatively large period with $D=4$ and a small period with $D=0.01$ in \eqref{eq:input_coff}.  In the testing phase, we choose testing input functions $J=I_1$, $J=I_2$, and $J=\frac{I_1+I_2}{2}$.
\begin{figure}[!htb]
    \centering
        \begin{minipage}[c]{0.3\textwidth}
            \centering
            \includegraphics[width=1\textwidth]{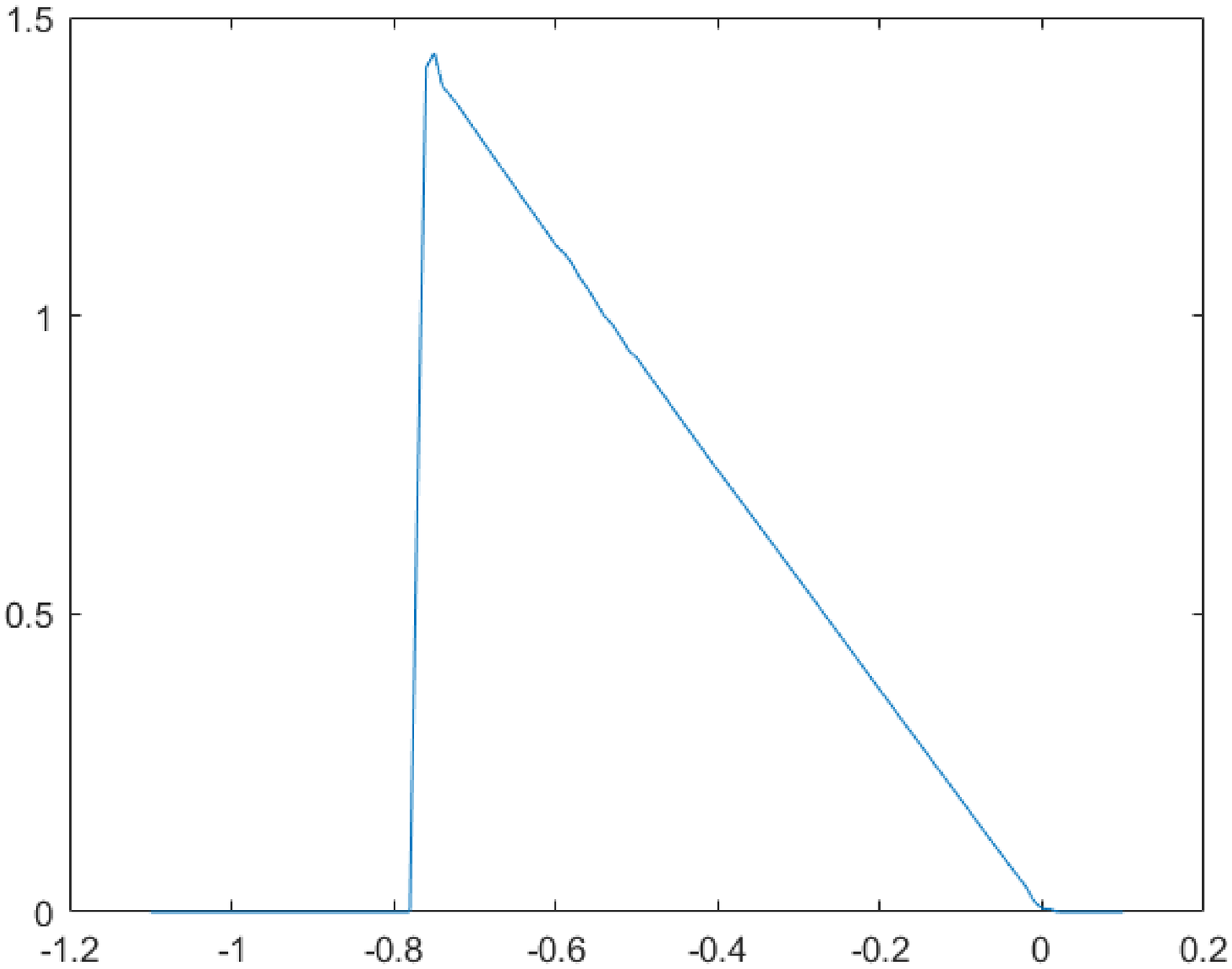}
        \end{minipage}
        \begin{minipage}[c]{0.3\textwidth}
            \centering
            \includegraphics[width=1\textwidth]{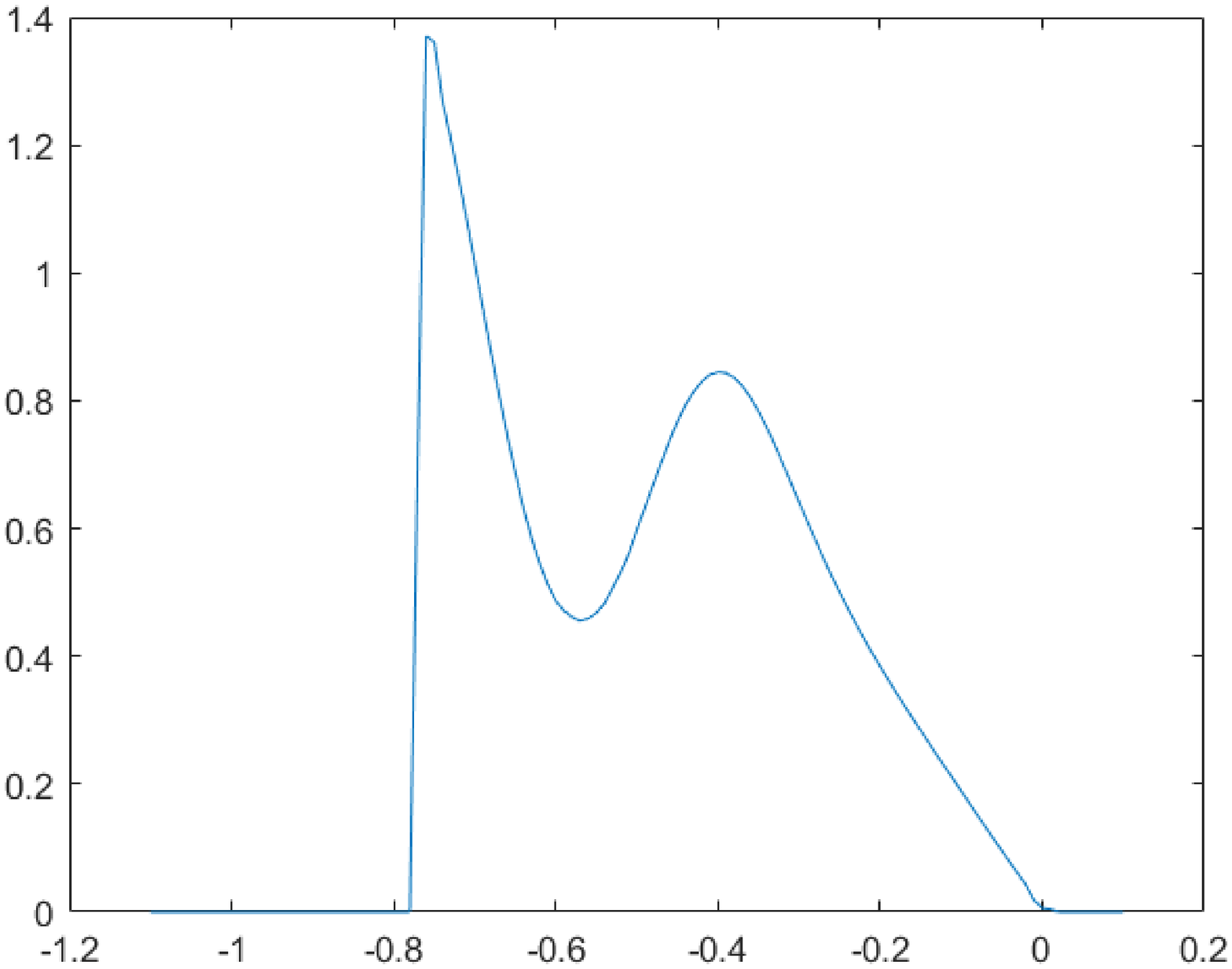}
        \end{minipage}
        \begin{minipage}[c]{0.3\textwidth}
            \centering
            \includegraphics[width=1\textwidth]{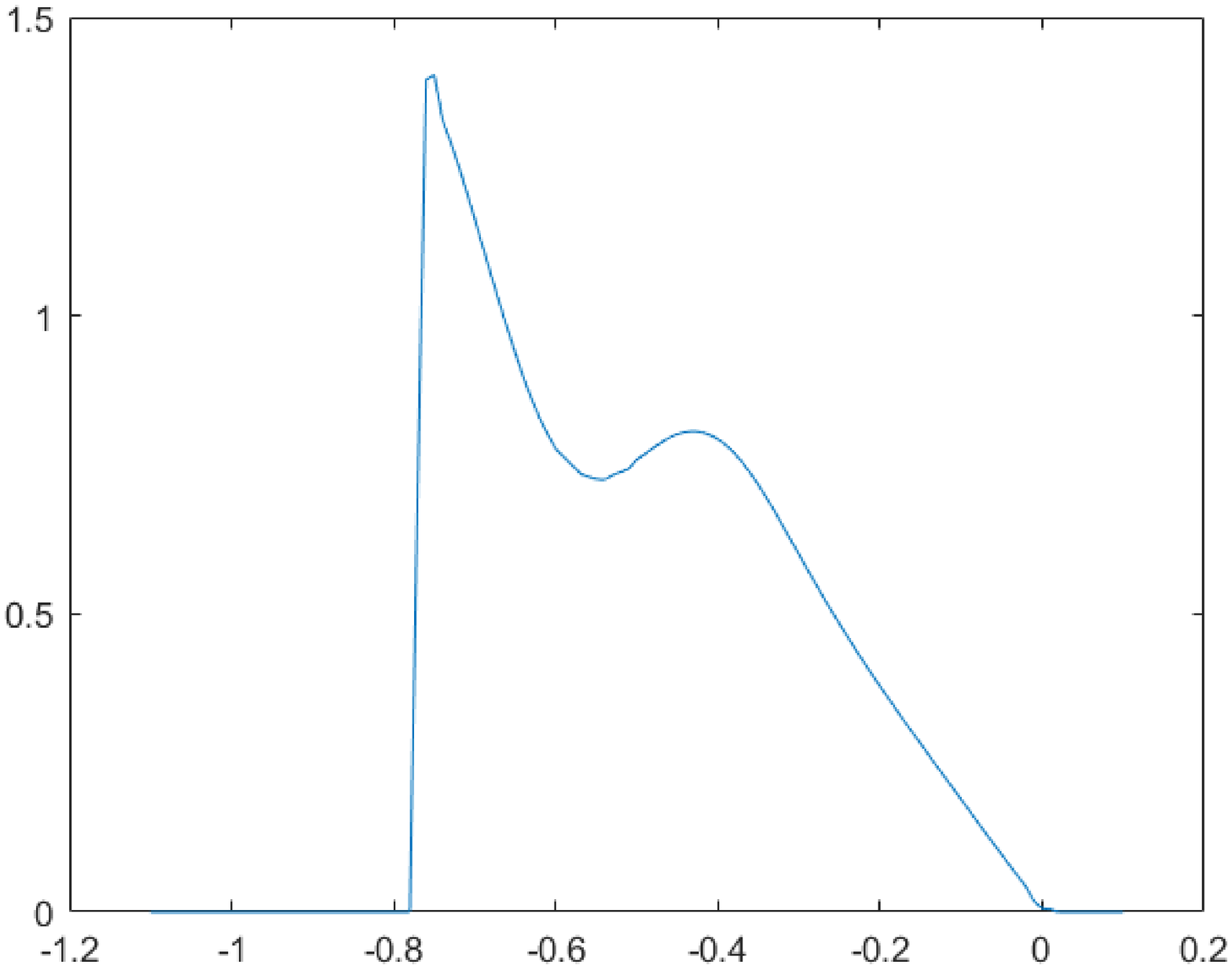}
        \end{minipage}
    \caption{(Output signal for the large period learning input) The
final firing rate $N(w)$ with different testing input $J(w)$. Equation parameters $a=1$ and $\varepsilon=0.1$, and the period of the input function in the learning phase is $D=4$. Left: Output signal with testing input function $J=I_1$. Middle: Output signal with testing input function $J=I_2$. Right: Output signal with testing input function $J=\frac{I_1+I_2}{2}$.}
    \label{fig:largeD}
\end{figure}
\begin{figure}[!htb]
    \centering
        \begin{minipage}[c]{0.3\textwidth}
            \centering
            \includegraphics[width=1\textwidth]{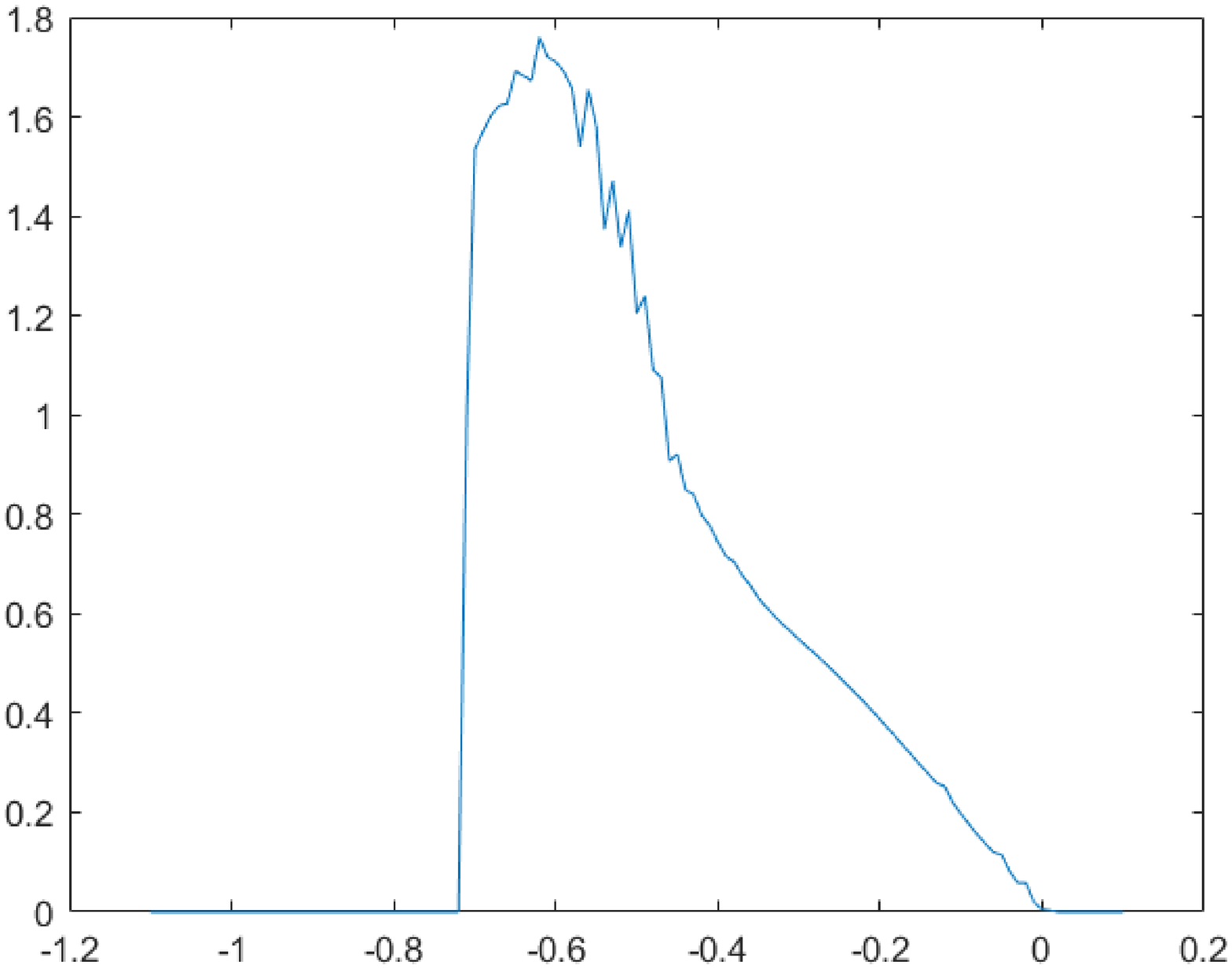}
        \end{minipage}
        \begin{minipage}[c]{0.3\textwidth}
            \centering
            \includegraphics[width=1\textwidth]{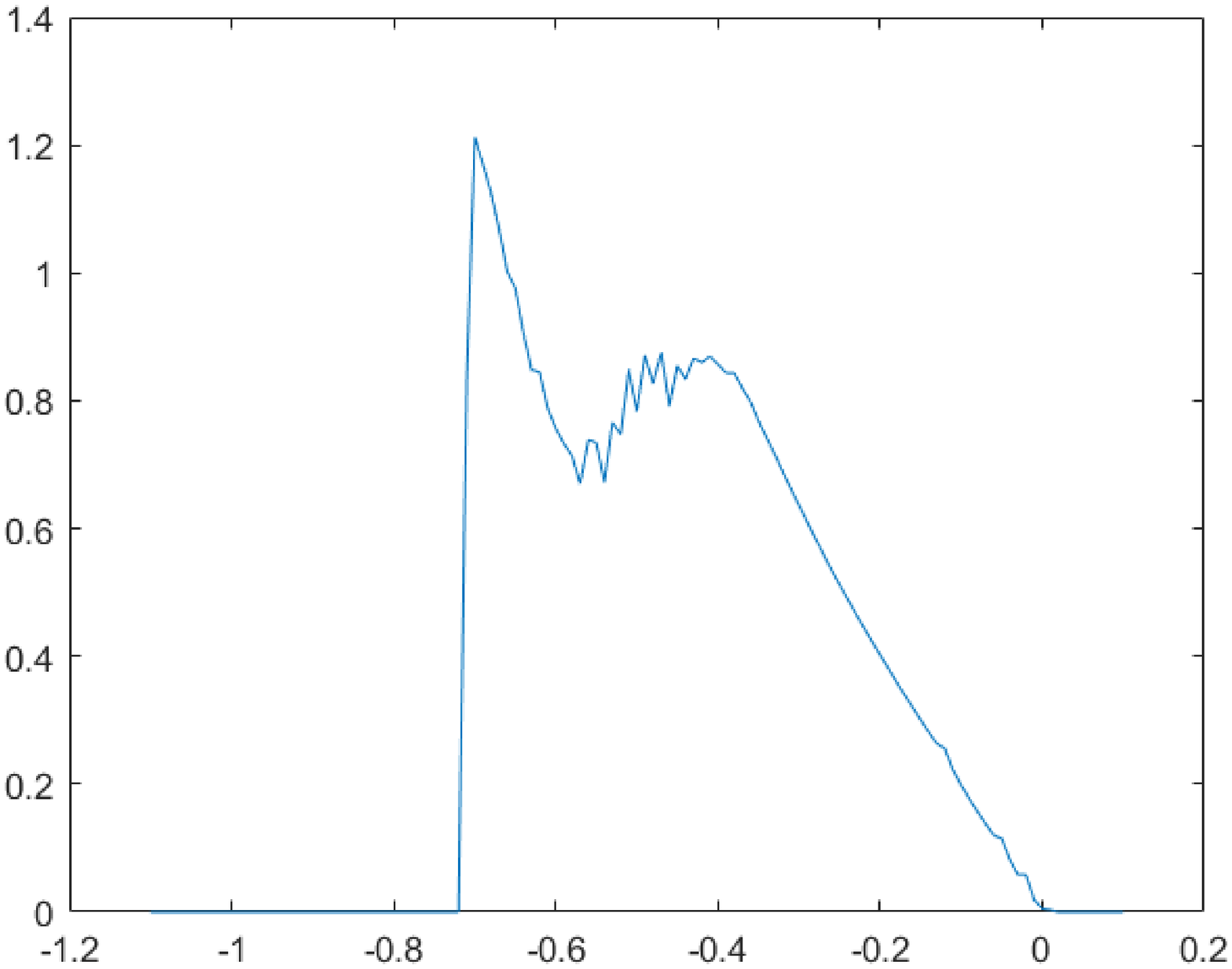}
        \end{minipage}
        \begin{minipage}[c]{0.3\textwidth}
            \centering
            \includegraphics[width=1\textwidth]{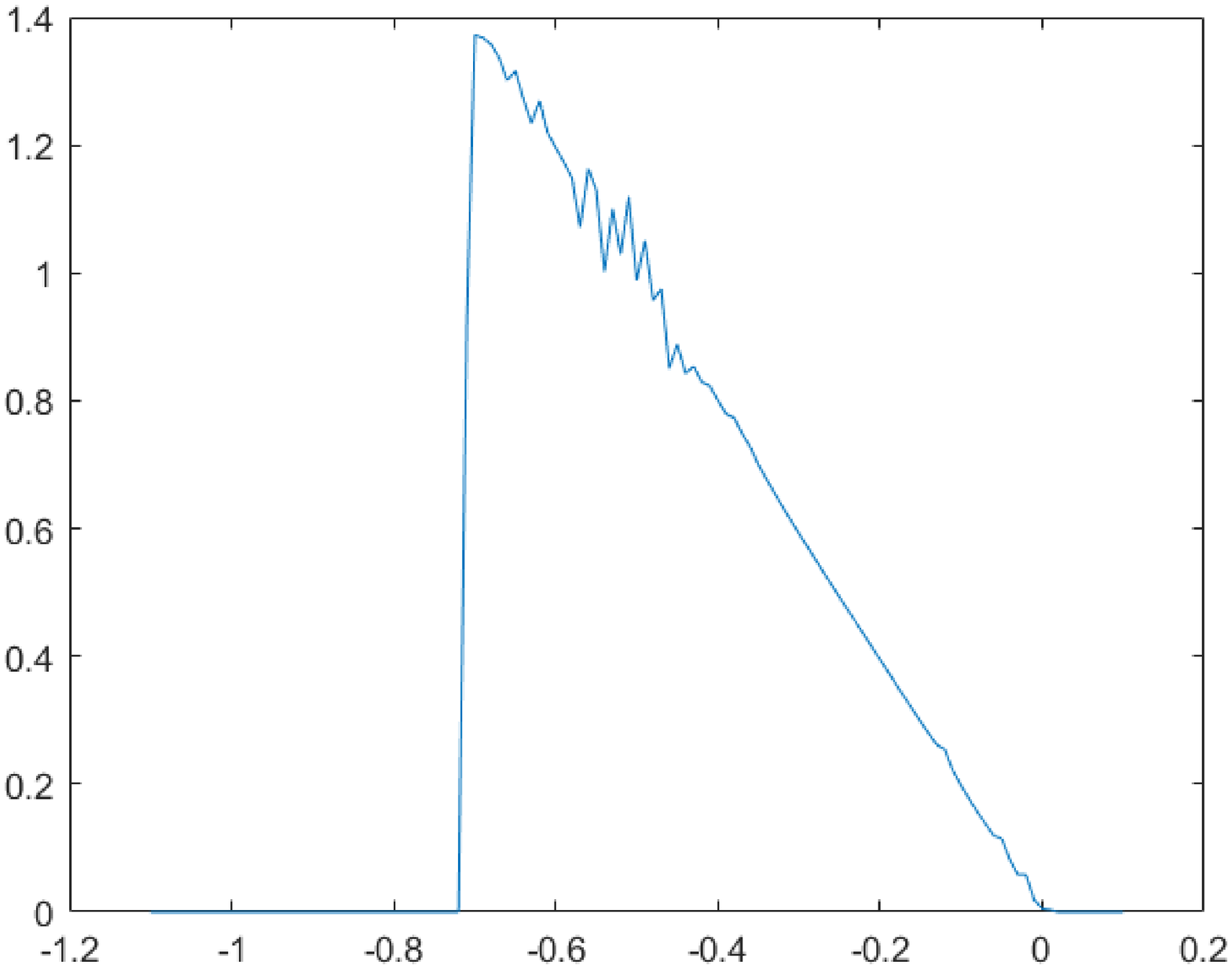}
        \end{minipage}
    \caption{Output signal for the small period learning input) The
final firing rate $N(w)$ with different testing input $J(w)$. Equation parameters $a=1$ and $\varepsilon=0.1$, and the period of the input function in the learning phase is $D=0.01$. Left: Output signal with testing input function $J=I_1$. Middle: Output signal with testing input function $J=I_2$. Right: Output signal with testing input function $J=\frac{I_1+I_2}{2}$.}
    \label{fig:shortD}
\end{figure}

Figure \ref{fig:largeD} shows the output signal of period $D=4$ with testing input function $J=I_1$, $J=I_2$ and $J=\frac{I_1+I_2}{2}$. When $J=I_1=I(w,t_\text{max})$, the output signal $N_J^*(w)$ is like a triangle. Figure \ref{fig:shortD} shows the output signal of period $D=0.01$ with testing input function $J=I_1$, $J=I_2$ and $J=\frac{I_1+I_2}{2}$. When $J=\frac{I_1+I_2}{2}$, the output signal $N_J^*(w)$ is like a triangle. Numerical results show that when the period is relatively large, the signal learned by the model matches $I_1$, and when the period is relatively small, it matches $\frac{I_1+I_2}{2}$. 

The experimental results can be interpreted as follows. When the period is large, the model has enough time to learn, so the learned signal is the input function at the last moment. And when the period is small, neither $I_1$ nor $I_2$ can be learned well, but the result of learning is the average of the two.  Because when the switching process is too fast, the effect of the model on the learning of either $I_1$ or $I_2$ is poor. Instead, the average signal $\frac{I_1+I_2}{2}$ is captured by the time averaging of the learning process.

 \paragraph{Test 4. Phase diagram for leaning.} There are multiple typical time scales in this model: the time scale for the voltage activities, the time scale for learning by redistributing the synaptic weights and the time period in the external input. When introducing the model, we perform a time rescaling for \eqref{eq:problem4}, where the parameter $\varepsilon$ reflects the ratio between the time scales of voltage activities and learning. In the next numerical experiment, we choose $\varepsilon=1,0.5,0.25,0.125$ and periods $D=2^2,2^1,\dots,2^{-7}$ to compare the results of the output signal under different parameters. After the testing phase, we choose the total activity $\bar{N}(t)$ to quantify the output signal:
\begin{equation}
    \label{judge_tool}
    E^{\varepsilon,D}_J=\left| \bar{N}^{\varepsilon,D}_J -\bar{N}^{\varepsilon}_J \right|.
\end{equation}
Here, $\bar{N}^{\varepsilon,D}_J$ denotes the total activity when the equation parameter is $\varepsilon$, the learning input function is given by \eqref{input} with period $D$, while the testing input function is $J$. $\bar{N}^{\varepsilon}_J$ represents the total activity where the equation parameter is $\varepsilon$ and both the learning input function and the testing input function are $J$. $E^{\varepsilon,D}_J$ can roughly measure the output signal. The closer the value of $E^{\varepsilon,D}_J$ is to zero, the superior the model's learning efficacy.

\begin{figure}[!htb]
    \centering
        \begin{minipage}[c]{0.49\textwidth}
            \centering
            \includegraphics[width=1\textwidth]{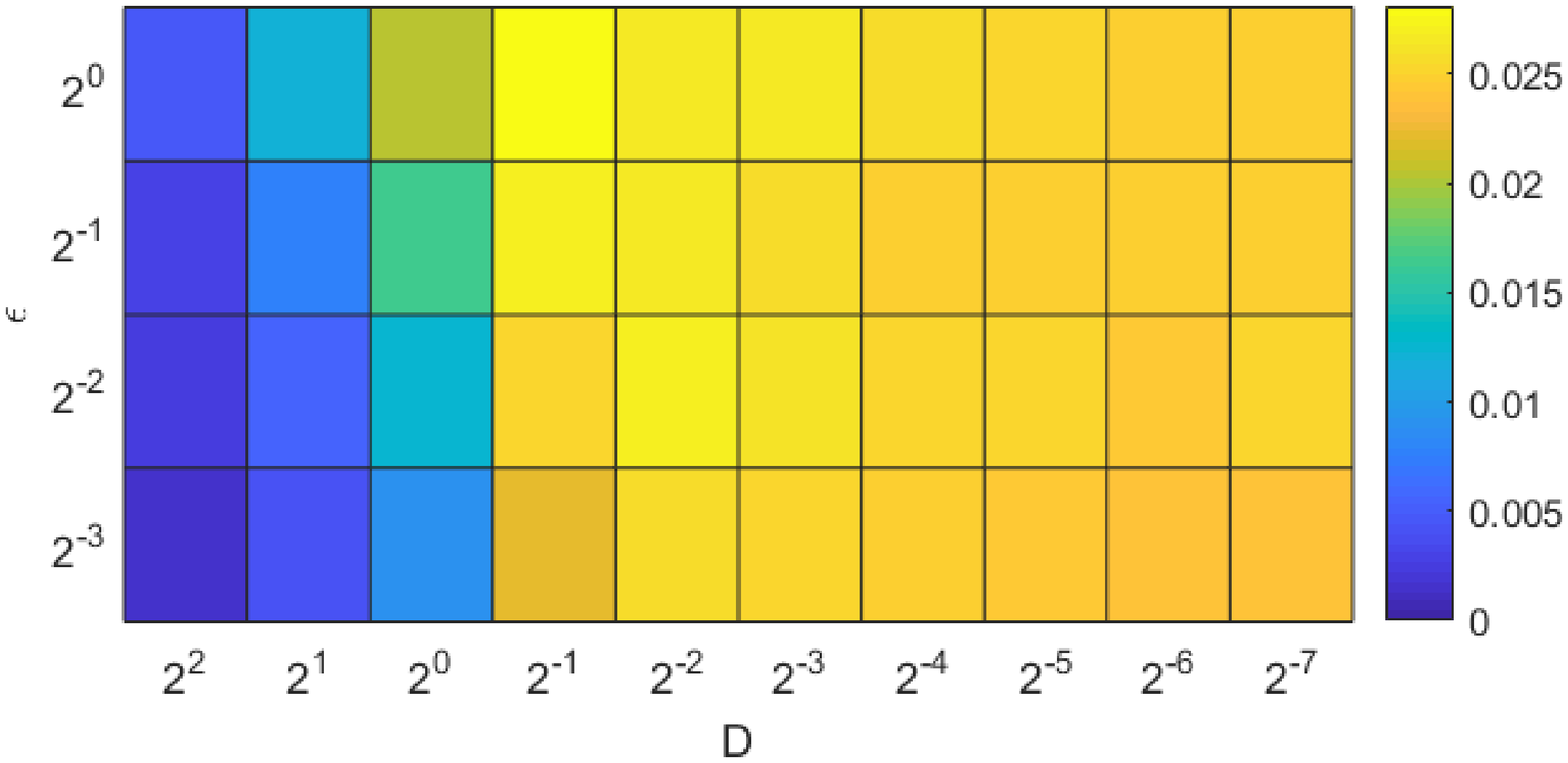}
        \end{minipage}
        \begin{minipage}[c]{0.49\textwidth}
            \centering
            \includegraphics[width=1\textwidth]{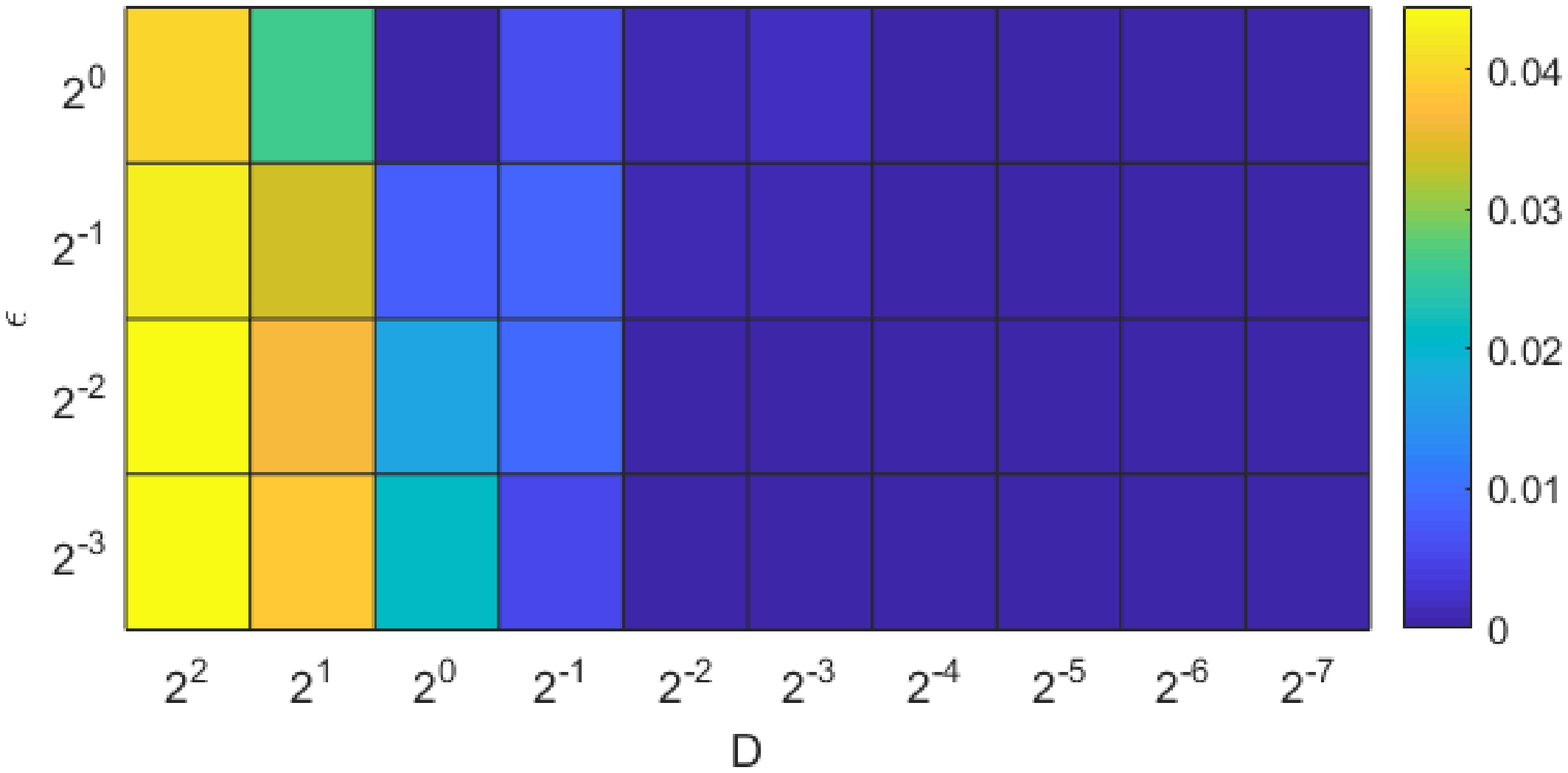}
        \end{minipage}
    \caption{Equation parameter $a=1$. Left: The value of $E^{\varepsilon,D}_J$ under different $D$ and $\varepsilon$ with input function $J=I_1$. Right: The value of $E^{\varepsilon,D}_J$ under different $D$ and $\varepsilon$ with input function $J=\frac{I_1+I_2}{2}$.}
    \label{fig:output3}
\end{figure}

 As shown in Figure \ref{fig:output3}, as the period becomes smaller, the testing indicator becomes less significant with respect to the testing input function $J=I_1$, and the test indicator becomes more significant with respect to the testing input function $J=\frac{I_1+I_2}{2}$. Besides, the numerical results also suggest that when epsilon is small, the transition in learning takes place at a smaller time period, whereas such a trend is not prominent. Although the experiments are not fully conclusive yet, they show a lot of promise for using the proposed numerical method to simulate large-scale tests.

\section{Conclusion} \label{sec:conclusion}

In this work, we have proposed a numerical scheme for approximating the Fokker-Planck equation using spectral methods for spatial discretization and successfully applied it to models with multiple time scales. Specific trial function space can guarantee that dynamic bounds are always satisfied, and suitable test function space is essential for ensuring stability and asymptotic-preserving. This method shows clear advantages regarding computational efficiency for high accuracy compared to existing methods. Besides essential convergence tests and assessments of computational efficiency, we execute assorted numerical examples exploring the response of various solutions. Subsequent to simulating the blow up phenomenon and relative entropy decay of the NNLIF model, we have studied the learning and discriminating behavior of the NNLIF model with learning rules when the input signal is time-dependent. The experimental results demonstrate that the learning behavior of the model obeys the general trend and provides clues for further research. Moving forward, we can still probe new numerical schemes for unbounded domains as well as explore more complex Fokker-Planck equations from neuroscience.

\section*{Acknowledgements}
The work of Z.Zhou is partially supported by the National Key R\&D Program
of China (Project No. 2021YFA1001200, 2020YFA0712000), and the National Natural Science Foundation of China (Grant No. 12031013, 12171013).
This work of Y.Wang is partially supported by the National Natural Science Foundation of China (Grant No. 12171026, U2230402 and 12031013), and Foundation of President of China Academy of Engineering Physics (YZJJZQ2022017).
\appendix

\section{Appendix}
\label{sec:app}
\subsection{The basis functions of $\mathrm{W}_1$ space}
\label{app:basis}

For simplicity, the basis functions of $\mathrm{W}_1$ can be selected as 3rd-degree piecewise polynomials, which is
 \begin{equation}
\begin{aligned}
    &\displaystyle g_1=\begin{cases}
        \displaystyle \frac{2}{(V_{\text{min}}-V_R)^3}v^3-\frac{3(V_{\text{min}}+V_R)}{(V_{\text{min}}-V_R)^3}v^2+\frac{6V_{\text{min}}V_R}{(V_{\text{min}}-V_R)^3}v-\frac{V_{\text{min}}^2(3V_R-V_{\text{min}})}{(V_{\text{min}}-V_R)^3} \qquad &v\in[V_{\text{min}},V_R),\\
        \displaystyle \frac{2}{(V_F-V_R)^3}v^3-\frac{3(V_F+V_R)}{(V_F-V_R)^3}v^2+\frac{6V_FV_R}{(V_F-V_R)^3}v-\frac{V_F^2(3V_R-V_F)}{(V_F-V_R)^3} &v\in[V_R,V_F],
    \end{cases}\\
    &\displaystyle g_2=\begin{cases}
        \displaystyle  \frac{1}{(V_{\text{min}}-V_R)^2}v^3-\frac{2V_{\text{min}}+V_R}{(V_{\text{min}}-V_R)^2}v^2+\frac{V_{\text{min}}(V_{\text{min}}+2V_R)}{(V_{\text{min}}-V_R)^2}v-\frac{V_{\text{min}}^2V_R}{(V_{\text{min}}-V_R)^2} \qquad &v\in[V_{\text{min}},V_R),\\
        \displaystyle  \frac{1}{(V_F-V_R)^2}v^3-\frac{2V_F+V_R}{(V_F-V_R)^2}v^2+\frac{V_F(V_F+2V_R)}{(V_F-V_R)^2}v-\frac{V_F^2V_R}{(V_F-V_R)^2} &v\in[V_R,V_F],
    \end{cases}\\
    &g_3=\begin{cases}
        0 \qquad &v\in [V_{\text{min}},V_R),\\
        \displaystyle  \frac{2}{(V_F-V_R)^2}v^3-\frac{3(V_F+V_R)}{(V_F-V_R)^2}v^2+\frac{(V_F+V_R)^2+2V_RV_F}{(V_F-V_R)^2}v-\frac{V_FV_R(V_R+V_F)}{(V_F-V_R)^2} \qquad &v\in[V_R,V_F].
    \end{cases}
\end{aligned}
\end{equation}
\bibliographystyle{plain}
\bibliography{reference}
\end{document}